\documentclass[12pt]{amsart}

\usepackage{amsfonts,amsbsy, amsthm, amsmath, amssymb, bbm, bm, color,  enumerate,tikz-cd,latexsym,amsopn,amstext, amsxtra,euscript,amscd,cite}

\usepackage{constants}   
\usepackage{hyperref}
\usepackage{soul}

\usepackage[all]{xy}
\usepackage[margin=2.6cm]{geometry}
\usepackage{float}
\usepackage{mathrsfs, mathtools}

\usepackage{helvet}

\numberwithin{equation}{section}

\allowdisplaybreaks

\newtheorem{theorem}{Theorem}
\numberwithin{theorem}{section}

\newtheorem{lemma}[theorem]{Lemma}
\newtheorem{proposition}[theorem]{Proposition}
\newtheorem{corollary}[theorem]{Corollary}

\theoremstyle{definition}

\newtheorem*{remark}{Remark}

\renewcommand{\phi}{\varphi}

\newcommand{\ZZ}{\mathbb{Z}}

\newcommand{\QQ}{\mathbb{Q}}

\renewcommand{\tilde}{\widetilde}

\newcommand{\R}{\mathbb{R}}
\newcommand{\Q}{\mathbb{Q}}

\newcommand{\kn}{\mathfrak{n}}

\newcommand{\kp}{\mathfrak{p}}

\newcommand{\re}{\textup{Re}}
\newcommand{\im}{\textup{Im}}

\newcommand{\kb}{\mathfrak{b}}

\newcommand{\GL}{\mathrm{GL}}

\newcommand{\N}{\mathrm{N}}

\newcommand{\ka}{\mathfrak{a}}

\renewcommand{\bar}{\overline}

\renewcommand{\epsilon}{\varepsilon}

\renewcommand{\pmod}[1]{\, (\mathrm{mod} {\, #1})}

\newcommand{\cO}{\mathcal{O}}

\newcommand{\kq}{\mathfrak{q}}

\renewcommand{\Re}{\mathop{\mathrm{Re}}}

\newconstantfamily{abcon}{symbol=c}

\renewcommand{\leq}{\leqslant}
\renewcommand{\geq}{\geqslant}
\renewcommand{\bar}{\overline}

\DeclareMathOperator{\Mod}{mod}
\renewcommand{\bmod}[1]{\,(\Mod{#1})}

\begin{document}
	\title{A Bombieri--Vinogradov theorem for higher rank groups}
	
	\author{Yujiao Jiang}
	
	\address{Yujiao Jiang\\
		School of Mathematics and Statistics
		\\
		Shandong University
		\\
		Weihai
		\\
		Shandong 264209
		\\
		China}
	\email{yujiaoj@sdu.edu.cn}
	
	\author{Guangshi L\"u}
	
	\address{Guangshi L\"u\\
		School of Mathematics
		\\
		Shandong University
		\\
		Jinan
		\\
		Shandong 250100
		\\
		China}
	\email{gslv@sdu.edu.cn}
	
	\author{Jesse Thorner}
		\address{Jesse Thorner\\
	Department of Mathematics
		\\
		University of Illinois
		\\
		Urbana
		\\
		IL 61801
		\\
	United States}
	\email{jesse.thorner@gmail.com}

	\author{Zihao Wang}
	
	\address{Zihao Wang\\
		School of Mathematics
		\\
		Shandong University
		\\
		Jinan
		\\
		Shandong 250100
		\\
		China}
	\email{wangzihao36@outlook.com }

	\date{\today}

	\begin{abstract}
We establish a result of Bombieri--Vinogradov type for the Dirichlet coefficients at prime ideals of the standard $L$-function associated to a self-dual cuspidal automorphic representation $\pi$ of $\operatorname{GL}_n$ over a number field $F$ which is not a quadratic twist of itself.  Our result does not rely on any unproven progress towards the generalized Ramanujan conjecture or the nonexistence of Landau--Siegel zeros. In particular, when $\pi$ is fixed and not equal to a quadratic twist of itself, we prove the first unconditional Siegel-type lower bound for the twisted $L$-values $|L(1,\pi\otimes\chi)|$ in the $\chi$-aspect, where $\chi$ is a primitive quadratic Hecke character over $F$.  Our result improves the levels of distribution in other works that relied on these unproven hypotheses. As applications, when $n=2,3,4$, we prove a $\GL_n$ analogue of the Titchmarsh divisor problem and a nontrivial bound for a certain $\GL_n\times\GL_2$ shifted convolution sum. 
	\end{abstract}
	
	\maketitle
	
\setcounter{tocdepth}{1}
\tableofcontents

\section{Introduction} The distribution of primes in arithmetic progressions attracts a lot of attention among mathematicians. Let $a, q$ be two integers such that $(a,q)=1$. We denote by $\pi(x)$ the number of primes $p\leq x$ and by $\pi(x; q, a)$ the number of primes $p\leq x$ satisfying $p\equiv a \bmod q$. Dirichlet's theorem indicates the following
\begin{equation*}
	\pi(x;q,a)\sim \frac{\pi(x)}{\phi(q)},
\end{equation*}
where $\phi$ is Euler's totient function. Later, after Siegel's result on the location of exceptional zero of Dirichlet $L$-functions, Walfisz proved that {for all $\alpha>0$, there exists an ineffective constant $c_{\alpha}>0$ such that if $q\leq (\log x)^{\alpha}$, then}
\begin{equation*}
	\pi(x;q,a)= \frac{\pi(x)}{\phi(q)}+O\big(x\exp(-{c_{\alpha}}(\log x)^{1/2})\big).
\end{equation*}
When the modulus $q$ gets larger, this problem becomes much more difficult. If the generalized Riemann hypothesis (GRH) holds, then
\begin{equation*}
	\pi(x;q,a)= \frac{\pi(x)}{\phi(q)}+O\big(x^\frac{1}{2}\log qx\big)
\end{equation*}
holds for $q\geq x^{1/2-\varepsilon}$. However, such a hypothesis is very far from being proved.

The celebrated Bombieri--Vinogradov theorem in some sense shows that GRH holds on average. To be precise, let $A$ be any positive real number, there exists $B=B(A)>0$ such that for $Q\leq x^\frac{1}{2}(\log x)^{-B}$,
\begin{equation}\label{BV-classical}
	\sum_{q \leqslant Q} \max _{(a, q)=1} \max _{y \leqslant x}\Big|\pi(y, q, a)-\frac{\pi(y)}{\phi(q)} \Big| \ll_A \frac{x}{(\log x)^{A}}.
\end{equation}
This can be viewed as a fine substitute for the GRH in many applications. The theorem was originally proved using zero density estimates. After the work of Bombieri and Vinogradov, different proofs of this theorem are given by Gallagher \cite{Gallagher-1968} and Vaughan \cite{Vaughan-1980}.

There are a lot of higher-rank analogues of the classical Bombieri--Vinogradov theorem. Firstly, by means of Gallagher's method, Grupp \cite{Grupp-1980} obtained under a certain condition concerning Siegel's zeros of $\operatorname{GL}_2$ automorphic $L$-functions,
\begin{equation*}
	\sum_{q \leq x^{2 / 9} (\log x)^{-B}}\max_{(a,q)=1} \Big|\sum_{\substack{n\leq x \\ n\equiv a\bmod q}}\Lambda(m)\tau(m)m^{-\frac{11}{2}}\Big|\ll_A \frac{x}{(\log x)^{A}},
\end{equation*}
where $\Lambda(m)$ is the von Mangoldt function and $\tau(m)$ is the Ramanujan $\tau$-function. Later, Perelli \cite{Perelli-1984} used the generalized Vaughan identity for $\operatorname{GL}_2$ automorphic $L$-functions and unconditionally proved the mean-value theorem with a level of distribution $2/5$ instead of $2/9$.
Actually, Perelli's approach still works for any holomorphic cusp form on $\operatorname{SL}_2(\ZZ)$. Recently, Acharya \cite{Acharya-2018} and the first two authors \cite{JL-2020} improved independently the level to $1/2$ for any holomorphic or Maass cusp form on $\operatorname{SL}_2(\ZZ)$. For any automorphic form $\pi$ on higher-rank group $\operatorname{SL}_n(\ZZ)$ with $n\geq 3$, let $\lambda_\pi(m)$ to be the $m$-th Dirichlet coefficient of the associated $L$-function $L(s, \pi)$, one can also show a result of Bombieri--Vinogradov type
\begin{equation}\label{bv-qm}
	\sum_{q \leq Q}\max_{(a,q)=1} \Big|\sum_{\substack{m\leq x \\ m\equiv a\bmod q}}\Lambda(m)\lambda_\pi(m)\Big|\ll_{A,\pi} \frac{x}{(\log x)^{A}}.
\end{equation}
For instance, the first two authors \cite{JL-2020} established \eqref{bv-qm} with $Q=x^{\frac{2}{n+1}}(\log x)^{-B}$ under the generalized Ramanujan conjecture (GRC) and a certain condition concerning Siegel's zeros of {the twisted $L$-functions} $L(s,\pi\otimes\chi)$. Wong \cite{Wong-2020} showed \eqref{bv-qm} with $Q=x^{\min\{\frac{1}{n-2}, \frac{1}{2}\}-\varepsilon}$ under two similar conditions. The main tools of Jiang and L\"u are the generalized Vaughan identity and the distribution of $\lambda_\pi(m)$ in arithmetic progressions, while that of Wong is Gallagher's technique as in \cite{Gallagher-1968}.

In this paper, we will explore further the possibility of Vaughan's method and show an unconditional result for higher-rank groups in a number field. We refer the reader to Section 2 for the detailed introduction to the notation. {Let $\mathbb{A}_F$ be the ring of adeles over a number field $F$, and let $\mathfrak{F}_n$ be the set of cuspidal automorphic representations of $\mathrm{GL}_n(\mathbb{A}_F)$ with unitary central character, normalized such that the central character is trivial on the diagonally embedded copy of the positive reals.  Given $\pi\in\mathfrak{F}_n$,}
let $\mathfrak{q}_\pi$ be the conductor of $\pi$, $L(s, \pi)$ be the associated standard $L$-function, and $\tilde{\pi}\in\mathfrak{F}_n$ be the contragredient representation.  We write $\lambda_\pi(\mathfrak{n})$ to be the $\mathfrak{n}$-th Dirichlet coefficient of $L(s, \pi)$, where $\mathfrak{n}$ is an integral ideal in $F$.
Let $\mathrm{N}=\mathrm{N}_{F/\mathbb{Q}}$ to be the numerical norm.
As the classical Bombieri--Vinogradov theorem \eqref{BV-classical}, we will consider estimates of large sieve type associated to $\lambda_{\pi}(\mathfrak{p})$ with certain congruence condition. We denote by $\mathrm{Cl}^{+}(\mathfrak{m})$ the narrow class group modulo $\mathfrak{m}$. 
Let $h(\mathfrak{m})$ be the cardinality of  $\mathrm{Cl}^{+}(\mathfrak{m})$ and $\phi_F(\mathfrak{m}):=\mathrm{N}\mathfrak{m}\prod_{\mathfrak{p}|\mathfrak{m}}(1-\mathrm{N}\mathfrak{p}^{-1})$.

Our arguments require that if $\pi\in\mathfrak{F}_n$, then $\pi=\tilde{\pi}$.  This self-duality implies that $\lambda_{\pi}(\kn)\in\R$ for all $\kn$.  Also, we require that for all $\mathfrak{m}\subseteq \cO_F$ and all nontrivial primitive {quadratic} Hecke characters of $\mathrm{Cl}^+(\mathfrak{m})$, we have $\pi\neq\pi\otimes\chi$.  We let $\mathfrak{F}_n^{\flat}$ denote the set of all $\pi\in\mathfrak{F}_n$ satisfying these two hypotheses.  We prove the following result.

\begin{theorem}\label{thm 1.1}
	 Fix $\pi\in\mathfrak{F}_n^\flat$.  If $A>0$, $B=2^{\frac{n[F:\Q]}{4}}(6A+12n+34)+2n-4$, and $\eta=\max\{2,\frac{n}{2}\}$, then
	\begin{equation*}
		\sum_{\mathrm{N}\mathfrak{m}\leq Q}\frac{h(\mathfrak{m})}{\phi_F(\mathfrak{m})}\max_{(\mathfrak{a},\mathfrak{m})=\cO_F}\max_{y\leq x}\Big|\sum_{\substack{\mathrm{N}\mathfrak{p}\leq y \\ \mathfrak{p}\equiv \mathfrak{a} \text{ in } \mathrm{Cl}^{+}(\mathfrak{m})}}\lambda_\pi(\mathfrak{p})\Big|\ll_{A} \frac{x}{(\log x)^A},
	\end{equation*}
    where $x\geq 2$ and $Q=x^{\frac{1}{\eta}}(\log x)^{-B}$. 
	The implied constant is ineffective.
\end{theorem}
~
\begin{remark}
~
\begin{enumerate}[1.]
	\item The congruence condition ``$\kp\equiv\ka\text{ in }\mathrm{Cl}^+(\mathfrak{m})$'' is defined in Section \ref{subsec:GL1twists}.  This generalizes the usual notion of congruences on the integers to the integral ideals of $F$.  In particular, if $F=\Q$, then the bound in Theorem \ref{thm 1.1} becomes
	\[
	\sum_{q\leq Q}\max_{\gcd(a,q)=1}\max_{y\leq x}\Big|\sum_{\substack{p\leq y \\ p\equiv a\pmod{q}}}\lambda_{\pi}(p)\Big|\ll_A\frac{x}{(\log x)^A},
	\]
	where $p$ (resp. $a$ and $q$) are rational primes (resp. rational integers).
	\item If we adjust Theorem \ref{thm 1.1} so that we sum over $\mathfrak{m}$ satisfying $(\mathfrak{m},\kq_{\pi})=\cO_F$, then we may obtain a similar result with the same level of distribution.  Our result would then hold for all self-dual $\pi$ since the condition $\pi\neq\pi\otimes\chi$ automatically holds.
	\item The weight $\frac{h(\mathfrak{m})}{\phi_F(\mathfrak{m})}$ is introduced by Huxley in \cite{Huxley-1971} to cancel the contribution coming from the unit group (see \eqref{exactseq}).
	\item  Note that the analogue of Elliott-Halberstam conjecture will predict that $\eta=1+\varepsilon$, and the GRH for automorphic $L$-functions will trivially give that $\eta=2$. Since the {arithmetic} conductor of $L(s,\pi\otimes\chi)$ might be quite large, it is hard to achieve any of them by our argument. 


\end{enumerate}
\end{remark}

To handle the contribution when $\N\mathfrak{m}$ is smaller than a power of $\log x$, we need to prove an analogue of the Siegel--Walfisz theorem for the Dirichlet coefficients of $-\frac{L'}{L}(s,\pi)$.  One of the novelties in our work which allows us to prove such a result without recourse to unproven hypotheses is a new Siegel-type lower bound for $|L(1,\pi\otimes\chi)|$ when $\chi$ is a primitive quadratic Hecke character and $\pi\in\mathfrak{F}_n$ (not necessarily self-dual) is not a quadratic twist of itself.

\begin{theorem}
\label{thm:siegel}
Fix $\pi\in\mathfrak{F}_n$, and suppose that $\pi\neq\pi\otimes\nu$ for all primitive quadratic Hecke characters $\nu$.  Let $\chi\pmod{\kq}$ be a primitive quadratic Hecke character.  For all $\epsilon>0$, there exists an ineffective constant $c_{\pi}'(\epsilon)>0$ such that $|L(1,\pi\otimes\chi)|\geq c_{\pi}'(\epsilon)\N\kq^{-\epsilon}$.
\end{theorem}
\begin{remark}
Over $\Q$, Theorem \ref{thm:siegel} was claimed by Molteni in his PhD thesis, but there is a serious deficiency in his argument.  Since this deficiency has occurred in several different papers (even before Molteni's), we detail the deficiency and address it in Section \ref{sec:ZFR}.
\end{remark}
To handle the contribution when $\N\mathfrak{m}$ is larger than a power of $\log x$, we require a modification of Vaughan's approach to the Bombieri--Vinogradov theorem. The problem of estimating 
\[
\sideset{}{}{\sum}_{\mathrm{N}\mathfrak{m}\leq Q}\frac{h(\mathfrak{m})}{\phi_F(\mathfrak{m})}\max_{(\mathfrak{a},\mathfrak{m})=\cO_F}\max_{y\leq X}\Big|\sum_{\substack{\mathrm{N}\mathfrak{p}\leq y \\ \mathfrak{p}\equiv \mathfrak{a} \text{ in } \mathrm{Cl}^{+}(\mathfrak{m})}}\lambda_\pi(\mathfrak{p})\Big|
\]
is equivalent to that of handling
\[
\sideset{}{}{\sum}_{\mathrm{N}\mathfrak{m}\leq Q} \frac{h(\mathfrak{m})}{\phi_F(\mathfrak{m})}\max_{(\mathfrak{a},\mathfrak{m})=\cO_F} \max_{y\leq x}\Big|\sum_{\substack{\mathrm{N}\mathfrak{n}\leq y \\ \mathfrak{n}\equiv\mathfrak{a} \text{ in } \mathrm{Cl}^{+}(\mathfrak{m})}}\Lambda_F(\mathfrak{n})a_\pi(\mathfrak{n})\Big|
\]
with a harmless error, where $\Lambda_F(\mathfrak{n})a_\pi(\mathfrak{n})$ is the coefficient of {$-\frac{L'}{L}(s,\pi)$.} 
We derive a generalized Vaughan identity, which gives an expression for $\Lambda_F(\mathfrak{n})a_\pi(\mathfrak{n})$,  and then apply it to decompose the above object into the Type I sum
\begin{equation}\label{linear-sum}
\sum_{\mathrm{N}\mathfrak{m}\leq Q}\frac{h(\mathfrak{m})}{\phi_F(\mathfrak{m})}\max_{(\mathfrak{a},\mathfrak{m})=\cO_F}\max_{y\leq x}\Big|\sum_{\substack{\mathrm{N}\mathfrak{n}\leq y \\ \mathfrak{n}\equiv \mathfrak{a} \text{ in } \mathrm{Cl}^{+}(\mathfrak{m})}}\lambda_\pi(\mathfrak{n})\Big|
\end{equation}
and the Type II sum
\begin{equation}\label{bilinear-sum}
\sum_{\mathrm{N}\mathfrak{m}\leq Q}\frac{h(\mathfrak{m})}{\phi_F(\mathfrak{m})} \max_{(\mathfrak{a},\mathfrak{m})=\cO_F} \max_{y\leq x} \Big| {\underset{\mathrm{N}\mathfrak{ln}\leq y\atop \mathfrak{ln}\equiv\mathfrak{a} \text{ in } \mathrm{Cl}^{+}(\mathfrak{m})}{\sum_{\mathrm{N}\mathfrak{l}\leq L}\sum_{\mathrm{N}\mathfrak{n}\leq N}}}a(\mathfrak{l})b(\mathfrak{n}) \Big|.
\end{equation}
Note that \eqref{bilinear-sum} is actually bilinear form with $L, N$ 
{in suitable ranges}, and $a(\mathfrak{n}), b(\mathfrak{n})$ are arithmetic functions related to $\pi$.

In our setting, a strong bound for \eqref{linear-sum} is already new.  Since it is useful in contexts beyond that of Theorem \ref{thm 1.1}, we state it as its own theorem. 

\begin{theorem}\label{thm-bv-integers}
	Fix $\pi\in\mathfrak{F}_n$.  If $A>0$, $B=2^{n[F:\Q]/4}(2A+16)+2n-5$, and $\eta=\max\{\frac{n}{2},2\}$, then
	\begin{equation*}
		\sideset{}{}{\sum}_{\mathrm{N}\mathfrak{m}\leq Q}\frac{h(\mathfrak{m})}{\phi_F(\mathfrak{m})}\max_{(\mathfrak{a},\mathfrak{m})=\cO_F}\max_{y\leq x}\Big|\sum_{\substack{\mathrm{N}\mathfrak{n}\leq y \\ \mathfrak{n}\equiv \mathfrak{a} \text{ in } \mathrm{Cl}^{+}(\mathfrak{m})}}\lambda_\pi(\mathfrak{n})\Big|\ll_{\pi} \frac{x}{(\log x)^A},
	\end{equation*}
	where $x\geq 2$ and $Q=x^{\frac{1}{\eta}}(\log x)^{-B}$.
\end{theorem}

\begin{remark}
	The value of $\eta$ in Theorem \ref{thm 1.1} is totally determined by the value of $\eta$ in Theorem \ref{thm-bv-integers}. We improve previous results because we notice that more cancellation can be obtained by summing over the modulus, which is absent in the work of \cite{JL-2020,Vaughan-1980}. In fact, Vaughan \cite{Vaughan-1980} directly used the P\'olya-Vinogradov inequality, and the first two authors \cite{JL-2020} used the Vorono\"i formula on ${\rm GL}(n)$ to treat the sum of $\lambda_\pi(\mathfrak{n})$ over a single arithmetic progression.
\end{remark}


We estimate \eqref{bilinear-sum} through bilinear sum methods, proving a general result similar to  \cite[Theorem 17.4]{IK-2004}. However, the important condition of \cite[Theorem 17.4]{IK-2004} is that one of these two arithmetic functions $a(\mathfrak{n})$ and $b(\mathfrak{n})$ satisfies a Siegel--Walfisz hypothesis. In our situation, we need to verify that both $\lambda_\pi(\mathfrak{n})$ and $\Lambda_F(\mathfrak{n})a_\pi(\mathfrak{n})$ satisfy a Siegel--Walfisz hypothesis.  This hypothesis is straightforward to verify for $\lambda_{\pi}(\kn)$, and as mentioned above, we verify this hypothesis for $\Lambda_F(\mathfrak{n})a_\pi(\mathfrak{n})$ as a corollary of Theorem \ref{thm:siegel}. Note that there is a cumbersome cut-off condition in \eqref{bilinear-sum}. To handle this, we adopt a trick of Vaughan in \cite{Vaughan-1980}.

Our upper bound for \eqref{bilinear-sum} involves the second moments of some arithmetic functions of length $x$, whose magnitudes need to be of order $O(x(\log x)^c )$ for some computable constant $c$.  If GRC holds for $\pi$, then the desired upper bound follows from elementary estimate of divisor functions. In \cite{JL-2020}, the first two authors bounded these arithmetic functions under Hypothesis H of Rudnick and Sarnak \cite{RS-1996}. This mild conjecture is implied by GRC and is only known to hold for few cases.  In order to circumvent this additional assumption, we instead bound them by some Dirichlet convolutions of $\lambda_{\pi \times \tilde{\pi}}(\mathfrak{n})$ through the dual Pieri rule and a combinatorial lemma of Soundrarajan. The desired upper bound then follows from the Rankin--Selberg theory. 


As in the classical case, Theorem \ref{thm 1.1} is a fruitful result. As an application, we will give one analogue of Titchmarsh's divisor problem on $\operatorname{GL}_n$ over $\QQ$ with $2\leq n\leq 4$. {Let $d(m)$ be the usual divisor function.}  It is known that $d(m)$ are Fourier coefficients of $\frac{\partial}{\partial s}E(z,s)$ at $s=\frac{1}{2}$, where $E(z,s)$ is the Eisenstein series for $\operatorname{SL}_2(\ZZ)$. Thus, the following result may be also viewed as the shifted convolution sum at primes for $\operatorname{GL}_n\times \operatorname{GL}_2$.

\begin{corollary}\label{cor-titchmarsh}
Let $2\leq n\leq 4$, and fix $\pi\in\mathfrak{F}_n^{\flat}$.  If $x\geq 2$ and $\pi$ is defined over $\mathbb{Q}$, then
	\begin{equation*}\label{shifted}
	\sum_{p\leq x}\lambda_\pi(p)d(p-1)\ll_\pi \frac{x(\log\log x)^{\frac{3}{2}}}{\sqrt{\log x}},
	\end{equation*}	
	where the implied constant depends on $\pi$.
\end{corollary}

\begin{remark}
	The case with $n=2$ is known by the work of Acharya \cite{Acharya-2018}. Under GRC, the first two authors \cite{JL-2020} handled the cases with $n=2,3$ and obtained a stronger upper bound than that in Corollary \ref{cor-titchmarsh}.
\end{remark}

{If we use Theorem \ref{thm-bv-integers} instead of Theorem \ref{thm 1.1}, then the argument leading to Corollary \ref{cor-titchmarsh} produces a corresponding shifted convolution bound over the integers.}
\begin{corollary}\label{cor-shifted}
Let $2\leq n\leq 4$, and fix $\pi\in\mathfrak{F}_n^{\flat}$.  If $x\geq 2$ and $\pi$ is defined over $\mathbb{Q}$, then
\begin{equation*}\label{shifted-2}
\sum_{m\leq x}\lambda_\pi(m)d(m-1)\ll_\pi x(\log\log x)^{\frac{3}{2}},
\end{equation*}	
where the implied constant depends on $\pi$.
\end{corollary}

 Finally, we show that Corollary \ref{cor-titchmarsh} and Corollary \ref{cor-shifted} do in fact provide non-trivial estimates. We first recall an elementary result (see \cite[page 937]{JL-2020}, for example)
\begin{equation}\label{phi-sum}
	\sum_{q \leq x} \frac{1}{\varphi(q)}=\frac{\zeta(2) \zeta(3)}{\zeta(6)} \log x+O(1).
\end{equation}
Assuming the Riemann hypothesis for all of the twisted $L$-functions $L(s,\pi\times(\tilde{\pi}\otimes\chi))$ as well as GRC, it follows from \cite[Theorem 5.15]{IK-2004} that
\[
\sum_{\substack{p\leq x\\  p\equiv 1\bmod q}}|\lambda_\pi(p)|^2\log p=\frac{x}{\varphi(q)}+O(x^{\frac{1}{2}}(\log qx)^2).
\]
We average over the modulus $q\leq x^{1/3}$ and obtain from \eqref{phi-sum} and partial summation that
\[
\sum_{q\leq x^{1/3}}\sum_{\substack{p\leq x\\  p\equiv 1\bmod q}}|\lambda_\pi(p)|^2\asymp x.
\]
One can easily verify that $|\lambda_\pi(p)|\gg  |\lambda_\pi(p)|^2$ under GRC, so the above estimate gives
\[
\sum_{p\leq x}|\lambda_\pi(p)|d(p-1)\gg \sum_{p\leq x}|\lambda_\pi(p)|^2 d(p-1)\gg \sum_{q\leq x^{1/3}}\sum_{\substack{p\leq x\\  p\equiv 1\bmod q}}|\lambda_\pi(p)|^2\gg x.
\]
This means that there exists some cancellation in the sequence $\{\lambda_\pi(p)d(p-1)\}$, where $p$ runs over all primes.  

For Corollary \ref{cor-shifted}, we argue as follows.  Firstly, we recall an interesting result in \cite{TT-1998}: ``\emph{Let a multiplicative function $f(m)\geq 0$ satisfy the following conditions: (i) $f(m)\geq 0$; $f(p^r)\leq A^r $ for some $A>0$; (ii) $f(m) \ll m^{\varepsilon}$ for any $\varepsilon>0$; (iii) $\sum_{p \leq x} f(p) \log p \geq \alpha x$ with some $\alpha>0$, then one has the asymptotic formula
	\begin{equation*}\label{MTinTT-1998}
		\sum_{m \leq x} f(m) d(m-1)=C_f \sum_{m \leq x} f(m) \log
		x(1+o(1))
	\end{equation*}
	for some constant $C_f$ depending on $f$.}" Next, suppose that GRC holds, we then obtain from \cite[p. 595]{JLW-2020} that
\begin{equation}\label{sum-primelog}
	\sum_{p \leq x}\left|\lambda_{\pi}(p)\right| \log p \geq\Big(\frac{1}{m}+o(1)\Big) x,
\end{equation}
and
\begin{equation}\label{lowerboundinteger}
\sum_{m \leq x}\left|\lambda_{\pi}(m)\right| \gg \frac{x}{(\log x)^{1-\frac{1}{n}}}.
\end{equation}
One can easily check that with the help of \eqref{sum-primelog}, the above conditions (i)-(iii) hold for $f(m)=|\lambda_\pi(m)|$ under GRC. Hence, we could get
\begin{equation}\label{MTinTT-1998-2}
\sum_{m \leq x} |\lambda_\pi(m)| d(m-1)=C \sum_{m \leq x}|\lambda_\pi(m)| \log x(1+o(1))
\end{equation}
for some constant $C$ depending on $\pi$. Combining \eqref{lowerboundinteger} with \eqref{MTinTT-1998-2}, we have
\[
\sum_{m \leq x} |\lambda_\pi(m)| d(m-1)\gg x(\log x)^{\frac{1}{n}},
\]
which implies that Corollary \ref{shifted-2} do give a non-trivial upper bound.

\section{Properties of $L$-functions}

\subsection{Conventions}

For a parameter $\delta$, we use the notation $f\ll_{\delta} g$ or $f=O_{\delta}(g)$ to denote that there exists a constant $c\geq 0$, depending at most on $\delta$ and $\pi$, such that such that $|f|\leq cg$ in a range that will be clear in context.

\subsection{Automorphic $L$-functions}

Let $F$ be a number field with discriminant $D_F$ and $d=[F: \mathbb{Q}]$. Let $\cO_F$ be
the ring of integers in $F$. For each place $v$ of $F$, denote by $F_v$ the completion of
$F$ with respect to $v$ and by $O_v$ the local ring of integers. The prime ideals $\kp\subseteq \cO_F$ and the nonarchimedean places $v$ are in bijective correspondence.  So we may write $\mathfrak{p}$ interchangeably with nonarchimedean places.  
Each $\pi \in \mathfrak{F}_n$ is a restricted tensor product $\bigotimes_{v} \pi_v$ of smooth admissible representations of $\GL_n(F_v)$ such that $\pi_v$ is unramified for almost all finite places $v$. Let $\mathfrak{q}_\pi$ be the conductor of $\pi$, which {has the property that $\pi_{\kp}$ is ramified if and only if $\kp|\kq_{\pi}$.}

For each {prime ideal} 
$\mathfrak{p}$, the standard local $L$-function is defined in terms of Satake parameters $A_\pi(\mathfrak{p})=\{\alpha_{1,\pi}(\mathfrak{p}),\dots,\alpha_{n,\pi}(\mathfrak{p})\}$ by
\begin{equation}\label{pi-euler}
	L(s,\pi_{\mathfrak{p}}):=\prod_{j=1}^n(1-\alpha_{j,\pi}(\mathfrak{p})\mathrm{N}\mathfrak{p}^{-s})^{-1},
\end{equation}
where $\mathrm{N}=\mathrm{N}_{F/\mathbb{Q}}$ is the absolute norm over $\Q$. For $\mathfrak{p}\nmid\mathfrak{q}_{\pi}$, we have $\alpha_{j,\pi}(\mathfrak{p})\neq0$ for all $i\in\{1, \ldots ,n\}$. However, it might be the case that $\alpha_{i,\pi}(\mathfrak{p})=0$ for some $j$ when $\mathfrak{p}|\mathfrak{q}_{\pi}$. The standard (finite) $L$-function is defined to be
\begin{equation}\label{pi-eu-2}
	L(s,\pi)=\prod_{\mathfrak{p}}L(s,\pi_{\mathfrak{p}}):=\prod_{\mathfrak{p}}\sum_{k=0}^{\infty}\frac{\lambda_{\pi}(\mathfrak{p}^k)}{\mathrm{N}\mathfrak{p}^{ks}}=\sum_{\mathfrak{n}\subset \cO_F}\frac{\lambda_{\pi}(\mathfrak{n})}{{\mathrm{N}\mathfrak{n}}^s}
\end{equation}
for $\Re s>1$, where the product is over all prime ideals $\mathfrak{p}$ and the sum is over all integral ideals $\mathfrak{n}$. We can see that $\lambda_{\pi}(\mathfrak{n})$ is multiplicative, that is $\lambda_{\pi}(\mathfrak{n}_1\mathfrak{n}_2)=\lambda_{\pi}(\mathfrak{n}_1)\lambda_{\pi}(\mathfrak{n}_2)$ for coprime integral ideals $\mathfrak{n}_1$ and $\mathfrak{n}_2$. We can also write $\lambda_{\pi}(\mathfrak{n})$ in terms of Satake parameters
\begin{equation*}
	\lambda_{\pi}(\mathfrak{p}^k)=\sum_{m_1+\cdots +m_n=k}\prod_{j=1}^n\alpha_{j , \pi}^{m_j}(\mathfrak{p})
\end{equation*}
and extend it to all integral ideals $\mathfrak{n}$ by multiplicativity. Taking logarithmic derivatives in \eqref{pi-eu-2}, we can see that for $\Re s>1$,
\begin{equation*}\label{logderL}
	-\frac{L^{\prime}}{L}(s,\pi)=\sum_{\mathfrak{p}}\sum_{k=1}^\infty \frac{a_{\pi}(\mathfrak{p}^k)\log \mathrm{N}\mathfrak{p}}{{\mathrm{N}\mathfrak{p}}^{ks}}=\sum_{\mathfrak{n}\subset \cO_F}\frac{\Lambda_F(\mathfrak{n})a_\pi(\mathfrak{n})}{{\mathrm{N}\mathfrak{n}}^s},
\end{equation*}
where
\begin{equation*}
\Lambda_F(\mathfrak{n}):=\left\{\begin{aligned}
&\log \mathrm{N}\mathfrak{p}  \;\;\text{ if }\mathfrak{n}=\mathfrak{p}^k \text{ for some }k\in\mathbb{N}, \\
&0 \,\;\;\;\quad\quad  \text{ otherwise,}
\end{aligned}
\right.
\end{equation*}
and $a_{\pi}(\mathfrak{p}^k)=\sum_{j=1}^n\alpha_{j , \pi}(\mathfrak{p})^k$ . We set $a_\pi(\mathfrak{n})=0$ if $\mathfrak{n}$ is not a prime ideal power. Note that $a_{\pi}(\mathfrak{p})=\lambda_{\pi}(\mathfrak{p})$. We write $\mu_\pi(\mathfrak{n})$ to be the coefficients of Dirichlet series $L(s,\pi)^{-1}$, namely
\begin{equation}\label{inv-lfun}
	L(s,\pi)^{-1}=\sum_{\mathfrak{n}\subset \cO_F}\frac{\mu_\pi(\mathfrak{n})}{{\mathrm{N}\mathfrak{n}}^s}
\end{equation}
for $\Re s>1$. Then it can be is given by
\begin{equation}\label{mu-def}
	\mu_\pi(\mathfrak{n}) = \begin{dcases}
		0 &\mbox{if $\mathfrak{p}^{n+1}|\mathfrak{n}$ for some prime $\mathfrak{p}$,}\\
		\prod\limits_{\substack{\mathfrak{p}^\ell \| \mathfrak{n} \\ \ell\leq n}} (-1)^\ell\sum\limits_{1\leq j_1<\cdots<j_\ell \leq n}
		\alpha_{j_1,\pi}(\mathfrak{p})\cdots \alpha_{j_\ell,\pi}(\mathfrak{p}) &\mbox{otherwise.}
	\end{dcases}
\end{equation}
Clearly, $\mu_\pi(\mathfrak{n})$ is multiplicative.

Now suppose $v$ is an archimedean place of $F$ {(denoted $v|\infty$)}, so $F_v=\mathbb{R}$ or $\mathbb{C}$. Denote $\Gamma(s)$ to be the usual gamma function and define
\[
\Gamma_{v}(s):=\left\{\begin{array}{ll}
	\pi^{-s / 2} \Gamma(s / 2) & \text { if } F_{v}=\mathbb{R}, \\
	2(2 \pi)^{-s} \Gamma(s) & \text { if } F_{v}=\mathbb{C}.
\end{array}\right.
\]
{For each archimedean place $v$, there exists $n$ Langlands parameters $\mu_{1,\pi}(v),\ldots,\mu_{n,\pi}(v)$ from which we define}
\begin{equation*}\label{ArchiL}
	L(s,\pi_v)=\prod_{j=1}^n\Gamma_v(s+\mu_{j,\pi}(v)).
\end{equation*}
If we denote
\[
L_{\infty}(s,\pi)=\prod_{v|\infty}L(s,\pi_v),
\]
then {for nontrivial $\pi$}, the complete $L$-function defined by
\begin{equation*}\label{completeL}
	\Lambda(s,\pi)=(D_F^n\mathrm{N}\mathfrak{q}_{\pi})^{\frac{s}{2}}L(s,\pi)L_{\infty}(s,\pi)
\end{equation*}
{extends} to an entire function of order 1 and is bounded in the vertical strip.
{Luo, Rudnick, and Sarnak \cite{LRS-1999} and M{\"u}ller and Speh \cite{MS-2004} proved that there exists $\theta_n\in[0,\frac{1}{2}-\frac{1}{n^2+1}]$ such that we have the uniform bounds}
\begin{equation}\label{boundcoeff}
	|\alpha_{j , \pi}(\mathfrak{p})| \leq {\mathrm{N}\mathfrak{p}}^{\theta_{n}} \quad \text { and } \quad -\operatorname{Re}(\mu_{j, \pi}(v)) \leq\theta_{n}.
\end{equation}
The generalized Ramanujan conjecture (GRC) predicts that $\theta_{n}=0$.

We denote by $\tilde{\pi}$ the contragradient representation of $\pi$ which is also an irreducible cuspidal automorphic representation with unitary central character. One can show that $\mathfrak{q}_{\tilde{\pi}}=\mathfrak{q}_{\pi}$. We can also define the $L$-function associated to $\tilde{\pi}$ in the same fashion. 
{We have the equalities of sets} $\{\bar{\alpha_{j , \pi}(\mathfrak{p})}\}_{j=1}^n=\{\alpha_{j, \tilde\pi}(\mathfrak{p})\}_{j=1}^n$ and $\{\bar{\mu_{j,\pi}(v)}\}_{j=1}^n=\{\mu_{j,\tilde\pi}(v)\}_{j=1}^n$.  There exists a complex number $\varepsilon(\pi)$ of modulus 1 such that
\begin{equation*}\label{FE}
	\Lambda(s,\pi)=\varepsilon(\pi)\Lambda(1-s,\tilde\pi).
\end{equation*}

Now we define the analytic conductor of $\pi$. We write $q(\pi):=D_F^n\mathrm{N}\mathfrak{q}_{\pi}$ for the arithmetic conductor, and the analytic conductor is defined by
\begin{equation*}\label{conductor}
	C(\pi,t):=q(\pi)\prod_{v|\infty}\prod_{j=1}^n(3+|it+\mu_{j,\pi}(v)|^{d_v}):=q(\pi)q_{\infty}(\pi,t),
\end{equation*}
where $d_v=1$ if $F_v=\R$ and $d_v=2$ if $F_v=\mathbb{C}$ 
This is an important parameter to describe $L(s,\pi)$. For example, the convexity bound, the zero-free region and second moment estimates can be described in terms of analytic conductor (see sections below).

\vskip 2mm

\subsection{Rankin--Selberg $L$-functions}
Let $\pi=\bigotimes_{v} \pi_{v} \in \mathfrak{F}_n$ and $\pi^{\prime}=\bigotimes_{v} \pi_{v}^\prime \in \mathfrak{F}_{n'}$. The Rankin--Selberg $L$-function at a finite place $\mathfrak{p}$ is defined to be
\begin{equation}\label{rs-ep}
	L(s, \pi_{\mathfrak{p}} \times \pi_{\mathfrak{p}}^{\prime})=\prod_{j=1}^{n} \prod_{j'=1}^{n^{\prime}}(1-\alpha_{j,j',\pi \times \pi^{\prime}}(\mathfrak{p}) \mathrm{N}\mathfrak{p}^{-s})^{-1},
\end{equation}
where $\alpha_{j,j',\pi \times \pi^{\prime}}(\mathfrak{p})$ are suitable complex numbers. For a finite place $\mathfrak{p}$ such that either $\pi_{\mathfrak{p}}$ or $\pi_{\mathfrak{p}}^{\prime}$ is unramified, we have the equality of sets $\{\alpha_{j,j',\pi \times \pi^{\prime}}(\mathfrak{p})\}=\{\alpha_{j,\pi}(\mathfrak{p}) \alpha_{j',\pi^{\prime}}(\mathfrak{p})\}$. We also define the (finite) Rankin--Selberg $L$-function to be
\begin{equation}\label{pipi-eu-2}
	L(s,\pi\times\pi^{\prime})=\prod_{\mathfrak{p}}L(s,\pi_{\kp}\times\pi_{\kp}'):=\prod_{\mathfrak{p}}\sum_{k=0}^{\infty}\frac{\lambda_{\pi\times\pi^{\prime}}(\mathfrak{p}^k)}{\mathrm{N}\mathfrak{p}^{ks}}=\sum_{\mathfrak{n}\subset \cO_F}\frac{\lambda_{\pi\times\pi^{\prime}}(\mathfrak{n})}{{\mathrm{N}\mathfrak{n}}^s}
\end{equation}
for $\Re s>1$, where the product is over all prime ideals $\mathfrak{p}$ and the sum is over nonzero integral ideals $\mathfrak{n}$.

For each archimedean place $v$, the local $L$-factor at $v$ is
\[
L(s, \pi_{v} \times \pi_{v}^{\prime})=\prod_{j=1}^{n} \prod_{j'=1}^{n^{\prime}} \Gamma_{v}(s+\mu_{j,j',\pi \times \pi^{\prime}}(v))
\]
for suitable complex numbers $\mu_{j,j',\pi \times \pi^{\prime}}(v)$. Define
\begin{equation*}\label{ArchiRSL}
	L_{\infty}(s,\pi\times\pi^{\prime})=\prod_{v|\infty}L(s,\pi_v\times\pi_{v}^{\prime}).
\end{equation*}
When $v$ is a place such that both $\pi_{v}$ and $\pi_v^{\prime}$ are unramified, then {we have the equality of sets} $\{\mu_{j,j',\pi \times \pi^{\prime}}(v)\}=\{\mu_{j,\pi}(v)+\mu_{j',\pi^{\prime}}(v)\}$.
By our normalization of the central characters, we have $L(s,\pi\times\pi')$ has a pole at $s=1$ with order $r_{\pi\times\pi'}=1$  if and only if $\pi'\simeq\tilde{\pi}$, and $r_{\pi\times\pi'}=0$ otherwise.  We can also associate an arithmetic conductor $q(\pi\times\pi^\prime)$ to $\pi\times\pi^\prime$, so the complete Rankin--Selberg $L$-function is defined by
\begin{equation*}\label{completeRSL}
	\Lambda(s,\pi\times\pi^{\prime}):=(s(s-1))^{r_{\pi\times\pi'}}q(\pi\times\pi^{\prime})^{\frac{s}{2}}L(s,\pi\times\pi^{\prime})L_{\infty}(s,\pi\times\pi^{\prime}).
\end{equation*}
{It is entire of order 1 and satisfies the following functional equation}
\begin{equation*}\label{RSFE}
	\Lambda(s,\pi\times\pi^{\prime})=\varepsilon(\pi\times\pi^{\prime})\Lambda(1-s,\tilde\pi\times\tilde{\pi}^{\prime}),
\end{equation*}
where $\varepsilon(\pi\times\pi^{\prime})$ is a complex number of modulus 1.  It follows from the explicit description of the numbers $\alpha_{j,j',\pi\times\pi'}(\kp)$ and $\mu_{j,j',\pi\times\pi'}(v)$ in \cite{ST-2019} and \cite[Appendix]{ST-2019} yields the bounds
\begin{equation}\label{boundcoeff2}
	|\alpha_{j,j',\pi\times\pi'}(\mathfrak{p})| \leq {\mathrm{N}\mathfrak{p}}^{\theta_{n}+\theta_{n'}} \quad \text { and } \quad -\operatorname{Re}(\mu_{j,j', \pi\times\pi'}(v)) \leq\theta_{n}+\theta_{n'}.
\end{equation}
We also define the analytic conductor $C(\pi\times\pi^\prime,t)$ by
\begin{equation*}\label{condutorRS}
	C(\pi\times\pi^\prime,t)=q(\pi\times\pi^\prime)\prod_{v|\infty}\prod_{j=1}^n\prod_{j=1}^{n^\prime}(3+|it+\mu_{i,j,\pi\times\pi^\prime}(v)|^{d_v}):=q(\pi\times\pi^\prime)q_{\infty}(\pi\times\pi^\prime,t)
\end{equation*}
for $d_v$ as above. An important inequality about conductors (see \cite{Brumley-2006}) is
\begin{equation}\label{conductorpair}
	C(\pi\times\pi^\prime,t)\ll C(\pi,0)^{n^{\prime}}C(\pi^\prime,0)^n(3+|t|)^{nn^{\prime}d}.
\end{equation}

We are especially interested in the case where $\pi'=\tilde{\pi}$.  In this case the Rankin--Selberg $L$-function $L(s,\pi\times\tilde{\pi})$ has non-negative Dirichlet coefficients $\lambda_{\pi\times\tilde{\pi}}(\mathfrak{n})$ (see Lemma \ref{ineq-2rs} for instance).  Moreover, $L(s,\pi\times\tilde{\pi})$ extends to the complex plane with a simple pole at $s = 1$.  Hence, it follows from a standard Tauberian argument that
\begin{equation}\label{asymRS}
	\sum_{\mathrm{N}\mathfrak{n}\leq x}\lambda_{\pi\times\tilde{\pi}}(\mathfrak{n})\sim x \mathop{\mathrm{Res}}_{s=1}L(s,\pi\times\tilde{\pi})\ll x.
\end{equation}

\subsection{$\mathrm{GL}_1$-twists}
\label{subsec:GL1twists}

Let $F$ be a number field. By a modulus $m$ of $F$, we mean a function
\begin{equation*}
	m:\{\text{all places of }F\}\rightarrow \mathbb{Z}
\end{equation*}
such that
\begin{enumerate}
	\item for all nonarchimedean places $v$, we have $m(v)\geq 0$, with $m(v)=0$ for all but finitely many $v$.
	\item if $v$ is a real archimedean place, then $m(v)=0$ or 1.
	\item if $v$ is a complex archimedean palce, then $m(v)=0$.
\end{enumerate}
For a modulus $m$, we write
\[
U_m(v):=\begin{cases}
	(F_{v}^\times)^{m(v)+1}&\mbox{if $v$ is archimedean,}\\
	1+v^{m(v)} &\mbox{if $v$ is nonarchimedean and $m(v)\neq 0$,} \\
	O_v^\times & \mbox{if $v$ is nonarchimedean and $m(v)=0$.}
\end{cases}
\]
Thus, in each case, $U_m(v)$ is a neighbourhood of 1 in $F_{v}^\times$. Note that $m(v)=0$ for all but finitely many nonarchimedean $v$, so
\[
U_m:=\prod_{v}U_m(v)
\]
is an open subset of the idele group $\mathbb{A}_F^\times$, where the product is over all places $v$ of $F$. For any modulus $m$, we can define the ray class group modulo $m$ to be $\mathrm{Cl}(m):=\mathbb{A}_F^\times/F^\times U_m$. By a narrow class group modulo an integral ideal $\mathfrak{a}$, we mean that it is defined by the modulus
\begin{equation*}
m_{\mathfrak{a}}(v):=\begin{cases}
\mathrm{ord}_{v}(\mathfrak{a}) &\mbox{if $v$ is nonarchimedean}, \\
1 &\mbox{if $v$ is real archimedean,} \\
0  &\mbox{if $v$ is complex archimedean.}
\end{cases}
\end{equation*}
where $\mathrm{ord}_{v}$ is the additive valuation with respect to $v$. We define $\mathrm{Cl}^{+}(\mathfrak{a}):=\mathbb{A}_F^\times/F^\times U_{m_{\mathfrak{a}}}$, which is a finite group, and whose cardinality is denoted by $h(\mathfrak{a})$. Later we may also use $\mathfrak{a}$ to denote this modulus for the simplicity of  notation. If $(\mathfrak{b},\mathfrak{a})=\cO_F$, one can use the map $\mathfrak{b}\mapsto \prod_{\mathfrak{p}}\varpi_{\mathfrak{p}}^{\mathrm{ord}_\mathfrak{p}(\mathfrak{b})}\times \prod_{v|\infty}1\mod F^\times U_{m_{\mathfrak{a}}}$ to projects $\mathfrak{b}$ to $\mathrm{Cl}^{+}(\mathfrak{a})$, where $\varpi_{\mathfrak{p}}$ is any fixed choice of uniformizer in $F_\mathfrak{p}$. So by ``$\mathfrak{b}\equiv\mathfrak{c}$ in $\mathrm{Cl}^{+}(\mathfrak{a})$'', we mean that both $\mathfrak{b}$ and $\mathfrak{c}$ are coprime with $\mathfrak{a}$ and they have the same image under this map.

One may also define the ray class group in terms of ideals. Let $J_F$ be the group of fractional ideals in $F$. If $S$ is a finite set of prime ideals in $F$, we denote by $J_F^S$ the subgroup of $J$ generated by the prime ideals not in $S$. Define
\[
F^S=\{x\in F: (x)\in J_F^S\}=\{x\in F: v_\mathfrak{p}(x)=0 \text{ for all finite }\mathfrak{p}\in S\}.
\]
Given a modulus $m$, we denote by $F_{m,1}$ the set consisting of elements $a\in F^\times$ satisfying
$$
\Big\{\begin{aligned}
	\mathrm{ord}_{v}(a-1) & \geq m(v) & & \,\, \text{all nonarchimedean $v$ with $m(v)>0$,} \\
	a_v &>0 & & \text { all real archimedean $v$ with $m(v)>0$,}
\end{aligned}
$$
where $a_v$ is the image of $a$ in $F_v$.
If $S(m)=\{\text{prime ideals } \mathfrak{p}\colon m(\mathfrak{p})>0\}$, then the ray class group modulo $m$ can also be defined by $J^{S(m)}/F_{m,1}$. By \cite[Theorem 1.7, Chapter 5]{milneCFT}, we have the following exact sequence
\begin{equation}\label{exactseq}
	0\rightarrow \cO_F^\times/(\cO_F^\times\cap F_{m,1})\rightarrow F^{S(m)}/F_{m,1}\rightarrow \mathrm{Cl}(m)\rightarrow \mathrm{Cl}_F\rightarrow 0,
\end{equation}
where $\mathrm{Cl}_F$ is the class group of $F$. Moreover, we have the following isomorphism
$$F^{S(m)}/F_{m,1}\simeq \prod_{\substack{v|\infty \text{ real}\\ m(v)>0}}\{\pm1\}\times\prod_{\substack{\mathfrak{p} \\ m(\kp)>0}}(\cO_F/\mathfrak{p}^{m(\mathfrak{p})})^\times.$$
As a result, if we define $\phi_F(\mathfrak{m})=\mathrm{N}\mathfrak{m}\prod_{\mathfrak{p}|\mathfrak{m}}(1-\frac{1}{\mathrm{N}\mathfrak{p}})$ to be Euler's totient function in $F$, then $$h(\mathfrak{m})=h\cdot2^r\cdot \phi_F(\mathfrak{m})\cdot |\cO_F^\times/(\cO_F^\times\cap F_{m_\mathfrak{m},1})|^{-1},$$ where $h$ is the class number of $F$ and $r$ is the number of real embeddings of $F$. One can show that $\phi_F(\mathfrak{m})\gg \mathrm{N}\mathfrak{m}/\log \mathrm{N}\mathfrak{m}$.

For any character $\chi$ on $\mathrm{Cl}^{+}(\mathfrak{a})$, there is a unitary Hecke character which is also denoted by $\chi=\prod_v\chi_v$ such that $\chi(\mathfrak{p})=\chi_{\mathfrak{p}}(\varpi_{\mathfrak{p}})$ if $\mathfrak{p}\nmid \mathfrak{q}_\chi$. One can see that the conductor of $\chi$ divides $\mathfrak{a}$. We say that $\chi$ is primitive modulo $\mathfrak{a}$ if $\mathfrak{q}_\chi=\mathfrak{a}$.  Now, for any $\pi \in \mathfrak{F}_n$, one has  $\pi\otimes \chi \in \mathfrak{F}_n$. 
By \cite{Cogdell-2008}, the standard $L$-function associated with $\pi\otimes\chi$ equals
\begin{equation*}\label{tw-lfunction}
	L(s,\pi\otimes\chi)=\sum_{\mathfrak{n}\subset \cO_F}\frac{\lambda_{\pi\otimes\chi}(\mathfrak{n})}{{\mathrm{N}\mathfrak{n}}^s}.
\end{equation*}
For a prime $\mathfrak{p}\nmid (\mathfrak{q}_\pi,\mathfrak{q}_\chi)$, we have $\{\alpha_{j,\pi\otimes\chi}(\mathfrak{p})\}=\{\alpha_{j,\pi}(\mathfrak{p})\alpha_{\chi}(\mathfrak{p})\}$. Recall that $\alpha_{\chi}(\mathfrak{p})=\chi_{\mathfrak{p}}(\varpi_{\mathfrak{p}})$ for any uniformizer $\varpi_{\mathfrak{p}}$ in $F_{\mathfrak{p}}$ if $\mathfrak{p}\nmid\mathfrak{q}_\chi$.
	We set $\chi(\mathfrak{p})=0$ for $\mathfrak{p|\mathfrak{q}_\chi}$, then by discussion above  we have 
	\begin{equation}\label{tw-coeff-deco}
	\lambda_{\pi\otimes\chi}(\mathfrak{n})=\lambda_{\pi}(\mathfrak{n})\chi(\mathfrak{n})\quad \text{when } (\mathfrak{n},\mathfrak{q}_\chi)=\cO_F,
	\end{equation}
	and when $\chi$ is primitive, we have for $\re(s)>1$ the identity
	\begin{equation}\label{tw-lfunction-2}
	\sum_{\mathfrak{n}\subset \cO_F}\frac{\lambda_{\pi}(\mathfrak{n})\chi(\mathfrak{n})}{{\mathrm{N}\mathfrak{n}}^s}=\prod_{\mathfrak{p}}\prod_{j=1}^m\Big(1-\frac{\alpha_{j,\pi}(\mathfrak{p})\chi(\mathfrak{p})}{\N\mathfrak{p}^{s}}\Big)^{-1}=L(s, \pi\otimes \chi)\prod_{\mathfrak{p}|\mathfrak{q}_\chi}\prod_{j=1}^{n}\Big({1-\frac{\alpha_{j,\pi\otimes\chi}(\mathfrak{p}) }{\N\mathfrak{p}^{s}}}\Big).
	\end{equation}



\section{Preliminary reductions and a generalized Vaughan identity}

We will prove the following theorem.

\begin{theorem}
	\label{thm:prime_power_version}
Let $\pi\in\mathfrak{F}_n^{\flat}$, $\eta=\max\{2,\frac{n}{2}\}$, $A>0$, and $B=2^{n[F:\Q]/4}(6a+12n+54)$.  If $x\geq 3$ and $Q=x^{\frac{1}{\eta}}(\log x)^{-B}$, then
\[
\sum_{\mathrm{N}\mathfrak{m}\leq Q}\frac{h(\mathfrak{m})}{\phi_F(\mathfrak{m})}\max_{(\mathfrak{a},\mathfrak{m})=\cO_F}\max_{y\leq x}\Big|\sum_{\substack{\mathrm{N}\mathfrak{p}\leq y \\ \mathfrak{p}\equiv \mathfrak{a} \text{ in } \mathrm{Cl}^{+}(\mathfrak{m})}}a_\pi(\mathfrak{n})\Lambda_F(\kn)\Big|\ll_{A} \frac{x}{(\log x)^A}.
\]
\end{theorem}

Assuming Theorem \ref{thm:prime_power_version}, we prove Theorem \ref{thm 1.1}.

\begin{proof}[Proof of Theorem \ref{thm 1.1}]
We will argue that Theorem \ref{thm:prime_power_version} implies that
\begin{equation}
\label{eqn:chebyshevversion}
\sum_{\mathrm{N}\mathfrak{m}\leq Q}\frac{h(\mathfrak{m})}{\phi_F(\mathfrak{m})}\max_{(\mathfrak{a},\mathfrak{m})=\cO_F}\max_{y\leq x}\Big|\sum_{\substack{\mathrm{N}\mathfrak{p}\leq y \\ \mathfrak{p}\equiv \mathfrak{a} \text{ in } \mathrm{Cl}^{+}(\mathfrak{m})}}\lambda_\pi(\mathfrak{p})\log\N\kp\Big|\ll_{A} \frac{x}{(\log x)^A}.
\end{equation}
The desired result will then follow by partial summation.  To see that Theorem \ref{thm:prime_power_version} implies \eqref{eqn:chebyshevversion}, note that if $\kp$ is a prime ideal, then $a_{\pi}(\kp)\Lambda_F(\kp)=\lambda_{\pi}(\kp)\log\N\kp$.  To estimate the contribution from higher powers of prime ideals, we observe that by \eqref{boundcoeff}, we have
\begin{align*}
\sum_{\substack{\mathrm{N}\mathfrak{p}^k\leq y \\ \mathfrak{p}^k\equiv\mathfrak{a} \text{ in } \mathrm{Cl}^{+}(\mathfrak{m}) \\ k\geq 2}}a_\pi(\mathfrak{p}^k)\log \mathrm{N}\mathfrak{p}
\ll_{\epsilon} y^{\theta_n+\varepsilon}\sum_{\substack{\mathrm{N}\mathfrak{p}^k\leq y \\ \mathfrak{p}^k\equiv\mathfrak{a} \text{ in } \mathrm{Cl}^{+}(\mathfrak{m}) \\ k\geq 2}}1.
\end{align*}
So the contribution of these terms to the average in Theorem \ref{thm:prime_power_version} is
\begin{align*}
&\ll\sum_{\mathrm{N}\mathfrak{m}\leq Q} \frac{h(\mathfrak{m})}{\phi_F(\mathfrak{m})}\max_{(\mathfrak{a},\mathfrak{m})=\cO_F} \max_{y\leq x}y^{\theta_n+\varepsilon}\sum_{\substack{\mathrm{N}\mathfrak{p}^k\leq y \\ \mathfrak{p}^k\equiv\mathfrak{a} \text{ in } \mathrm{Cl}^{+}(\mathfrak{m}) \\ k\geq 2}}1
\ll \sum_{\mathrm{N}\mathfrak{m}\leq Q} \max_{\mathrm{N}\mathfrak{a}\leq x}x^{\theta_n+\varepsilon}\sum_{\substack{\mathrm{N}\mathfrak{p}^k\leq x \\ \mathfrak{p}^k\equiv\mathfrak{a} \text{ in } \mathrm{Cl}^{+}(\mathfrak{m}) \\ k\geq 2}}1  \\
&\ll x^{\theta_n+\varepsilon}\max_{\mathrm{N}\mathfrak{a}\leq x}\sum_{\substack{\mathrm{N}\mathfrak{p}^k\leq x \\  k\geq 2}}\sum_{\substack{\mathrm{N}\mathfrak{m}\leq Q \\ \mathfrak{p}^k\equiv\mathfrak{a} \text{ in } \mathrm{Cl}^{+}(\mathfrak{m}) \\ (\mathfrak{m},\mathfrak{p}^k)=(\mathfrak{m},\mathfrak{a})=\cO_F }}1,
\end{align*}
by the convention that $\sum_{\mathfrak{b}\equiv\mathfrak{a} \text{ in } \mathrm{Cl}^{+}(\mathfrak{m})}1=0$ if $(\mathfrak{a},\mathfrak{m})\neq \cO_F$.  We now argue that
\begin{equation}
\label{6.1}
\max_{\N\ka\leq x}\max_{\substack{\N\kp^k\leq x \\ k\geq 2}}\sum_{\substack{\mathrm{N}\mathfrak{m}\leq Q \\ \mathfrak{p}^k\equiv\mathfrak{a} \text{ in } \mathrm{Cl}^{+}(\mathfrak{m}) \\ (\mathfrak{m},\mathfrak{p}^k)=(\mathfrak{m},\mathfrak{a})=\cO_F }}1\ll_{\epsilon}x^{\epsilon}.
\end{equation}
Once we establish this, it follows from the prime ideal theorem that
{\small
\begin{equation}\label{primetopower}
\sum_{\mathrm{N}\mathfrak{m}\leq Q} \frac{h(\mathfrak{m})}{\phi_F(\mathfrak{m})}\max_{(\mathfrak{a},\mathfrak{m})=\cO_F} \max_{y\leq x}\Big|\sum_{\substack{\mathrm{N}\mathfrak{n}\leq y \\ \mathfrak{n}\equiv\mathfrak{a} \text{ in } \mathrm{Cl}^{+}(\mathfrak{m})}}a_\pi(\mathfrak{n})\Lambda_F(\mathfrak{n})-\sum_{\substack{\mathrm{N}\mathfrak{p}\leq y \\ \mathfrak{p}\equiv\mathfrak{a} \text{ in } \mathrm{Cl}^{+}(\mathfrak{m})}}\lambda_\pi(\mathfrak{p})\log\N\kp\Big|\ll_{\epsilon}x^{\frac{1}{2}+\theta_n+\varepsilon}.\hspace{-1mm}
\end{equation}}%
Since $x^{\frac{1}{2}+\theta_n+\epsilon}\ll_A x(\log x)^{-A}$, this finishes the passage from powers of prime to prime ideals.

In order to prove \eqref{6.1}, we begin with the fact that if $\mathfrak{b}\equiv\mathfrak{a}$ in $\mathrm{Cl}^{+}(\mathfrak{m})$ with $(\mathfrak{m},\mathfrak{b})=(\mathfrak{m},\mathfrak{a})=\cO_F$ and $\mathrm{N}\mathfrak{b}$, $\mathrm{N}\mathfrak{a}\leq x$, then there exists $\omega\in F^\times$ such that $\omega$ is totally positive, $\kp|\mathfrak{m}$ implies $\mathrm{ord}_{\kp}(\omega-1)\geq \mathrm{ord}_{\kp}(\mathfrak{m})$, and $(\omega)=\kb \ka^{-1}$.  Let $h$ be the class number of $F$, then $\omega^h=ba^{-1}$ for some nonzero $a,b\in \cO_F$ with $\mathfrak{b}^h=(b)$ and $\mathfrak{a}^h=(a)$. Since $\mathrm{ord}_\mathfrak{p}(a)=0$ for $\mathfrak{p}|\mathfrak{m}$, $\mathrm{ord}_\mathfrak{p}(b-a)=\mathrm{ord}_\mathfrak{p}(b-a)-\mathrm{ord}_\mathfrak{p}(a)=\mathrm{ord}_\mathfrak{p}(\omega^h-1)=\mathrm{ord}_\mathfrak{p}(\omega-1)+\mathrm{ord}_\mathfrak{p}(\omega^{h-1}+\cdots +1)\geq\mathrm{ord}_\mathfrak{p}(\mathfrak{m})$ if $\mathfrak{p}|\mathfrak{m}$. Hence $\mathfrak{m}|(b-a)$. Recall that if $a\in F^\times$, then $\prod_v|a|_v=1$. As a result, $\prod_{v|\infty}|a|_v\leq x^h$ and $\prod_{v|\infty}|b|_v\leq x^h$. All of the conjugates of an algebraic integer are algebraic integers, so their absolute values have a uniform lower bound depending only on $F$. We can see from this fact that $|a|_v, |b|_v\ll x^h$ for all $v|\infty$. Hence $N(b-a)=\prod_{v|\infty}|b-a|_v\ll x^{h[F:\Q]}$. One can check that there are at most $\tau_{[F:\mathbb{Q}]}(m)$ integral ideals with norm $m$, where $\tau_{[F:\mathbb{Q}]}(m)$ is the $m$-th Dirichlet coefficient of $\zeta(s)^{[F:\Q]}$. Since $\tau_{[F:\Q]}(m)\ll_{\epsilon}m^{\epsilon}$, the innermost sum can be bounded by $\tau_{[F:\mathbb{Q}]}(\N(b-a))$, which is therefore $\ll_{\epsilon}x^\varepsilon$, as desired.
\end{proof}

Our proof of Theorem \ref{thm:prime_power_version} partially follows the approach in Chapter 9 of \cite{FI-2010}; see also \cite{Tao-2015}. As stated in \cite{FI-2010}, what we need is a combinatorial identity for sums over primes to produce a bilinear form to which the large sieve inequality can be applied. We choose to use a generalized version of  Vaughan identity.  Define
\begin{equation*}
M(s)=\sum_{\mathrm{N}\mathfrak{n}\leq X}\frac{\mu_\pi(\mathfrak{n})}{{\mathrm{N}\mathfrak{n}}^s},\qquad N(s)=\sum_{\mathrm{N}\mathfrak{n}\leq Y}\frac{\Lambda_F(\mathfrak{n})a_\pi(\mathfrak{n})}{{\mathrm{N}\mathfrak{n}}^s}.
\end{equation*}

\begin{lemma}\label{vaughan-hg}
	Let $X\geq 1$ and $Y\geq 1$.  If $\kn$ is an integral ideal with $\mathrm{N}\mathfrak{n}>Y$, then
	 we have
	\begin{equation*}
	\begin{aligned}
	\Lambda_F(\mathfrak{n})a_\pi(\mathfrak{n})&=\sum_{\substack{\mathfrak{n=ab}\\ \mathrm{N}\mathfrak{b}\leq X}} \mu_\pi(\mathfrak{b})\lambda_\pi(\mathfrak{a})\log \mathrm{N}\mathfrak{a}-{\underset{\substack{\mathfrak{n=abc}\\\mathrm{N}\mathfrak{b}\leq X,\mathrm{N}\mathfrak{c}\leq Y}}{\sum\sum}} \lambda_\pi(\mathfrak{a})\mu_\pi(\mathfrak{b})\Lambda_F(\mathfrak{c})a_\pi(\mathfrak{c})\\
	&+{\underset{\substack{\mathfrak{n=abc}\\ \mathrm{N}\mathfrak{b}>X,\mathrm{N}\mathfrak{c}> Y}}{\sum\sum}} \lambda_\pi(\mathfrak{a})\mu_\pi(\mathfrak{b})\Lambda_F(\mathfrak{c})a_\pi(\mathfrak{c}).
	\end{aligned}
	\end{equation*}
\end{lemma}
\begin{proof}
	If $\re(s)>1$, then we have the identity
\begin{eqnarray*}
	\frac{L'}{L}(s, \pi)&=&L'(s, \pi)M(s)+ L(s, \pi)M(s)N(s)\\
	&&+\Big(\frac{L'}{L}(s, \pi)+N(s)\Big)\Big(1-L(s, \pi)M(s)\Big)-N(s).
\end{eqnarray*}
Once we identify the coefficients of $\N\kn^{-s}$ on each side, we obtain the desired result.
\end{proof}

We apply Lemma \ref{vaughan-hg} with $X=Y<y$ for the $L$-function $L(s, \pi)$ and find that
\[
\sum_{\substack{\mathrm{N}\mathfrak{n}\leq y \\ \mathfrak{n}\equiv\mathfrak{a} \text{ in } \mathrm{Cl}^{+}(\mathfrak{m})}}\Lambda_F(\mathfrak{n})a_\pi(\mathfrak{n})=S_1+S_2-S_3+S_4,
\]
where
\begin{align}
S_1 &:=\sum_{\substack{\mathrm{N}\mathfrak{n}\leq X \\ \mathfrak{n}\equiv\mathfrak{a} \text{ in } \mathrm{Cl}^{+}(\mathfrak{m})}}\Lambda_F(\mathfrak{n})a_\pi(\mathfrak{n}),  \\
S_2 &:= \sum_{\mathrm{N}\mathfrak{b}\leq X} \mu_\pi(\mathfrak{b}) \sum_{\substack{\mathrm{N}\mathfrak{c}\leq y/\mathrm{N}\mathfrak{b}\\ \mathfrak{c}\equiv\mathfrak{ab^{-1}} \text{ in } \mathrm{Cl}^{+}(\mathfrak{m})}}\lambda_\pi(\mathfrak{c})\log \mathrm{N}\mathfrak{c}, \\
S_3 &:= \sum_{\mathrm{N}\mathfrak{b}\leq X}  \sum_{\mathrm{N}\mathfrak{c}\leq X} \sum_{\substack{\mathrm{N}\mathfrak{d}\leq y/\mathrm{N}\mathfrak{bc}\\ \mathfrak{d}\equiv\mathfrak{a(bc)^{-1}} \text{ in } \mathrm{Cl}^{+}(\mathfrak{m})}}\lambda_\pi(\mathfrak{d})\mu_\pi(\mathfrak{b})\Lambda_F(\mathfrak{c})a_\pi(\mathfrak{c})\\ 
&=  \sum_{\substack{\mathrm{N}\mathfrak{n}\leq X^2}} \Big(\sum_{\substack{\mathfrak{bc=n}\\ \mathrm{N}\mathfrak{b}\leq X,\mathrm{N}\mathfrak{c}\leq X}}\mu_\pi(\mathfrak{b})\Lambda_F(\mathfrak{c})a_\pi(\mathfrak{c})\Big)\sum_{\substack{\mathrm{N}\mathfrak{d}\leq y/\mathrm{N}\mathfrak{n}\\ \mathfrak{d}\equiv\mathfrak{an^{-1}} \text{ in } \mathrm{Cl}^{+}(\mathfrak{m})}}\lambda_\pi(\mathfrak{d}),\notag\\
S_4 &:= \sum_{\mathrm{N}\mathfrak{b}>X}  \sum_{\mathrm{N}\mathfrak{c}> X}\sum_{\substack{\mathrm{N}\mathfrak{d}\leq y/\mathrm{N}\mathfrak{bc}\\ \mathfrak{d}\equiv\mathfrak{a(bc)^{-1}} \text{ in } \mathrm{Cl}^{+}(\mathfrak{m})}}\lambda_\pi(\mathfrak{d})\mu_\pi(\mathfrak{b})\Lambda_F(\mathfrak{c})a_\pi(\mathfrak{c})\\
&= \sum_{X<\substack{\mathrm{N}\mathfrak{n}<y/X}}\Big(\sum_{\substack{\mathfrak{bd=n}\\ \mathfrak{b}> X}}\mu_\pi(\mathfrak{b})\lambda_\pi(\mathfrak{d})\Big)
\sum_{\substack{X<\mathrm{N}\mathfrak{c} \leq y/\mathrm{N}\mathfrak{n}\\ \mathfrak{c}\equiv\mathfrak{an^{-1}} \text{ in } \mathrm{Cl}^{+}(\mathfrak{m})}} \Lambda_F(\mathfrak{c})a_\pi(\mathfrak{c}).\notag
\end{align}

After applying the identity with suitable parameters $X$ and $Y$, we are going to estimate sums involving $S_1$, $S_2$, $S_3$, and $S_4$.  Note that their definitions depend on $y$, $\pi$, $\mathfrak{a}$, and the parameter for truncation $X$ in Vaughan's identity. The estimates of $S_i$ rely on Theorem \ref{thm-bv-integers} and Lemma \ref{max-lem-Vaughan} (see below). We only demonstrate in this section the outline of the proof assuming Theorem \ref{thm-bv-integers} and Lemma \ref{max-lem-Vaughan}. We estimate the sum involving $S_2$ because it is the most typical one. That is, we need to estimate
\begin{equation*}
\sum_{\mathrm{N}\mathfrak{m}\leq Q} \frac{h(\mathfrak{m})}{\phi_F(\mathfrak{m})}\max_{(\mathfrak{a},\mathfrak{m})=\cO_F} \max_{y\leq x}|S_2|.
\end{equation*}
It behooves us to decompose $S_2$ as $S_2'+S_2''$, where
\[
S_2' := \sum_{\mathrm{N}\mathfrak{b}\leq H} \mu_\pi(\mathfrak{b}) \sum_{\substack{\mathrm{N}\mathfrak{c}\leq y/\mathrm{N}\mathfrak{b}\\ \mathfrak{c}\equiv\mathfrak{ab^{-1}} \text{ in } \mathrm{Cl}^{+}(\mathfrak{m})}}\lambda_\pi(\mathfrak{c})\log \mathrm{N}\mathfrak{c},\quad S_2'':=\sum_{H<\mathrm{N}\mathfrak{b}\leq X}\mu_\pi(\mathfrak{b}) \sum_{\substack{\mathrm{N}\mathfrak{c}\leq y/\mathrm{N}\mathfrak{b}\\ \mathfrak{c}\equiv\mathfrak{ab^{-1}} \text{ in } \mathrm{Cl}^{+}(\mathfrak{m})}}\lambda_\pi(\mathfrak{c})\log \mathrm{N}\mathfrak{c}
\]
for some parameter $H<X$. Theorem \ref{thm-bv-integers} and partial summation give that
\begin{equation*}
	\sum_{\mathrm{N}\mathfrak{m}\leq Q} \frac{h(\mathfrak{m})}{\phi_F(\mathfrak{m})}\max_{(\mathfrak{a},\mathfrak{m})=\cO_F} \max_{y\leq x}|S_2'|\ll\frac{Hx}{(\log x)^A}
\end{equation*}
for some $B$ depending on $A$. For $S_2''$, we use Lemma \ref{max-lem-Vaughan} and the remark below it to obtain
\begin{equation*}
		\sum_{\mathrm{N}\mathfrak{m}\leq Q} \frac{h(\mathfrak{m})}{\phi_F(\mathfrak{m})}\max_{(\mathfrak{a},\mathfrak{m})=\cO_F} \max_{y\leq x}|S_2''|\ll (\log xQ)^5 \Big(Q\sqrt{x}+\sqrt{Xx}+\frac{x}{\sqrt{H}}+\frac{x}{(\log x/X)^{A}}\Big).
\end{equation*}
Lemma \ref{max-lem-Vaughan} requires a Siegel--Walfisz condition for the sequence $\{\lambda_\pi(\mathfrak{c})\log \N\mathfrak{c}\}$, which we prove in Section 9.  Moreover, the result relies on the $\ell^2$-estimates of $\mu_\pi(\mathfrak{b})$ and $\lambda_\pi(\mathfrak{c})\log \mathrm{N}\mathfrak{c}$, which is also given in Section 9 based on the inequalities in Section 5. For the sum involving $S_3$ and $S_4$ can be treated similarly, but we also require a Siegel--Walfisz condition for the sequence $\{a_\pi(\mathfrak{c})\Lambda_F(\mathfrak{c})\}$ in $S_4$.  This condition, given by Corollary \ref{cor-SW-thm}, is proved in Section 4; it relies on Theorem \ref{thm:siegel}.

Now it remains to prove Theorem \ref{thm-bv-integers} and Lemma \ref{max-lem-Vaughan}, which are proved in Sections 7 and 8, respectively.  Note that Lemma \ref{max-lem-Vaughan} is not of the form as in \cite{FI-2010}. For the proof of Lemma 8.2, we employ the trick of Fourier transform as Vaughan did in \cite{Vaughan-1980}.  Sections 5 and 6 supply several important estimates for our proofs of Theorem \ref{thm-bv-integers} and Lemma \ref{max-lem-Vaughan}.

%

\section{Zero-free regions}
\label{sec:ZFR}

Let $\pi\in\mathfrak{F}_n$.  We let $\mathbbm{1}\in\mathfrak{F}_1$ denote the trivial representation, whose $L$-function is the Dedekind zeta function $\zeta_F(s)$.  Recall that $\mathfrak{F}_n^{\flat}\subseteq \mathfrak{F}_n$ is the subset consisting of $\pi\in\mathfrak{F}_n$ such that $\pi=\tilde{\pi}$ and $\pi\neq\pi\otimes\chi$ for all nontrivial quadratic primitive Hecke characters $\chi$.  In this section, we prove a zero-free region for $L(s,\pi\otimes\chi)$ which is comparable to that of Dirichlet characters, including the first unconditional $\N\kq$-aspect bound on a possible Landau--Siegel zero.  We then use this zero-free region along with standard contour integration techniques to prove an analogue of the Siegel--Walfisz theorem for the Dirichlet coefficients $\lambda_{\pi}(\kp)$.  We now present the main result of this section.

\begin{theorem}
\label{thm:zero-free_region}
	Let $Q\geq 3$ and $\pi\in\mathfrak{F}_{n}^{\flat}$.  There exists a constant $c_{\pi}>0$, depending effectively on $\pi$, such that for all primitive Hecke characters $\chi\pmod{\kq}$ with $\N\kq\leq Q$ with at most one exception, the $L$-function $L(s,\pi\otimes\chi)$ is nonzero in the region
	\[
	\re(s)\geq 1-\frac{c_{\pi}}{\log(Q(3+|\im(s)|))}.
	\]
	If the exceptional character $\chi_1$ exists, then
	\begin{itemize}
		\item $\chi_1$ is quadratic.
		\item $L(s,\chi_1)$ has exactly one zero $\beta_1$ in this region, and $\beta_1$ is both real and simple.
		\item For all $\epsilon>0$, there exists an ineffective constant $c_{\pi}(\epsilon)>0$ such that $\beta_1\leq 1-c_{\pi}(\epsilon)Q^{-\epsilon}$.
	\end{itemize}
\end{theorem}

\subsection{Preliminaries for the zero-free region}

We begin with a standard zero-free region.

\begin{lemma}
	\label{lem:ZFR}
	Let $\pi\in\mathfrak{F}_n$, and let $\chi\pmod{\kq}$ be a primitive Hecke character.  There exists an effectively computable constant $\Cl[abcon]{ZFR_const}=\Cr{ZFR_const}(\pi)>0$ such that $L(s,\pi\otimes\chi)\neq 0$ in the region
	\[
	\re(s)\geq 1-\frac{\Cr{ZFR_const}}{\log(\N\kq(3+|\im(s)|))}
	\]
with the possible exception of one real zero $\beta_1<1$ when $\pi\otimes\chi$ is self-dual.  When $\pi=\tilde{\pi}$, the exceptional zero can only exist when $\chi$ is primitive, nontrivial, and quadratic.
\end{lemma}
\begin{proof}
When $|\im(s)|\neq 0$ or $\pi\otimes\chi$ is not self-dual, then the result follows from \cite[Theroem A.1]{HB-2019} with $\pi$ (respectively $\pi'$) therein replaced by $\pi\otimes\chi$ (respectively $\mathbbm{1}$).  When $\im(s)=0$ and $\pi\otimes\chi$ is self-dual, then by \cite[Theorem A.1]{HB-2019}, there exists at most one zero $\beta_1<1$ in the stated region, while the nonvanishing of $L(1,\pi\otimes\chi)$ follows from \cite[Theorem A.1]{Lapid}.  If $\pi=\tilde{\pi}$, then $\pi\otimes\chi$ is self-dual if and only if $\chi$ is real and primitive.  When $\pi=\tilde{\pi}$ and $\chi$ is trivial, then by \cite[Theorem A.1]{Lapid}, there exists effectively computable constant $\Cl[abcon]{brumleyzfrtriv}=\Cr{brumleyzfrtriv}(\pi)>0$ such that if $1-\Cr{brumleyzfrtriv}\leq s<1$, then $L(s,\pi)\neq 0$.  This exhausts all cases once $\Cr{ZFR_const}$ is made suitably large (in an effective manner depending at most on $\pi$).
\end{proof}

Next, we quantify the idea that exceptional zeros are rare.

\begin{lemma}
\label{lem:Page}
	Let $\pi\in\mathfrak{F}_n$.  Among the primitive quadratic Hecke characters $\chi\pmod{\kq}$ with $\N\kq\leq Q$, at most one, say $\chi_1$, has the property that $L(s,\pi\otimes\chi)$ has a real zero $\beta_1$ in the interval
	\[
	1-\frac{\Cr{ZFR_const}}{\log Q}\leq s<1.
	\]
\end{lemma}

\begin{proof}
This follows from \cite[Theorem A]{HR-1995}.  We may take $\Cr{ZFR_const}$ to be the same as in Lemma \ref{lem:ZFR} once $\Cr{ZFR_const}$ is made suitably small (in a manner that depends at most on $\pi$).
\end{proof}

\subsection{Preceding literature}

Siegel proved that if $\chi\pmod{q}$ is a primitive nontrivial quadratic Dirichlet character, then for all $\epsilon>0$, there exists an ineffective constant $c(\epsilon)>0$ such that $L(1,\chi)>c(\epsilon)q^{-\epsilon}$.  All known proofs except for one by Bombieri \cite[Th{\'e}or{\`e}me 15]{Bombieri_crible} use the fact that if $\chi\pmod{q}$ and $\chi'\pmod{q'}$ are distinct primitive nontrivial quadratic Dirichlet characters, and $\chi''$ is he primitive Dirichlet character that induces $\chi\chi'$, then there exists a Dirichlet series $F(s)$, depending explicitly on $\chi$ and $\chi'$, such that $F(s)$ has
\begin{enumerate}[(i)]
\item a pole of odd order $r\geq 1$ at $s=1$,
\item nonnegative Dirichlet coefficients, and
\item an analytic continuation to suitable region past $\re(s)=1$ (e.g., $\mathbb{C}-\{1\}$),
\item and a residue at $s=1$ that has $L(1,\chi)$ as a factor with integral multiplicity at least one.
\end{enumerate}
To study $L(1,\chi)$, the most natural choice of $F(s)$ is $\zeta(s)L(s,\chi)L(s,\chi')L(s,\chi'')$.  A pole of odd order $r$ at $s=1$ is important; under the above hypotheses, the residue $R_F$ of $F(s)$ at $s=1$ satisfies $R_F>0$, and as $s\to 1$ along the real line, we have
\begin{equation}
\label{eqn:Laurent_asymp}
F(s)\sim \frac{R_F}{s-1}.
\end{equation}
If $0<\epsilon<1$ and there exists $\chi'\pmod{q'}$ such that $L(s_0,\chi')=0$ for some $s_0\in(1-\epsilon,1)$, then $F(s_0)\leq 0$.  On the other hand, if no such $\chi'$ exists, then by \eqref{eqn:Laurent_asymp}, we have that $F(s_0)\leq 0$ for some $s_0\in(1-\epsilon/2,1)$.  Therefore, for all $0<\epsilon<1$, there exists $s_0\in(1-\epsilon,1)$ and $\chi'\pmod{q'}$, both depending {\it only on $\epsilon$}, such that $F(s_0)\leq 0$.  Davenport's book \cite[Ch. 20]{Davenport} is a standard source; it gives Estermann's proof, which requires this argument as a key step.

Let $\pi\in\mathfrak{F}_n$, and let $\chi\pmod{\kq}$ and $\chi'\pmod{\kq'}$ be distinct nontrivial primitive quadratic Hecke characters.  Let $\psi$ be the primitive character that induces $\chi'\chi$ (whose conductor necessarily divides $\kq\kq'$).  The possible existence of the exceptional real zero of $L(s,\pi\otimes\chi)$ in Lemma \ref{lem:ZFR} was eliminated by Hoffstein and Ramakrishnan \cite[Theorem B]{HR-1995} under the assumption of automorphy for certain Rankin--Selberg convolutions depending on $\pi$.  When their automorphy hypothesis is not known to be satisfied, it is unclear how to construct a Dirichlet series $L(s)$ depending on $\chi$ and $\chi'$ with nonnegative Dirichlet coefficients, an analytic continuation, a pole of odd order at $s=1$, and a residue at $s=1$ that has $L(1,\pi\otimes\chi)$ as a factor with integral multiplicity at least one.

If we allow for a pole of {\bf even} order at $s=1$, then we can construct an $L(s)$ satisfying properties (ii)-(iv) above.  Let $\pi\in\mathfrak{F}_n$, and suppose that $\pi\neq\pi\otimes\nu$ for all primitive nontrivial quadratic Hecke characters over $F$.   Define $\Pi=\mathbbm{1}\boxplus\pi$ and $(\Pi\times\tilde{\Pi})_{\chi}=\chi\boxplus\pi\otimes\chi\boxplus\tilde{\pi}\otimes\chi\boxplus\pi\times(\tilde{\pi}\otimes\chi)$. Consider the representation
\begin{equation}
\label{eqn:pi_star_def}
\Pi^{\star} = (\Pi\times\tilde{\Pi})\boxplus (\Pi\times\tilde{\Pi})_{\chi}\boxplus (\Pi\times\tilde{\Pi})_{\chi'}\boxplus(\Pi\times\tilde{\Pi})_{\psi}
\end{equation}
along with its $L$-function
\begin{equation}
\label{eqn:aux_L-function}
\begin{aligned}
L(s,\Pi^{\star})&=\zeta_F(s)L(s,\chi)L(s,\chi')L(s,\psi)L(s,\pi) L(s,\tilde{\pi}) L(s,\pi\otimes\chi)L(s,\tilde{\pi}\otimes\chi)\\
&\cdot L(s,\pi\otimes\chi')L(s,\tilde{\pi}\otimes\chi') L(s,\pi\otimes\psi)L(s,\tilde{\pi}\otimes\psi)L(s,\pi\times\tilde{\pi})L(s,\pi\times(\tilde{\pi}\otimes\chi))\\
&\cdot L(s,\pi\times(\tilde{\pi}\otimes\chi'))L(s,\pi\times(\tilde{\pi}\otimes\psi)).
\end{aligned}
\end{equation}
Our twist hypothesis for $\pi$ ensures that $L(s,\Pi^{\star})$ is holomorphic on $\mathbb{C}-\{1\}$ with a pole of order two at $s=1$.  This auxiliary $L$-function was suggested by Molteni \cite[p. 141]{Molteni-2002} in a special case, with \eqref{eqn:aux_L-function} providing a natural generalization.  Instead of providing full details for how to prove a Siegel-type lower bound for $|L(1,\pi\otimes\chi)|$ using $L(s,\Pi^{\star})$, Molteni references a ``standard approach to Siegel-type theorems'' in a paper by Golubeva and Fomenko \cite{GF}.  However, in \cite[pp. 87-88]{GF}, Golubeva and Fomenko only say that Estermann's proof of Siegel's theorem (the version in Davenport \cite[Ch. 20]{Davenport}) applies to \eqref{eqn:pi_star_def} (with $\pi\in\mathfrak{F}_2$) ``after fairly tedious calculations.''

As stated above, \eqref{eqn:aux_L-function} satisfies properties (ii)-(iv), but not (i), since $L(s,\Pi^{\star})$ has a pole of order two at $s=1$.  In this situation, a Siegel-type lower bound does {\bf not} follow from a direct generalization of the arguments in \cite[Ch. 20]{Davenport}, or any other argument that proves Siegel's theorem using the above auxiliary $L$-function $\zeta(s)L(s,\chi)L(s,\chi')L(s,\chi'')$.  Since $L(s,\Pi^{\star})$ has a pole of order 2 at $s=1$ and nonnegative Dirichlet coefficients, it follows that there exists a constant $R_{\Pi^{\star}}>0$ such that as $s\to 1$ along the reals, we have
\[
L(s,\Pi^{\star})\sim \frac{R_{\Pi^{\star}}}{(s-1)^2}.
\]
Consequently, there exists $s'<1$ such that $L(s,\Pi^{\star})>0$ for all $s\in[s',1)$.  Therefore, it is no longer true that for all $0<\epsilon<1$, there exists $s_0\in(1-\epsilon,1)$ and $\chi'\pmod{\kq'}$ such that $L(s_0,\Pi^{\star})\leq 0$.  An identical error can also be found in \cite{Ichihara}.

In summary, a separate approach is needed in order to produce a lower bound for $|L(1,\pi\otimes\chi)|$ when $\epsilon$ is so small that $L(s_0,\Pi^{\star})>0$ for all $s_0\in(1-\epsilon,1)$.  Such an approach is provided by the following lemma.  This leads to a correction and substantial generalization of the works in \cite{Molteni-2002,GF,Ichihara}.

\begin{lemma}
\label{lem:Li_lower}
	Let $\chi\pmod{\kq}$ be a primitive quadratic Hecke character, and let $\pi\in\mathfrak{F}_n$.  If $L(s,\pi\otimes\chi)\neq 0$ in the region
	\[
	\re(s)\geq 1-\frac{\Cr{ZFR_const}}{\log(\N\kq(3+|\im(s)|))},
	\]
	then there exist constants $\Cl[abcon]{XLi1}=\Cr{XLi1}(\pi)>0$ and $\Cl[abcon]{XLi2}=\Cr{XLi2}(\pi)>0$ such that
	\[
	|L(1,\pi\otimes\chi)|\geq \Cr{XLi1}\exp(-\Cr{XLi2}\sqrt{\log \N\kq}).
	\]
\end{lemma}
\begin{proof}
	This follows from work of Li \cite[Corollary 7]{Lixian-2010}.  While the proofs in \cite{Lixian-2010} are performed over $F=\Q$, an extension over number fields follows {\it mutatis mutandis}.
\end{proof}

\subsection{An extension of Siegel's theorem}

Let $\Pi=\pi\boxplus\mathbbm{1}$, and define the numbers $a_{\Pi\times\tilde{\Pi}}(\kp^k)$ by the Dirichlet series identity
\[
\sum_{\kp}\sum_{k=1}^{\infty}\frac{a_{\Pi\times\tilde{\Pi}}(\kp^k)}{k\mathrm{N}\kp^{ks}}=\log L(s,\Pi\times\tilde{\pi}),\qquad\re(s)>1.
\]

\begin{lemma}
\label{lem:nonneg_coeffs}
	Let $\pi\in\mathfrak{F}_n$, let $\chi\pmod{\kq}$ and $\chi'\pmod{\kq'}$ be primitive quadratic Hecke characters, and let $\psi$ be the primitive Hecke character inducing $\chi'\chi$. Let $\Pi=\mathbbm{1}\boxplus\pi$, and recall the definition of $\Pi^{\star}$ in \eqref{eqn:pi_star_def}.  There exists an entire function $\mathcal{H}(s)=\mathcal{H}_{\pi}(s,\chi,\chi')$ such that
	\[
	\sum_{\kp\nmid\kq\kq'\kq_{\pi}}\sum_{k=1}^{\infty}\frac{a_{\Pi\times\tilde{\Pi}}(\kp^k)}{k\N\kp^{ks}}(1+\chi(\kp^k))(1+\chi'(\kp^k))=\log(L(s,\Pi^{\star})\mathcal{H}(s)).
	\]
	The Dirichlet coefficients $\lambda^{\star}(\kn)$ of $L(s,\Pi^{\star})\mathcal{H}(s)$ are nonnegative, and $\lambda^{\star}(\cO_F)=1$.  We have the bounds $|\mathcal{H}(1)|,|\mathcal{H}'(1)|\ll_{\epsilon}(\N\kq'\kq)^{\epsilon}$ for all $\epsilon>0$.  Finally, if $t\in\R$, then $|\mathcal{H}(\frac{1}{2}+it)|\ll (\N\kq'\kq)^{3n^2/2}$.
\end{lemma}
\begin{proof}
%
We determine $\mathcal{H}(s)$ explicitly using the local calculations in \cite[Lemma 2.1]{LRS-1999}:
\begin{align*}
	\mathcal{H}(s)=\prod_{\kp\nmid \kq\kq'\kq_{\pi}}&\Big\{\Big[\prod_{j=1}^n\Big(1-\frac{\alpha_{j,\pi\otimes\chi}(\kp)}{\N\kp^s}\Big)\Big(1-\frac{\alpha_{j,\tilde{\pi}\otimes\chi}(\kp)}{\N\kp^s}\Big)\Big]\Big[\prod_{j=1}^n \prod_{j'=1}^n\Big(1-\frac{\alpha_{j,j',\pi\times(\tilde{\pi}\otimes\chi)}(\kp)}{\N\kp^s}\Big)\Big]\\
	&\cdot 	\Big[\prod_{j=1}^n\Big(1-\frac{\alpha_{j,\pi\otimes\chi'}(\kp)}{\N\kp^s}\Big)\Big(1-\frac{\alpha_{j,\tilde{\pi}\otimes\chi'}(\kp)}{\N\kp^s}\Big)\Big]\Big[\prod_{j=1}^n \prod_{j'=1}^n\Big(1-\frac{\alpha_{j,j',\pi\times(\tilde{\pi}\otimes\chi')}(\kp)}{\N\kp^s}\Big)\Big]\\
	&\cdot 	\Big[\prod_{j=1}^n\Big(1-\frac{\alpha_{j,\pi\otimes\psi}(\kp)}{\N\kp^s}\Big)\Big(1-\frac{\alpha_{j,\tilde{\pi}\otimes\psi}(\kp)}{\N\kp^s}\Big)\Big]\Big[\prod_{j=1}^n \prod_{j'=1}^n\Big(1-\frac{\alpha_{j,j',\pi\times(\tilde{\pi}\otimes\psi)}(\kp)}{\N\kp^s}\Big)\Big]\Big\}.
\end{align*}
The claimed bounds of $|\mathcal{H}(1)|$, $|\mathcal{H}'(1)|$, and $|\mathcal{H}(\frac{1}{2}+it)|$ follow from \eqref{boundcoeff} and \eqref{boundcoeff2} along with the bound $\#\{\kp\colon \kp|\kn\}\ll (\log\log \N\kn)^{-1}\log\N\kn$.

The nonnegativity of $a_{\Pi\times\tilde{\Pi}}(\kp^k)$ follows from the proof of \cite[Lemma a]{HR-1995}.  The nonnegativity of $(1+\chi(\kp^k))(1+\chi'(\kp^k))$ follows from the fact that $\chi$ and $\chi'$ are quadratic.  Thus, the nonnegativity of the Dirichlet coefficients $\lambda^{\star}(\kn)$ follows by exponentiation.
\end{proof}

\begin{lemma}
	\label{lem:residue_calc}
Let $\pi\in\mathfrak{F}_n$, and let $\chi$, $\chi'$ and $\psi$ be as in Lemma \ref{lem:nonneg_coeffs}.  Suppose that $\pi\neq\pi\otimes\nu$ for all $\nu\in\{\chi,\chi',\psi\}$.  Recall the definition of $\Pi^{\star}$ from \eqref{eqn:pi_star_def}.  If $\beta\in(0,1)$, $x\geq 3$, and $\epsilon>0$, then
\[
\mathop{\mathrm{Res}}_{s=1-\beta}L(s+\beta,\Pi^{\star})\mathcal{H}(s+\beta)x^s\Gamma(s) \ll_{\chi',\beta,\epsilon}|L(1,\pi\otimes\chi)|\N\kq^{\epsilon}x^{1-\beta+\epsilon}.
\]
\end{lemma}
\begin{proof}
Define $H(s)= L(s,\Pi^{\star})\mathcal{H}(s) L(s,\pi\otimes\chi)^{-1}L(s,\tilde{\pi}\otimes\chi)^{-1}\zeta_F(s)^{-1}L(s,\pi\times\tilde{\pi})^{-1}$.  The hypothesis that $\pi\neq\pi\otimes\nu$ for all $\nu\in\{\chi,\chi',\psi\}$ ensures that $H(s)$ is entire.  In a neighbourhood of $s=1$, we have the Laurent expansions
\[
\zeta_F(s)=\frac{\kappa_{F}}{s-1}+\kappa_F'+O(s-1),\qquad L(s,\pi\times\tilde{\pi}) = \frac{\kappa_{\pi}}{s-1}+\kappa_{\pi}'+O(s-1).
\]
In view of these definitions, the residue $\mathop{\mathrm{Res}}_{s=1-\beta}L(s+\beta,\Pi^{\star})\mathcal{H}(s+\beta)x^s\Gamma(s)$ equals
\begin{align*}
	x^{1-\beta}\Gamma(1-\beta)L(1,\pi\otimes\chi)\Big(&H(1)L(1,\tilde{\pi}\otimes\chi)\kappa_F \kappa_{\pi}\Big(\log x+\frac{\Gamma'(1-\beta)}{\Gamma(1-\beta)}\Big)\\
	&+H(1)L(1,\tilde{\pi} \otimes\chi)(\kappa_{\pi}\kappa_F'+\kappa_F\kappa_{\pi}')\\
	&+\kappa_F \kappa_{\pi}(L(1,\tilde{\pi} \otimes\chi)H'(1)+L'(1,\pi \otimes\chi)H(1))\\
	&+H(1)\kappa_F\kappa_{\pi}\frac{L(1,\tilde{\pi}\otimes\chi)}{L(1,\pi\otimes\chi)}L'(1,\pi\otimes\chi)\Big).
\end{align*}
Because $\chi$ is quadratic, we have $L(1,\tilde{\pi}\otimes\chi)=\overline{L(1,\pi\otimes\chi)}$, so the ratio $\frac{L(1,\tilde{\pi}\otimes\chi)}{L(1,\pi\otimes\chi)}$ has modulus 1.  Therefore, the lemma will follow from the following estimates
\[
|L(1,\tilde{\pi}\otimes\chi)|,~|L'(1,\tilde{\pi}\otimes\chi)|,~|H(1)|,~|H'(1)|\ll_{\chi',\epsilon}\N\kq^{\epsilon}.
\]
The first three estimates follow directly from \cite[Theorem 2]{Lixian-2010}. For the last inequality, it suffices to know that if $\nu\in\{\chi,\chi',\psi\}$, then $|L'(1,\pi\times(\tilde{\pi}\otimes\nu))|\ll_{\chi',\epsilon}\N\kq^{\epsilon}$.  By Cauchy's integral formula for derivatives, we have
\[
L'(1,\pi\times(\tilde{\pi}\otimes\nu))=\frac{1}{2\pi i}\int_{\Omega}\frac{L(s,\pi\times(\tilde{\pi}\otimes\nu))}{(s-1)^2}ds\ll (\log\N\kq)\max_{z\in \Omega}|L(s,\pi\times(\tilde{\pi}\otimes\nu))|,
\]
where $\Omega$ is the circle of radius $O_{\chi'}(\frac{1}{\log \N\mathfrak{q}})$ centered at $s=1$. By \cite[Theorem 2]{Lixian-2010} and the Phragm\'en-Lindel\"of principle, we have $|H'(1)|\ll_{\chi',\epsilon}\N\kq^{\epsilon}$ to finish the proof.
\end{proof}

We now perform an auxiliary computation using Lemmata \ref{lem:nonneg_coeffs} and \ref{lem:residue_calc}.   If $x\geq 3$ and $\beta\in(0,1)$, then we compute
\[
	\frac{1}{2}\leq e^{-1/x}\leq \sum_{\kn}\frac{\lambda^{\star}(\kn)}{\N\kn^{\beta}}e^{-\N\kn/x}=\frac{1}{2\pi i}\int_{3-i\infty}^{3+i\infty}L(s+\beta,\Pi^{\star})\mathcal{H}(s+\beta)x^s\Gamma(s)ds.
\]
Once we push the contour to the line $\re(s) = \frac{1}{2}-\beta$, the contour integral equals
\[
\mathop{\mathrm{Res}}_{s=1-\beta}L(s+\beta,\Pi^{\star})\mathcal{H}(s+\beta)x^s\Gamma(s) + L(\beta,\Pi^{\star})\mathcal{H}(\beta)+\frac{1}{2\pi i}\int_{\frac{1}{2}-\beta-i\infty}^{\frac{1}{2}-\beta+i\infty}L(s+\beta,\Pi^{\star})\mathcal{H}(s+\beta))x^s\Gamma(s)ds.
\]
By Lemma \ref{lem:nonneg_coeffs}, we have for any $x\geq 3$ and $\beta\in(0,1)$ the bound 
\begin{equation}
\label{eqn:lower_bound_res}
\frac{1}{2}\leq \mathop{\mathrm{Res}}_{s=1-\beta}L(s+\beta,\Pi^{\star})\mathcal{H}(s+\beta)x^s\Gamma(s) + L(\beta,\Pi^{\star})\mathcal{H}(\beta) +O_{\chi',\beta,\epsilon}(\N\kq^{2(n+1)^2+\epsilon}x^{\frac{1}{2}-\beta}).
\end{equation}

\begin{proof}[Proof of Theorem \ref{thm:siegel}]
It suffices to let $\mathrm{N}\kq$ be large (with respect to $\pi$) and $(\log\N\kq)^{-1}<\epsilon<\frac{1}{2}$.  If there exists at most one primitive quadratic nontrivial Hecke character $\nu$ such that $L(s,\pi\otimes\nu)= 0$ for some $s\in(1-\frac{\epsilon}{4},1)$, then the desired result follows from Lemmata \ref{lem:ZFR} and \ref{lem:Li_lower} once we make $\Cr{XLi1}$ and $\Cr{XLi2}$ sufficiently small (depending at most on $\pi$ and the sole exceptional character, if it exists).  For the rest of the proof, we may assume that there exist two distinct primitive quadratic Hecke characters $\nu_1$ and $\nu_2$ such that both $L(s,\pi\otimes\nu_1)$ and $L(s,\pi\otimes\nu_2)$ vanish somewhere in the interval $(1-\frac{\epsilon}{4},1)$.

Subject to this hypothesis, we can choose $\chi'\in\{\nu_1,\nu_2\}-\{\chi\}$.  If we choose $\beta\in(1-\frac{\epsilon}{4},1)$ to be a point at which $L(s,\pi\otimes\chi')$ vanishes, then we may conclude that for all $0<\epsilon<\frac{1}{2}$, there exist $\chi'\pmod{\kq'}$ and $\beta\in(1-\frac{\epsilon}{4},1)$ (depending at most on $\epsilon$ and $\pi$) such that $L(\beta,\Pi^{\star})=0$.  With these choices of $\chi'\pmod{\kq'}$ and $\beta$, the bound \eqref{eqn:lower_bound_res} reduces to
\[
\frac{1}{2}\leq \mathop{\mathrm{Res}}_{s=1-\beta}L(s+\beta,\Pi^{\star})\mathcal{H}(s)x^s\Gamma(s) + O_{\epsilon}(\N\kq^{2(n+1)^2+\epsilon}x^{\frac{1}{2}-\beta}).
\]
Since $\beta\in(1-\frac{\epsilon}{4},1)$, it follows from Lemma \ref{lem:residue_calc} that
\[
1\ll_{\epsilon}|L(1,\pi\otimes\chi)|\N\kq^{\epsilon}x^{\frac{\epsilon}{2}} + \N\kq^{2(n+1)^2+\epsilon}x^{\frac{\epsilon}{2}-\frac{1}{2}}.
\]
We choose $x=\N\kq^{4(n+1)^2}|L(1,\pi\otimes\chi)|^{-2}$.  Note that there exist effectively computable constants $\Cl[abcon]{lb1}=\Cr{lb1}(\pi)>0$ and $\Cl[abcon]{lb2}=\Cr{lb2}(\pi)>0$ such that $|L(1,\pi\otimes\chi)|\leq \Cr{lb1}\exp(\Cr{lb2}\sqrt{\log\mathrm{N}\kq})$ by \cite[Theorem 3]{Lixian-2010}.  Therefore, since we have assumed that $\N\kq$ is large, we have that $x\geq 3$.  We achieve the desired result by solving for $|L(1,\pi\otimes\chi)|$, and rescaling $\epsilon$ in terms of $n$ alone.
\end{proof}

\begin{corollary}
\label{cor:siegel_zero}
Let $\pi\in\mathfrak{F}_n$, and suppose that $\pi\neq\pi\otimes\nu$ for all primitive quadratic Hecke characters $\nu$.  Let $\chi\pmod{\kq}$ be a primitive quadratic Hecke character.  For all $\epsilon>0$, there exists an ineffective constant $c_{\pi}(\epsilon)>0$ such that $L(s,\pi\otimes\chi)\neq 0$ for $s\geq1-c_{\pi}(\epsilon)\N\kq^{-\epsilon}$.
\end{corollary}
\begin{proof}
	Suppose that $L(s,\pi\otimes\chi)$ has a real exceptional zero $\beta_1$ in the region given by Lemma \ref{lem:ZFR}.  By the mean value theorem and Lemma \ref{thm:siegel}, there exists
	\[
	\sigma\in\Big[1-\frac{\Cr{ZFR_const}}{\log(3\N\kq)},1\Big]
	\]
	such that $|L'(\sigma,\pi\otimes\chi)|(1-\beta_1)=|L(1,\pi\otimes\chi)|\geq c_{\pi}'(\frac{\epsilon}{2})\N\kq^{-\epsilon/2}$.  Therefore, we have
	\[
	\beta_1\leq 1-\frac{c_{\pi}'(\frac{\epsilon}{2})}{\N\kq^{\epsilon/2}|L'(\sigma,\pi\otimes\chi)|}.
	\]
	The upper bound $|L'(\sigma,\pi\otimes\chi)|\ll_{\epsilon}\N\kq^{\epsilon/2}$ follows from \cite[Corollary 6]{Lixian-2010} for all $\sigma$ in our range, and the result follows.
\end{proof}

\begin{proof}[Proof of Theorem \ref{thm:zero-free_region}]
This follows from Lemmata \ref{lem:ZFR} and \ref{lem:Page} and Corollary \ref{cor:siegel_zero}.	
\end{proof}

\subsection{An estimate of Siegel--Walfisz type}

We apply our zero-free region in Theorem \ref{thm:zero-free_region} to prove the following result.

\begin{corollary}\label{cor-SW-thm}
Let $\pi\in\mathfrak{F}_n^{\flat}$ and $(\mathfrak{m},\mathfrak{a})=\cO_F$.  For all $A>0$, there exists an ineffective constant $\Cl[abcon]{SW}=\Cr{SW}(\pi,F,A)>0$ such that for $\N\mathfrak{m}\leq (\log x)^A$, we have
	\begin{equation}\label{twistL-Siegel}
	\sum_{\substack{\mathrm{N}\mathfrak{n}\leq x \\ \mathfrak{n}\equiv\mathfrak{a} \text{ in } \mathrm{Cl}^{+}(\mathfrak{m})}}\Lambda_F(\mathfrak{n})a_\pi(\mathfrak{n})\ll_{A} x\exp(-\Cr{SW}\sqrt{\log x}).
	\end{equation}
\end{corollary}

\begin{proof}
Using the orthogonality of characters, we find that
\[
\sum_{\substack{\mathrm{N}\mathfrak{n}\leq x \\ \mathfrak{n}\equiv\mathfrak{a} \text{ in } \mathrm{Cl}^{+}(\mathfrak{m})}}\Lambda_F(\mathfrak{n})a_\pi(\mathfrak{n})\ll_A (\log x)^A \max_{\N\mathfrak{m}\leq (\log x)^A} \max_{{\psi \in \widehat{\mathrm{Cl}^{+}(\mathfrak{m})}} }\Big|\sum_{\mathrm{N}\mathfrak{n}\leq x}\Lambda_F(\mathfrak{n})a_\pi(\mathfrak{n})\psi(\mathfrak{n})\Big|.
\]
If $\chi\pmod{\kq}$ is the primitive Hecke character that induces $\psi\pmod{\mathfrak{m}}$, then \eqref{boundcoeff} and our constraint that $\N\mathfrak{m}\leq (\log x)^A$ implies that for all $\varepsilon>0$, we have
\[
\Big|\sum_{\mathrm{N}\mathfrak{n}\leq x}\Lambda_F(\mathfrak{n})a_\pi(\mathfrak{n})\psi(\mathfrak{n})-\sum_{\mathrm{N}\mathfrak{n}\leq x}\Lambda_F(\mathfrak{n})a_{\pi\otimes\chi}(\mathfrak{n})\Big|\leq \sum_{\kp|\mathfrak{m}\kq_{\pi}}\sum_{\substack{k\geq 1 \\ \N\kp^k\leq x}}|a_{\pi}(\kp^k)|\log N\kp\ll_{\varepsilon} x^{\theta_n+\epsilon}.
\]
%
%
%
Without loss of generality, assume that $\chi$ is the exceptional character $\chi_1$ in Theorem \ref{thm:zero-free_region}, and let $\beta_1$ be the corresponding exceptional zero.  We proceed as in \cite[Theorem 5.13]{IK-2004} and conclude that
\[
\sum_{\mathrm{N}\mathfrak{n}\leq x}\Lambda_F(\mathfrak{n})a_{\pi\otimes\chi_1}(\mathfrak{n})= -\frac{x^{\beta_{1}}}{\beta_{1}}+O\Big( \mathrm{N}\mathfrak{m}^{\frac{n}{2}} x\exp(-\tfrac{c_{\pi}}{2}\sqrt{\log x})+\sum_{x<\N\kn\leq x+\frac{x}{\exp(\frac{1}{3}\sqrt{\log x})}}\Lambda_F(\kn)|a_{\pi}(\kn)|\Big).
\]
Since $2|a_{\pi}(\kn)|\leq 1+a_{\pi\times\tilde{\pi}}(\kn)$ \cite[Theorem A.1]{ST-2019} and $L(1+it,\pi\times\tilde{\pi})\neq 0$ for all $t\in\R$ \cite[Appendix]{Lapid}, we find that $\sum_{\N\kn\leq x}|a_{\pi}(\kn)|\Lambda_F(\kn)\ll x$.  Therefore, by the Cauchy--Schwarz inequality and the bound $\N\mathfrak{m}\leq(\log x)^A$, it follows that
\[
\sum_{\mathrm{N}\mathfrak{n}\leq x}\Lambda_F(\mathfrak{n})a_{\pi\otimes\chi_1}(\mathfrak{n})= -\frac{x^{\beta_{1}}}{\beta_{1}}+O( x\exp(-\min\{\tfrac{c_{\pi}}{4},\tfrac{1}{6}\}\sqrt{\log x})).
\]

Theorem \ref{thm:zero-free_region} gives the bound $\beta_{1}\leq 1-c_\pi(\varepsilon)\mathrm{N}\mathfrak{m}^{-\varepsilon}$ for any fixed $\epsilon>0$.  Since $\mathrm{N}\mathfrak{m}\leq (\log x)^A$ for some fixed $A>0$, we set $\varepsilon=1/(2A)$ so that
 \[
 x^{\beta_{1}}\ll x\exp(-c_\pi(\tfrac{1}{2A})(\log x)\mathrm{N}\mathfrak{m}^{-\varepsilon})\ll x\exp(-c_\pi(\tfrac{1}{2A})\sqrt{\log x}).
 \]
We put $c=\frac{1}{10}\min\{\frac{c_{\pi}}{3}, c_\pi(\frac{1}{2A}),\frac{1}{6}\}$, and the desired result follows.
\end{proof}

\section{Auxiliary estimates}
\label{sec-inequality}

We begin with a combinatorial lemma due to Soundrarajan \cite{Sound-2010}. It is useful to prove some inequalities.
\begin{lemma}\label{sound-auto}
	Let $b(1), b(2),\ldots$ be a sequence of complex numbers. Define
	the sequence $c(0)=1, c(1), c(2),\ldots$ by means of the formal identity
	\begin{equation*}
	\exp\big(\sum_{k=1}^{\infty}\frac{b(k)}{k}x^{k}\big)=\sum_{m=0}^{\infty}c(m)x^{m}.
	\end{equation*}
	Define the sequence $C(0)=1, C(1), C(2),\ldots$ by means of the formal identity
	\begin{equation*}
	\exp\big(\sum_{k=1}^{\infty}\frac{|b(k)|^2}{k}x^{k}\big)=\sum_{m=0}^{\infty}C(m)x^{m}.
	\end{equation*}
	Then $|c(m)|^2\leq C(m)$ for all $m$.
\end{lemma}

Now we here introduce several arithmetic inequalities, which will be used below. The first one is about the coefficients of logarithmic derivatives of $L$-functions (see the appendix by Brumley in \cite{ST-2019})
\begin{equation}\label{coeffpair}
|a_{\pi}(\mathfrak{n})|^2\leq a_{\pi\times\tilde{\pi}}(\mathfrak{n})
\end{equation}
for any $\mathfrak{n}\subset \cO_F,$ where $a_{\pi\times\tilde{\pi}}(\mathfrak{n})$ is given by
\[
-\frac{L^{\prime}}{L}(s,\pi\times\tilde{\pi})=\sum_{\mathfrak{n}\subset \cO_F}\frac{\Lambda_F(\mathfrak{n})a_{\pi\times\tilde{\pi}}(\mathfrak{n})}{{\mathrm{N}\mathfrak{n}}^{s}}.
\]
Using Shahidi's non-vanishing result of $L(s,\pi\times\tilde{\pi})$ at $\operatorname{Re} s=1$ (see \cite{Shahidi-1981}), one has
\begin{equation}\label{PNT-RS}
\sum_{\mathrm{N}\mathfrak{n}\leq x}\Lambda_F(\mathfrak{n})a_{\pi \times \widetilde \pi}(\mathfrak{n})\sim x.
\end{equation}
Similar to \eqref{coeffpair}, we can prove that the corresponding inequality holds when we replace $a_\pi(\mathfrak{n})$ and $a_{\pi\times\tilde{\pi}}(\mathfrak{n})$  by $\lambda_\pi(\mathfrak{n})$ and $\lambda_{\pi\times\tilde{\pi}}(\mathfrak{n})$, respectively.

\begin{lemma}\label{ineq-2rs}
	With the above notation, we have $|\lambda_{\pi}(\mathfrak{n})|^2\leq \lambda_{\pi\times\tilde{\pi}}(\mathfrak{n})$ for all integral ideals $\mathfrak{n}$.
\end{lemma}
\begin{proof}
	See \cite[Lemma 3.1]{JLW-2020}.
\end{proof}

\begin{lemma}\label{ineq-lammucov}
	With the notation as above, let $b_{F,\pi}(\mathfrak{n})=\sum\limits_{\mathfrak{a}\mathfrak{b}=\mathfrak{n}}|\lambda_{\pi}(\mathfrak{a})\mu_\pi(\mathfrak{b})|$. If $(\mathfrak{n},\mathfrak{q}_\pi)=\cO_F$, then
	\begin{align}
	b_{F,\pi}(\mathfrak{n})^2 &\leq \underbrace{(\lambda_{\pi\times\tilde{\pi}}*\cdots*\lambda_{\pi\times\tilde{\pi}})}_{\text{$4(n+1)$ terms}}(\mathfrak{n}), \label{bfpi}\\
	|\mu_\pi(\mathfrak{n})|^2&\leq \lambda_{\pi\times\tilde{\pi}}(\mathfrak{n}),\label{mobius-Fpi}\\
	d_F(\mathfrak{n})|\mu_\pi(\mathfrak{n})|^2&\leq \underbrace{(\lambda_{\pi\times\tilde{\pi}}*\cdots*\lambda_{\pi\times\tilde{\pi}})}_{\text{$n+1$ terms}}(\mathfrak{n}) \label{mobius+divisor-Fpi},
	\end{align}
	where $\underbrace{(\lambda_{\pi\times\tilde{\pi}}*\cdots*\lambda_{\pi\times\tilde{\pi}})}_{\text{$k$ terms}}(\mathfrak{n})=\sum\limits_{\mathfrak{n}_1\cdots \mathfrak{n}_{k}=\mathfrak{n}}\lambda_{\pi\times\tilde{\pi}}(\mathfrak{n}_1)\cdots\lambda_{\pi\times\tilde{\pi}}(\mathfrak{n}_k)$, and $d_F(\mathfrak{n})$ is the divisor function on $F$.
\end{lemma}

\begin{proof}
	From the facts that $\lambda_{\pi}(\mathfrak{n})$, $\mu_\pi(\mathfrak{n})$ and $d_F(\mathfrak{n})$ are all multiplicative, we know that $b_{F,\pi}(\mathfrak{n})^2, |\mu_\pi(\mathfrak{n})|^2, d_F(\mathfrak{n})|\mu_\pi(\mathfrak{n})|^2$ are also multiplicative. It suffices to show the corresponding inequalities hold at  $\mathfrak{n}=\mathfrak{p}^k$ for any $k\geq 0$ and $\mathfrak{p} \nmid \mathfrak{q}_\pi.$
	
	(1) For the first inequality, we actually show a slightly stronger inequality at prime ideal powers as follows:	
	\begin{equation}\label{ineq-pk}
	b_{F,\pi}(\mathfrak{p}^k)^2\leq 4(n+1) \lambda_{\pi\times\tilde{\pi}}(\mathfrak{p}^k) \quad \text{for}\quad  \mathfrak{p} \nmid \mathfrak{q}_\pi.
	\end{equation}
	For a set $\{\alpha_{1},\ldots,\alpha_{n}\}$, we define the polynomial $e_{l}(\alpha_1,\ldots,\alpha_n)$ by
	$$
	e_{l}(\alpha_1,\ldots,\alpha_n)=\sum_{1 \leq j_{1}<\cdots<j_{l} \leq n} \alpha_{j_{1}} \alpha_{j_{2}} \cdots \alpha_{j_{l}}.
	$$
	The polynomial $e_{l}$ is called the $l$-th elementary symmetric polynomial.  If $l=0$, then $e_{l}(x_{1}, \ldots, x_{n})\equiv 1$. By convention, $e_{l}(\alpha_1,\ldots,\alpha_n)=0$ for $l>n$. By \eqref{mu-def}, we know that
	\begin{equation}\label{u-pi-e}
	\mu_\pi(\mathfrak{p}^l)=(-1)^l e_{l}(A_{\pi}(\mathfrak{p})).
	\end{equation}

	A partition $\lambda=(\lambda(i))_{i=1}^{\infty}$ is a sequence of nonincreasing nonnegative integers $\lambda(1)\geq\lambda(2)\geq\cdots$ with only finitely many nonzero entries.  For a partition $\lambda$, let $\ell(\lambda)$ be the number of nonzero $\lambda(i)$, and let $|\lambda|=\sum_{i=1}^{\infty} \lambda(i)$.  For a set $\{\alpha_{1},\ldots,\alpha_{n}\}$ and a partition $\lambda$ with $\ell(\lambda)\leq n$, let $s_{\lambda}(\alpha_1,\ldots,\alpha_n)$ be the Schur polynomial $\det[(\alpha_{i}^{\lambda(j)+n-j})_{ij}] / \det[(\alpha_{i}^{n-j})_{ij}]$ associated to $\lambda$.  If $|\lambda|=0$, then $s_{\lambda}(\alpha_1,\ldots,\alpha_n)\equiv 1$.  By convention, if $\ell(\lambda)>n$, then $s_{\lambda}(\alpha_1,\ldots,\alpha_n)\equiv 0$. Cauchy's identity \cite[Chapter 38]{Bump-2013}, tells us that
	\[
	L(s,\pi_{\mathfrak{p}})=\prod_{j=1}^n\Big(1-\frac{\alpha_{j,\pi}(\mathfrak{p})}{\mathrm{N}\mathfrak{p}^{s}}\Big)^{-1} = \sum_{k=0}^\infty \frac{s_{(k,0,0,\ldots)}(A_{\pi}(\mathfrak{p}))}{\mathrm{N}\mathfrak{p}^{ks}}
	\]
	and, for $\mathfrak{p} \nmid \mathfrak{q}_\pi$,
	\[
	L(s, \pi_{\mathfrak{p}} \times\tilde{\pi}_{\mathfrak{p}})=\prod_{j=1}^{n} \prod_{j'=1}^{n}\Big(1-\frac{\alpha_{j,\pi}(\mathfrak{p}) \bar{\alpha_{j',\pi}(\mathfrak{p})}}{\mathrm{N}\mathfrak{p}^{s}}\Big)^{-1} = \sum_{\lambda}\frac{s_{\lambda}(A_{\pi}(\mathfrak{p}))\bar{s_{\lambda}(A_{\pi}(\mathfrak{p}))}}{\mathrm{N}\mathfrak{p}^{s|\lambda|}},
	\]
	where the sum ranges over all partitions. Then we have
	\begin{equation}\label{lambda-schur}
	\lambda_{\pi}(\mathfrak{p}^k) = s_{(k,0,0,\ldots)}(A_{\pi}(\mathfrak{p})),\qquad \lambda_{\pi\times \tilde{\pi}}(\mathfrak{p}^k) = \sum_{ |\lambda|=k}|s_{\lambda}(A_{\pi}(\mathfrak{p}))|^2.
	\end{equation}
	
	Thus, by \eqref{u-pi-e} and \eqref{lambda-schur}, the dual Pieri rule \cite[Theorem 40.4]{Bump-2013} yields that
	\begin{equation*}
	\begin{aligned}
	b_{F,\pi}(\mathfrak{p}^k)&=\sum_{l=0}^{\min\{k,n\}}\big|e_{l}(A_{\pi}(\mathfrak{p}))s_{(k-l,0,0,\ldots)}(A_{\pi}(\mathfrak{p}))\big|\\
	\leq &\big|s_{(k,0,0,\ldots)}(A_{\pi}(\mathfrak{p}))\big|+\big|s_{({\tiny\underbrace{1,1,\ldots,1}_{\text{$k$
					terms}}},0,0,\ldots)}(A_{\pi}(\mathfrak{p}))\big|\\
	&+\sum_{l=1}^{\min\{k,n\}-1}(\big|s_{({\tiny\underbrace{k-l+1,1,\ldots,1}_{\text{$l$
					terms}}},0,0,\ldots)}(A_{\pi}(\mathfrak{p}))\big|+\big|s_{({\tiny\underbrace{k-l,1,\ldots,1}_{\text{$l+1$
					terms}}},0,0,\ldots)}(A_{\pi}(\mathfrak{p}))\big|).
	\end{aligned}
	\end{equation*}
	By the Cauchy--Schwarz inequality and \eqref{lambda-schur}, we then have
	\begin{align*}
	&|b_{F,\pi}(\mathfrak{p}^k)|^2\\
	&\leq (n+1)\Big(\big|s_{(k,0,0,\ldots)}(A_{\pi}(\mathfrak{p}))\big|^2+\big|s_{({\tiny\underbrace{1,1,\ldots,1}_{\text{$k$
					terms}}},0,0,\ldots)}(A_{\pi}(\mathfrak{p}))\big|^2\Big)\\
	&+2(n+1)\sum_{l=1}^{\min\{k,n\}-1}\Big(\big|s_{({\tiny\underbrace{k-l+1,1,\ldots,1}_{\text{$l$
					terms}}},0,0,\ldots)}(A_{\pi}(\mathfrak{p}))\big|^2+\big|s_{({\tiny\underbrace{k-l,1,\ldots,1}_{\text{$l+1$
					terms}}},0,0,\ldots)}(A_{\pi}(\mathfrak{p}))\big|^2\Big) \\
	&\leq 4(n+1)\lambda_{\pi\times \tilde{\pi}}(\mathfrak{p}^k).
	\end{align*}
	This completes the proof of \eqref{ineq-pk}, which further implies the inequality \eqref{bfpi}. 
	
	(2) For the second one, We see from \eqref{pi-euler} and \eqref{rs-ep} that 
	\begin{equation*}\label{dirichlet-log}
	\log L(s,\pi_{\mathfrak{p}})=\sum_{j=1}^n\sum_{k=1}^{\infty}\frac{\alpha_{j,\pi}(\mathfrak{p})^{k}}{k \mathrm{N}\mathfrak{p}^{ks}}=\sum_{k=1}^{\infty}\frac{a_{\pi}(\mathfrak{p}^{k})}{k \mathrm{N}\mathfrak{p}^{k s}}
	\end{equation*}
	and 
	\begin{equation*}
	\log L(s, \pi_{\mathfrak{p}} \times\tilde{\pi}_{\mathfrak{p}})=\sum_{1\leq j,j'\leq n }\sum_{k=1}^{\infty}\frac{(\alpha_{j,\pi}(\mathfrak{p})\overline{\alpha_{j',\pi}(\mathfrak{p})})^{k}}{k \mathrm{N}\mathfrak{p}^{k s}}=\sum_{k=1}^{\infty}\frac{|a_{\pi}(\mathfrak{p}^{k})|^2}{k \mathrm{N}\mathfrak{p}^{k s}}
	\end{equation*}
	for $\mathfrak{p} \nmid \mathfrak{q}_\pi$.
	Comparing these with \eqref{inv-lfun} and \eqref{pipi-eu-2}, we have 
	\begin{equation*}
	\exp\Big(-\sum_{k=1}^{\infty}\frac{a_{\pi}(\mathfrak{p}^{k})}{k \mathrm{N}\mathfrak{p}^{k s}}\Big)=\sum_{k=0}^{\infty}\frac{\mu_{\pi}(\mathfrak{p}^k)}{\mathrm{N}\mathfrak{p}^{ks}}
	\end{equation*}
	and
	\begin{equation*}
	\exp\Big(\sum_{k=1}^{\infty}\frac{|a_{\pi}(\mathfrak{p}^{k})|^2}{k \mathrm{N}\mathfrak{p}^{k s}}\Big)=\sum_{k=0}^{\infty}\frac{\lambda_{\pi \times \tilde{\pi}}(\mathfrak{p}^k)}{\mathrm{N}\mathfrak{p}^{ks}}.
	\end{equation*}
	By Lemma \ref{sound-auto}, we then get
	\begin{equation}\label{mu2-rs-prime}
	|\mu_\pi(\mathfrak{p}^{k})|^2 \leq \lambda_{\pi\times\tilde{\pi}}(\mathfrak{p}^{k})
	\end{equation}
	for any $k\geq 0$ and $\mathfrak{p} \nmid \mathfrak{q}_\pi$. By multiplicativity, the inequality \eqref{mobius-Fpi} follows.
	
	(3) By \eqref{mu-def} and \eqref{mu2-rs-prime}, we have 
	\[
	d_F(\mathfrak{p}^{k})|\mu_\pi(\mathfrak{p}^{k})|^2\leq (n+1)\lambda_{\pi\times\tilde{\pi}}(\mathfrak{p}^{k}) \quad \text{for any}\; k\geq 0 \;\text{and } \mathfrak{p} \nmid \mathfrak{q}_\pi.
	\]
	So the inequality \eqref{mobius+divisor-Fpi} follows.
\end{proof}

\begin{lemma}\label{ba2-sum}
	Let $b_{F,\pi}(\mathfrak{n})$ be defined as in Lemma \ref{ineq-lammucov}. Then we have
	\[
	\sum_{\mathrm{N}\mathfrak{n}\leq x}b_{F,\pi}(\mathfrak{n})^2\ll  x(\log x)^{4n+3},\qquad \sum_{\mathrm{N}\mathfrak{n}\leq x}	|\mu_\pi(\mathfrak{n})|^2\ll x, \qquad \sum_{\mathrm{N}\mathfrak{n}\leq x} d_F(\mathfrak{n})|\mu_\pi(\mathfrak{n})|^2\ll x(\log x)^{n}.
	\]
\end{lemma}
\begin{proof}
	In order to prove three upper bound estimates in a unified way, we introduce an  arithmetic function $f(\mathfrak{n})$ satisfying the following two conditions:
	\begin{enumerate}
		\item $0\leq f(\mathfrak{n})\ll \underbrace{(\lambda_{\pi\times\tilde{\pi}}*\cdots*\lambda_{\pi\times\tilde{\pi}})}_{\text{$k$ terms}}(\mathfrak{n})$ for all $(\mathfrak{n},\mathfrak{q}_\pi)=\cO_F$ and some $k\geq 1$;
		\item $f(\mathfrak{n})\ll \mathrm{N}\mathfrak{n}^{\delta}$ for any integral ideal $\mathfrak{n}$ and some $\delta< 1$.
	\end{enumerate}
	It is clear that the generating series of $\lambda_{\pi\times\tilde{\pi}}(\mathfrak{n})*\cdots*\lambda_{\pi\times\tilde{\pi}}(\mathfrak{n})$ is exactly $L(s,\pi\times\tilde{\pi})^{k}$. Using the analytic properties of the Rankin--Selberg $L$-function $L(s,\pi\times\tilde{\pi})$ and the Tauberian theorem, we find that there exists a constant $\Cl[abcon]{leading}=\Cr{leading}(\pi,k)>0$ such that
	\[
	\sum_{\mathrm{N}\mathfrak{n}\leq x} \underbrace{(\lambda_{\pi\times\tilde{\pi}}*\cdots*\lambda_{\pi\times\tilde{\pi}})}_{\text{$k$ terms}}(\mathfrak{n})\sim \Cr{leading}x(\log x)^{k-1}.
	\]
	Now, it follows that
	\[
	\sum_{\substack{\mathrm{N}\mathfrak{n} \leq x \\ (\mathfrak{n}, \mathfrak{q}_\pi)=\cO_{F}}} f(\mathfrak{n}) \ll x (\log x)^{k-1}.
	\]
	Moreover, for any $\mathfrak{n} \subset \cO_{F},$ one has
	$\mathfrak{n}=\mathfrak{n}_{1} \mathfrak{n}_{2}$ with $\mathfrak{n}_{1}|\mathfrak{q}_{\pi}^{\infty}$ 
	and $(\mathfrak{n}_{2}, \mathfrak{q}_{\pi})=\cO_{F}${, where $\mathfrak{n}_{1}|\mathfrak{q}_{\pi}^{\infty}$ means that $\mathfrak{p}|\mathfrak{n}_1$ implies $\mathfrak{p}|\mathfrak{q}_\pi$ for any $\mathfrak{p}$.} Thus, we have
	\begin{equation}\label{f-upperbound}
	\begin{aligned}
	\sum_{\mathrm{N}\mathfrak{n}\leq x}	f(\mathfrak{n})&=	\sum_{\substack{\mathrm{N}\mathfrak{n}_{1}\leq x \\ \mathfrak{n}_{1} \mid \mathfrak{q}_{\pi}^{\infty}}}	f(\mathfrak{n}_1) \sum_{\substack{\mathrm{N}\mathfrak{n}_{2}\leq x/\mathrm{N}\mathfrak{n}_{1} \\ (\mathfrak{n}_{2}, \mathfrak{q}_{\pi})=\cO_{F}}}	f(\mathfrak{n}_2)\\
	&\ll x (\log x)^{k-1} \prod_{\mathfrak{p}|\mathfrak{q}_{\pi}}\Big(1+O\Big(\frac{1}{\mathrm{N}\mathfrak{p}^{1-\delta}}\Big)\Big)\ll x (\log x)^{k-1}.
	\end{aligned}
	\end{equation}
	
	By \eqref{boundcoeff}, it is easy to see that $b_{F,\pi}(\mathfrak{n})^2, |\mu_\pi(\mathfrak{n})|^2$ and $d_F(\mathfrak{n})|\mu_\pi(\mathfrak{n})|^2$ are all $\ll \mathrm{N}\mathfrak{n}^{2\theta_{n}+\varepsilon}$. Note that $2\theta_{n}\leq 1-\frac{2}{n^{2}+1}$. Together this with Lemma \ref{ineq-lammucov}, we can take $f(\mathfrak{n})$ to be $b_{F,\pi}(\mathfrak{n})^2, |\mu_\pi(\mathfrak{n})|^2$ and $ d_F(\mathfrak{n})|\mu_\pi(\mathfrak{n})|^2$ with $k=4(n+1), 1$ and $n+1$, respectively. Thus, the estimate \eqref{f-upperbound} yields this lemma.
\end{proof}

\vskip 5mm

\section{Large-sieve type estimates}
In this section we obtain large sieve estimates for Dirichlet polynomials and $L$-functions. The main tool is the large sieve inequality for number fields introduced by Huxley \cite{Huxley-1970}.
\begin{lemma}\label{lem-largesieve}
	Let $c(\mathfrak{n})$ be any complex coefficients and define the Dirichlet polynomial to be
	$$D(s,\chi)=\sum_{\mathrm{N}\mathfrak{n}\leq x}\frac{c(\mathfrak{n})\chi(\mathfrak{n})}{{\mathrm{N}\mathfrak{n}}^s}.$$
Then we have
$$\sum_{\mathrm{N}\mathfrak{m}\leq Q} \frac{\mathrm{N}\mathfrak{m}}{\phi_F(\mathfrak{m})} \sideset{}{^*}{\sum}_{\chi\in\widehat{\mathrm{Cl}^{+}(\mathfrak{m})}}|D(s,\chi)|^2\ll (Q^2+x)\sum_{\mathrm{N}\mathfrak{n}\leq x}\Big|\frac{c(\mathfrak{n})}{\mathrm{N}\mathfrak{n}^s}\Big|^2, $$
where the second sum is over all primitive narrow ideal class character modulo $\mathfrak{m}$ and the implied constant depends on the number field.
\end{lemma}

To obtain large sieve inequality for $L$-functions, we use the approximation functional equation to approximate the $L$-functions in the critical strip by Dirichlet series. We state the approximate functional equation as in \cite[Theorem 5.3]{IK-2004}.
\begin{lemma}\label{lem 3.2}
	Let $X>0$ and $\pi\in \mathfrak{F}_n$. Then for $\Re s\in (0,1)$, we have
	$$
	L(s, \pi)=\sum_{\mathfrak{n}\subset \cO_F} \frac{\lambda_{\pi}(\mathfrak{n})}{{\mathrm{N}\mathfrak{n}}^s} V_{s}\Big(\frac{X\mathrm{N}\mathfrak{n}}{\sqrt{q(\pi)}}\Big)+\varepsilon(s,\pi)  \sum_{\mathfrak{n}\subset \cO_F} \frac{\bar{\lambda_{\pi}(\mathfrak{n})}}{{\mathrm{N}\mathfrak{n}}^{1-s}} V_{1-s}\Big(\frac{\mathrm{N}\mathfrak{n}}{X\sqrt{q(\pi)}}\Big),
	$$
	where
	$$\varepsilon(s,\pi)=\varepsilon(\pi)q(\pi)^{\frac{1}{2}-s}\frac{L_\infty(1-s,\tilde\pi)}{\L_\infty(s,\pi)}.$$
	Moreover, for any $A>0$, $V_{s}(y)$ is a function satisfying the following estimate
	$$V_{s}(y)\ll_{A}\Big(1+\frac{y}{\sqrt{q_{\infty}(\pi,\operatorname{Im} s)}}\Big)^{-A}.	$$
\end{lemma}

We shall use these two lemmata to deduce the second moments of twisted automorphic $L$-functions.

\begin{proposition}\label{prop 4.3}
	For any real number $t$ and $Q>0$, we have
	\[
	\sum_{\mathrm{N}\mathfrak{m}\leq Q} \sideset{}{^*}{\sum}_{\chi\in\widehat{\mathrm{Cl}^{+}(\mathfrak{m})}}|L(\tfrac{1}{2}+it,\pi\otimes\chi)|^2 \ll(Q^2+Q^{\frac{n}{2}}(3+|t|)^{\frac{n[F:\Q]}{2}})(\log Q(3+|t|))^2,
	\]
	where the second sum is over all primitive narrow ideal class character modulo $\mathfrak{m}$. The implied constant depends on $F$ and $\pi$.
\end{proposition}
\begin{proof}
By $q(\pi\otimes\chi)\sim Q_1$, we mean that $Q_1<q(\pi\otimes\chi) \leq 2Q_1$. Since $\chi$ is a character modulo $\mathfrak{m}$ and $\mathrm{N}\mathfrak{m}\leq Q$, it follows from \eqref{conductorpair} that $q(\pi\otimes\chi)\ll Q^n$. As a result,
\begin{align}\label{sec-1}
\sum_{\mathrm{N}\mathfrak{m}\leq Q} \sideset{}{^*}{\sum}_{\substack{\chi\in\widehat{\mathrm{Cl}^{+}(\mathfrak{m})} }}|L(\tfrac{1}{2}+it,\pi\otimes\chi)|^2
\ll (\log Q)\max_{Q_1\ll Q^n}\sum_{\mathrm{N}\mathfrak{m}\leq Q} \sideset{}{^*}{\sum}_{\substack{\chi\in\widehat{\mathrm{Cl}^{+}(\mathfrak{m})}\\ q(\pi\otimes\chi)\sim Q_1}}|L(\tfrac{1}{2}+it,\pi\otimes\chi)|^2 .
\end{align}
Since $\pi\otimes\chi \in \mathfrak{F}_n$ and $(\mathfrak{m},\mathfrak{q}_\pi)=\cO_F$, we obtain from \eqref{tw-lfunction-2} and Lemma \ref{lem 3.2} that
\begin{align*}
|L(\tfrac{1}{2}+it,\pi\otimes\chi)|^2
&\ll \Big|\sum_{\mathfrak{n}\subset \cO_F}\frac{\lambda_{\pi\otimes\chi}(\mathfrak{n})}{\mathrm{N}\mathfrak{n}^{1/2+it}}V_{\frac{1}{2}+it}\Big(\frac{\mathrm{N}\mathfrak{n}}{X\sqrt{q(\pi\otimes\chi)}}\Big) \Big|^2 \\
&+\Big|\sum_{\mathfrak{n}\subset \cO_F}\frac{\bar{\lambda_{\pi\otimes\chi}(\mathfrak{n})}}{\mathrm{N}\mathfrak{n}^{1/2-it}}V_{\frac{1}{2}-it}\Big(\frac{X\mathrm{N}\mathfrak{n}}{\sqrt{q(\pi\otimes\chi)}}\Big) \Big|^2 :=|D_1(X)|^2+|D_2(X)|^2
\end{align*}
for all $X>0$. As a result,
\begin{equation*}
|L(\tfrac{1}{2}+it,\pi\otimes\chi)|^2 \ll
\int_1^2|D_1(X)|^2\frac{dX}{X}+\int_1^2|D_2(X)|^2\frac{dX}{X}.
\end{equation*}
Note that the conductor $q(\pi\otimes\chi)\sim Q_1$. We denote $X_1=(Q_1/q(\pi\otimes\chi))^\frac{1}{2}$. We perform a change of variable $X\mapsto XX_1$ for the first integral, while $X\mapsto XX_1^{-1}$ for the second one. Consequently,
\begin{equation}\label{sec-2}
\begin{aligned}
|L(\tfrac{1}{2}+i t, \pi\otimes\chi)|^{2}
&\ll \int_{X_1}^{2X_1}\Big|\sum_{\mathfrak{n}\subset \cO_F}\frac{\lambda_{\pi\otimes\chi}(\mathfrak{n})}{\mathrm{N}\mathfrak{n}^{1/2+it}}V_{\frac{1}{2}+it}\bigg(\frac{\mathrm{N}\mathfrak{n}}{X\sqrt{Q_1}}\bigg)\Big|^{2}\frac{dX}{X} \\
&+\int_{X_1^{-1}}^{2X_1^{-1}}\Big|\sum_{\mathfrak{n}\subset \cO_F}\frac{\bar{\lambda_{\pi\otimes\chi}(\mathfrak{n})}}{\mathrm{N}\mathfrak{n}^{1/2-it}}V_{\frac{1}{2}-it}\bigg(\frac{X\mathrm{N}\mathfrak{n}}{\sqrt{Q_1}}\bigg)\Big|^{2}\frac{dX}{X}.
\end{aligned}
\end{equation}

For any $A>0$, we can see that $V_s(\frac{\mathrm{N}\mathfrak{n}}{X\sqrt{Q_1}})\ll \frac{(Q(3+|t|)^{[F:\Q]})^\frac{nA}{2}}{\mathrm{N}\mathfrak{n}^{A}}$ when $\mathrm{N}\mathfrak{n}>({Q(3+|t|)^{[F:\Q]}})^{\frac{n}{2}}$. For any sufficiently small $\varepsilon>0$, we can choose $A=A(\varepsilon)$ in Lemma \ref{lem 3.2} to such that
$$\sum_{\mathrm{N}\mathfrak{n}>({Q(3+|t|)^{[F:\Q]}})^{\frac{n}{2}+\varepsilon}}\frac{\lambda_{\pi\otimes\chi}(\mathfrak{n})}{\mathrm{N}\mathfrak{n}^{1/2+it}}V_{\frac{1}{2}+it}\bigg(\frac{\mathrm{N}\mathfrak{n}}{X\sqrt{Q_1}}\bigg)\ll 1.$$
When $\mathrm{N}\mathfrak{n}\leq ({Q(3+|t|)^{[F:\Q]}})^{\frac{n}{2}+\varepsilon}$, we shall make use of the large sieve inequality. For this, we first apply \eqref{boundcoeff} and the relation \eqref{tw-coeff-deco}, and then get
\begin{multline*}
\sum_{\mathrm{N}\mathfrak{n}\leq ({Q(3+|t|)^{[F:\Q]}})^{\frac{n}{2}+\varepsilon}}\frac{\lambda_{\pi\otimes\chi}(\mathfrak{n})}{\mathrm{N}\mathfrak{n}^{1/2+it}}V_{\frac{1}{2}+it}\bigg(\frac{\mathrm{N}\mathfrak{n}}{X\sqrt{Q_1}}\bigg)\\
=	\sum_{\substack{\mathrm{N}\mathfrak{n}\leq ({Q(3+|t|)^{[F:\Q]}})^{\frac{n}{2}+\varepsilon}\\ (\mathfrak{n},\mathfrak{q}_\pi)=\cO_F}}\frac{\lambda_{\pi}(\mathfrak{n})\chi(\mathfrak{n})}{\mathrm{N}\mathfrak{n}^{1/2+it}}V_{\frac{1}{2}+it}\bigg(\frac{\mathrm{N}\mathfrak{n}}{X\sqrt{Q_1}}\bigg)+O(N\mathfrak{q}_\pi^\varepsilon).
\end{multline*}

Accordingly, the contribution from the first term on the right-hand side of \eqref{sec-2} to the double sum on the right-hand side of \eqref{sec-1} is
\begin{multline}
\label{sed-onesum}
\sum_{\mathrm{N}\mathfrak{m}\leq Q} \sideset{}{^*}{\sum}_{\substack{\chi\in\widehat{\mathrm{Cl}^{+}(\mathfrak{m})}\\ q(\pi\otimes\chi)\sim Q_1}}\int_{X_1}^{2X_1}\Big|\sum_{\mathfrak{n}\subset \cO_F}\frac{\lambda_{\pi\otimes\chi}(\mathfrak{n})}{\mathrm{N}\mathfrak{n}^{1/2+it}}V_{\frac{1}{2}+it}\bigg(\frac{\mathrm{N}\mathfrak{n}}{X\sqrt{Q_1}}\bigg)\Big|^{2}\frac{dX}{X} \\
\ll  \sum_{\mathrm{N}\mathfrak{m}\leq Q} \sideset{}{^*}{\sum}_{\substack{\chi\in\widehat{\mathrm{Cl}^{+}(\mathfrak{m})}\\ q(\pi\otimes\chi)\sim Q_1}}\int_{X	_1}^{2X_1}\Big|\sum_{\substack{\mathrm{N}\mathfrak{n}\leq({Q(3+|t|)^{[F:\Q]}})^{\frac{n}{2}+\varepsilon}\\ (\mathfrak{n},\mathfrak{q}_\pi)=\cO_F}}\frac{\lambda_{\pi}(\mathfrak{n})\chi(\mathfrak{n})}{\mathrm{N}\mathfrak{n}^{1/2+it}}V_{\frac{1}{2}+it}\bigg(\frac{\mathrm{N}\mathfrak{n}}{X\sqrt{Q_1}}\bigg)\Big|^{2}\frac{dX}{X}+Q^2 \\
\ll  Q^2+\int_{X_1}^{2X_1}\sum_{\mathrm{N}\mathfrak{m}\leq Q} \sideset{}{^*}{\sum}_{\substack{\chi\in\widehat{\mathrm{Cl}^{+}(\mathfrak{m})}\\ q(\pi\otimes\chi)\sim Q_1}}\Big|\sum_{\substack{\mathrm{N}\mathfrak{n}\leq({Q(3+|t|)^{[F:\Q]}})^{\frac{n}{2}}\\ (\mathfrak{n},\mathfrak{q}_\pi)=\cO_F}}\frac{\lambda_{\pi}(\mathfrak{n})\chi(\mathfrak{n})}{\mathrm{N}\mathfrak{n}^{1/2+it}}V_{\frac{1}{2}+it}\bigg(\frac{\mathrm{N}\mathfrak{n}}{X\sqrt{Q_1}}\bigg)\Big|^{2}\frac{dX}{X} \\
+(\log QT)\max_{\frac{n}{2}\leq \frac{\log M}{\log(Q(3+|t|)^{[F:\Q]})}\leq \frac{n}{2}+\varepsilon}\int_{X_1}^{2X_1}\sum_{\mathrm{N}\mathfrak{m}\leq Q} \sideset{}{^*}{\sum}_{\substack{\chi\in\widehat{\mathrm{Cl}^{+}(\mathfrak{m})}\\ q(\pi\otimes\chi)\sim Q_1}}\Big|\sum_{\substack{\mathrm{N}\mathfrak{n}\sim M\\ (\mathfrak{n},\mathfrak{q}_\pi)=\cO_F}}\frac{\lambda_{\pi}(\mathfrak{n})\chi(\mathfrak{n})}{\mathrm{N}\mathfrak{n}^{1/2+it}}V_{\frac{1}{2}+it}\bigg(\frac{\mathrm{N}\mathfrak{n}}{X\sqrt{Q_1}}\bigg)\Big|^{2}\frac{dX}{X}.
\end{multline}

Moreover,  $V_{\frac{1}{2}+it}\big(\frac{\mathrm{N}\mathfrak{n}}{X\sqrt{Q_1}}\big)$ is bounded by $O(1)$ when $\mathrm{N}\mathfrak{n}\leq({Q(3+|t|)^{[F:\Q]}})^{\frac{n}{2}}$.
Thus, by Lemma \ref{lem-largesieve}, Lemma \ref{ineq-2rs} and \eqref{asymRS}, we have
\begin{equation}\label{sec-3}
\begin{aligned}
&\sum_{\mathrm{N}\mathfrak{m}\leq Q} \sideset{}{^*}{\sum}_{\substack{\chi\in\widehat{\mathrm{Cl}^{+}(\mathfrak{m})}\\ q(\pi\otimes\chi)\sim Q_1}}\Big|\sum_{\substack{\mathrm{N}\mathfrak{n}\leq({Q(3+|t|)^{[F:\Q]}})^{\frac{n}{2}}\\ (\mathfrak{n},\mathfrak{q}_\pi)=\cO_F}}\frac{\lambda_{\pi}(\mathfrak{n})\chi(\mathfrak{n})}{\mathrm{N}\mathfrak{n}^{1/2+it}}V_{\frac{1}{2}+it}\bigg(\frac{\mathrm{N}\mathfrak{n}}{X\sqrt{Q_1}}\bigg)\Big|^{2}\\
&\ll  (Q^2+Q^{\frac{n}{2}}(3+|t|)^{\frac{n[F:\Q]}{2}})\log Q(3+|t|).
\end{aligned}
\end{equation}
For $({Q(3+|t|)^{[F:\Q]}})^{\frac{n}{2}}\leq \mathrm{N}\mathfrak{n} \leq ({Q(3+|t|)^{[F:\Q]}})^{\frac{n}{2}+\varepsilon}$, we take $A=1$ in Lemma \ref{lem 3.2} to have $V_{\frac{1}{2}+it}(\frac{\mathrm{N}\mathfrak{n}}{X\sqrt{Q_1}})\ll \frac{({Q(3+|t|)^{[F:\Q]}})^{\frac{n}{2}}}{\mathrm{N}\mathfrak{n}}$.
Similar to the estimate \eqref{sec-3}, we get
\begin{equation}\label{sec-4}
\begin{aligned}
&\sum_{\mathrm{N}\mathfrak{m}\leq Q} \sideset{}{^*}{\sum}_{\substack{\chi\in\widehat{\mathrm{Cl}^{+}(\mathfrak{m})}\\ q(\pi\otimes\chi)\sim Q_1}}\Big|\sum_{\substack{\mathrm{N}\mathfrak{n}\sim M\\ (\mathfrak{n},\mathfrak{q}_\pi)=\cO_F}}\frac{\lambda_{\pi}(\mathfrak{n})\chi(\mathfrak{n})}{\mathrm{N}\mathfrak{n}^{1/2+it}}V_{\frac{1}{2}+it}\bigg(\frac{\mathrm{N}\mathfrak{n}}{X\sqrt{Q_1}}\bigg)\Big|^{2} \\
&\ll  (Q^2+M)Q^nT^{n[F:\Q]}\sum_{\mathrm{N}\mathfrak{n}\sim M}\frac{|\lambda_\pi(\mathfrak{n})|^2}{{\mathrm{N}\mathfrak{n}^3}}\\
&\ll (Q^2+M)\frac{({Q(3+|t|)^{[F:\Q]}})^{n}}{M^2}.
\end{aligned}
\end{equation}
Inserting \eqref{sec-3} and \eqref{sec-4} into \eqref{sed-onesum}, we have
\[
\sum_{\mathrm{N}\mathfrak{m}\leq Q} \sideset{}{^*}{\sum}_{\substack{\chi\in\widehat{\mathrm{Cl}^{+}(\mathfrak{m})}\\ q(\pi\otimes\chi)\sim Q_1}}\int_{X_1}^{2X_1}\Big|\sum_{\mathfrak{n}\subset \cO_F}\frac{\lambda_{\pi\otimes\chi}(\mathfrak{n})}{\mathrm{N}\mathfrak{n}^{1/2+it}}V_{\frac{1}{2}+it}\bigg(\frac{\mathrm{N}\mathfrak{n}}{X\sqrt{Q_1}}\bigg)\Big|^{2}\frac{dX}{X}	
\ll (Q^2+Q^{\frac{n}{2}}(3+|t|)^{\frac{n[F:\Q]}{2}})\log Q(3+|t|).
\]
We could treat the dual sum similarly and derive the contribution from the first term on the right-hand side of \eqref{sec-2} is also bounded by $O((Q^2+Q^{\frac{n}{2}}(3+|t|)^{\frac{n[F:\Q]}{2}})\log Q(3+|t|))$ Then this proposition follows.
\end{proof}

\vskip 5mm

\section{Type I sums:  Proof of Theorem \ref{thm-bv-integers}}

For technical convenience, one usually works with the weighted sum 
\begin{equation}\label{Ax-def}
	\psi_{\rho}(y,\pi,\mathfrak{m},\mathfrak{a}) :=\sum_{\substack{\mathrm{N}\mathfrak{n}\leq y \\ \mathfrak{n}\equiv \mathfrak{a} \text{ in } \mathrm{Cl}^{+}(\mathfrak{m})}}\lambda_\pi(\mathfrak{n}) \Big(1-\frac{\mathrm{N}\mathfrak{n}}{y}\Big)^{\rho},
\end{equation}
where $(\mathfrak{m},\mathfrak{a})=\cO_F$, $\rho\geq 0$.
 We want to prove a Bombieri--Vinogradov theorem with the smooth weight for $\lambda_\pi(\mathfrak{n})$. To be precise, we have
\begin{lemma}\label{lem-ap-smooth}
	Let $\eta=\max\{2,\frac{n}{2}\}$, $\rho=\lfloor\frac{n[F:\Q]}{4}\rfloor+1$, where $\lfloor\cdot\rfloor$ denotes the usual floor function.  If $A$ is any positive number, then we have
	$$\sum_{\mathrm{N}\mathfrak{m}\leq Q}\frac{h(\mathfrak{m})}{\phi_F(\mathfrak{m})}\max_{(\mathfrak{a},\mathfrak{m})=\cO_F}\max_{y\leq x}\Big|\psi_{\rho}(y,\pi,\mathfrak{m},\mathfrak{a}) \Big|\ll \frac{x}{(\log x)^A}, $$
	where $Q=x^{\frac{1}{\eta}}(\log x)^{-B}$ with $B=A+2n+3$.
\end{lemma}

\begin{proof}
Detecting the congruence condition in \eqref{Ax-def} by the multiplicative characters
$\chi \in \widehat{\mathrm{Cl}^{+}(\mathfrak{m})},$ we obtain the identity
\begin{equation}\label{smoothed-1}
\sum_{\substack{\mathrm{N}\mathfrak{n}\leq y \\ \mathfrak{n}\equiv \mathfrak{a} \text{ in } \mathrm{Cl}^{+}(\mathfrak{m})}}\lambda_\pi(\mathfrak{n}) \Big(1-\frac{\mathrm{N}\mathfrak{n}}{y}\Big)^{\rho}=\frac{1}{h(\mathfrak{m})}\sum_{\chi \in \widehat{\mathrm{Cl}^{+}(\mathfrak{m})} }\overline{\chi}(\mathfrak{a})\sum_{\mathrm{N}\mathfrak{n}\leq y}\lambda_\pi(\mathfrak{n}) \chi(\mathfrak{n})\Big(1-\frac{\mathrm{N}\mathfrak{n}}{y}\Big)^{\rho}.
\end{equation}
We  shall treat the innermost sum on the right-hand side of \eqref{smoothed-1} by the technique of standard contour integration, which could give a direct link between the summation associated to an arithmetic function and the corresponding Dirichlet series.  If $\rho$ is any positive integer and $c>0,$ then we have the Mellin inversion formula
\begin{equation}\label{perron}
\frac{1}{2\pi i} \int_{(c)}\frac{x^s}{s(s+1)\cdots (s+\rho)}ds=
\left\{
\begin{array}{ll}
\frac{1}{\rho!}(1-\frac{1}{x})^\rho  & \hbox{if $x \geq 1$},\\

0 & \hbox{if $0 \leq x \leq 1$.}
\end{array}
\right.
\end{equation}
Then it follows from \eqref{perron} that
\[
\sum_{\mathrm{N}\mathfrak{n}\leq y}\lambda_\pi(\mathfrak{n}) \chi(\mathfrak{n})\Big(1-\frac{\mathrm{N}\mathfrak{n}}{y}\Big)^{\rho}=\frac{1}{2\pi i}\int_{(1+\varepsilon)}\frac{\Gamma(s)}{\Gamma(\rho+1+s)}	\Big(\sum_{\mathfrak{n}\subset \cO_F}\frac{\lambda_{\pi}(\mathfrak{n})\chi(\mathfrak{n})}{{\mathrm{N}\mathfrak{n}}^s}\Big) y^{s}ds.
\]
If $\chi$ is induced by a primitive character $\chi_1 (\text{mod}\; \mathfrak{m}^\prime),$ then $\mathfrak{m}^\prime|\mathfrak{m}$. We further get from \eqref{tw-lfunction-2} that
\begin{multline}\label{DS-L}
\sum_{\mathfrak{n}\subset \cO_F}\frac{\lambda_{\pi}(\mathfrak{n})\chi(\mathfrak{n})}{{\mathrm{N}\mathfrak{n}}^s}=\Big(\prod_{\mathfrak{p}}\prod_{j=1}^n\Big(1-\frac{\alpha_{j,\pi}(\mathfrak{p})\chi_1(\mathfrak{p})}{\N\mathfrak{p}^{s}}\Big)^{-1}\Big)\prod_{\mathfrak{p}|\mathfrak{m}}\prod_{j=1}^n\Big(1-\frac{\alpha_{j,\pi}(\mathfrak{p})\chi_1(\mathfrak{p})}{\N\mathfrak{p}^{s}}\Big)\\
=L(s, \pi\otimes \chi_1)\Big(\prod_{\mathfrak{p}|\mathfrak{m}^\prime}\prod_{j=1}^{n}\Big(1-\frac{\alpha_{j,\pi\otimes\chi}(\mathfrak{p}) }{\N\mathfrak{p}^{s}}\Big)\Big)\prod_{\mathfrak{p}|\mathfrak{m}}\prod_{j=1}^n\Big(1-\frac{\alpha_{j,\pi}(\mathfrak{p})\chi_1(\mathfrak{p})}{\N\mathfrak{p}^{s}}\Big).
\end{multline}
Due to the estimate \eqref{boundcoeff}, for any $\varepsilon>0,$
$$\Big(\prod_{\mathfrak{p}|\mathfrak{m}^\prime}\prod_{j=1}^{n}\Big(1-\frac{\alpha_{j,\pi\otimes\chi}(\mathfrak{p}) }{\N\mathfrak{p}^{s}}\Big)\Big)\prod_{\mathfrak{p}|\mathfrak{m}}\prod_{j=1}^n\Big(1-\frac{\alpha_{j,\pi}(\mathfrak{p})\chi_1(\mathfrak{p})}{\N\mathfrak{p}^{s}}\Big)\ll d_F(\mathfrak{m})^{2n}$$
at $\Re s=1/2$.
Applying the analytic properties of $L(s, \pi \otimes \chi_1)$ and then moving the line of integration to $\Re(s)=1/2$. Thus, by the residue theorem, it is bounded by
\[
d_F(\mathfrak{m})^{2n} y^{\frac{1}{2}}\int_{(1/2)}|L(s, \pi \otimes \chi_1)|\frac{|ds|}{|s|^{\rho+1}}.
\]
Gathering these estimates, we then have
\begin{equation}\label{mean-rho}
\begin{aligned}
&\sum_{\mathrm{N}\mathfrak{m}\leq Q}\frac{h(\mathfrak{m})}{\phi_F(\mathfrak{m})}\max_{(\mathfrak{a},\mathfrak{m})=\cO_F}\max_{y\leq x}\Big|\sum_{\substack{\mathrm{N}\mathfrak{n}\leq y \\ \mathfrak{n}\equiv \mathfrak{a} \text{ in } \mathrm{Cl}^{+}(\mathfrak{m})}}\lambda_\pi(\mathfrak{n}) \Big(1-\frac{\mathrm{N}\mathfrak{n}}{y}\Big)^{\rho}\Big|\\
&\ll x^{\frac{1}{2}}\int_{(1/2)}\sum_{\mathrm{N}\mathfrak{m}\leq Q}\frac{d_F(\mathfrak{m})^{2n}}{\phi_F(\mathfrak{m})}\sum_{\mathfrak{m^\prime}|\mathfrak{m}} \;\sideset{}{^*}{\sum}_{\chi_1 \in \widehat{\mathrm{Cl}^{+}(\mathfrak{m^\prime})} }	|L(s, \pi \otimes \chi_1)|\frac{|ds|}{|s|^{\rho+1}}\\
&\ll x^{\frac{1}{2}}\int_{(1/2)}\sum_{\mathrm{N}\mathfrak{r}\leq Q}\frac{d_F(\mathfrak{r})^{2n}}{\phi_F(\mathfrak{r})}\sum_{\N\mathfrak{m^\prime}\leq Q/\mathrm{N}\mathfrak{r}} \frac{d_F(\mathfrak{m^\prime})^{2n}}{\phi_F(\mathfrak{m^\prime})}\;\sideset{}{^*}{\sum}_{\chi_1 \in \widehat{\mathrm{Cl}^{+}(\mathfrak{m^\prime})} }	|L(s, \pi \otimes \chi_1)|\frac{|ds|}{|s|^{\rho+1}},
\end{aligned}
\end{equation}	
where the trivial inequalities $d_F(\mathfrak{bc})\leq d_F(\mathfrak{b})d_F(\mathfrak{c})$ and $\phi_F(\mathfrak{bc})\geq \phi_F(\mathfrak{b})\phi_F(\mathfrak{c})$. are used in the last step.
By Proposition \ref{prop 4.3}, the Cauchy--Schwarz inequality and the elementary estimate $\sum_{\N\mathfrak{n}\leq x} d_F(\mathfrak{n})^k\ll x(\log x)^{2^k-1}$ for any positive integer $k$, we have
\[
\begin{aligned}
&\sum_{\N\mathfrak{m^\prime}\leq Q/\mathrm{N}\mathfrak{r}} \frac{d_F(\mathfrak{m^\prime})^{2n}}{\phi_F(\mathfrak{m^\prime})}\sideset{}{^*}{\sum}_{\chi_1 \in \widehat{\mathrm{Cl}^{+}(\mathfrak{m^\prime})} }	|L(s, \pi \otimes \chi_1)|  \\
&\ll(\log Q)^2 \max_{R\leq Q/\mathrm{N}\mathfrak{r}} \frac{1}{R} \sum_{\mathrm{N}\mathfrak{m^\prime}\sim R}d_F(\mathfrak{m^\prime})^{2n}\sideset{}{^*}{\sum}_{\chi \in \widehat{\mathrm{Cl}^{+}(\mathfrak{m})} }	|L(s, \pi \otimes \chi)|\\
&\ll(\log Q)^{2n+2}\max_{R\leq Q/\mathrm{N}\mathfrak{r}} \Big(\sum_{\mathrm{N}\mathfrak{m}\sim R}\sum_{\chi \in \widehat{\mathrm{Cl}^{+}(\mathfrak{m})} }	|L(s, \pi \otimes \chi)|^2\Big)^{\frac{1}{2}} \\
&\ll\Big(\frac{Q}{\mathrm{N}\mathfrak{r}}+\big(\frac{Q}{\mathrm{N}\mathfrak{r}}\big)^{\frac{n}{4}}\Big)|s|^{\frac{n[F:\Q]}{4}}(\log Q|s|)^{2n+3}.
\end{aligned}
\]	
The lemma follows once we insert this estimate into \eqref{mean-rho}.
\end{proof}

\begin{proof}[Proof of Theorem \ref{thm-bv-integers}]
Now we turn to proof of Theorem \ref{thm-bv-integers}. To do this, we follow the method of \cite[Lemma 2]{Motohashi-1980}, which is originally due to \cite{Gallagher-1968}.
Let $z=z(y)=\frac{y}{(\log y)^{A/2}}$ with the parameter $A$ as in Lemma \ref{lem-ap-smooth}. It is easy to check that
\begin{align}\label{int-short-z}
\int_{y-z}^{y}t^{\rho-1}\psi_{\rho-1}(t,\pi,\mathfrak{m},\mathfrak{a})dt
=\frac{y^\rho}{\rho}\psi_{\rho}(y,\pi,\mathfrak{m},\mathfrak{a})-\frac{(y-z)^\rho}{\rho}\psi_{\rho}(y-z,\pi,\mathfrak{m},\mathfrak{a}).
\end{align}
We can rewrite this integral on the left-hand side of \eqref{int-short-z} as
\begin{equation}\label{int-short-z2}
zy^{\rho-1}\psi_{\rho-1}(y,\pi,\mathfrak{m},\mathfrak{a})-\int_{y-z}^y\Big(y^{\rho-1}\psi_{\rho-1}(y,\pi,\mathfrak{m},\mathfrak{a})-t^{\rho-1}\psi_{\rho-1}(t,\pi,\mathfrak{m},\mathfrak{a})\Big)dt.
\end{equation}
It is obvious that the integrand of \eqref{int-short-z2} equals to
\begin{equation}\label{intfun2}
\begin{aligned}
&\sum_{\substack{\mathrm{N}\mathfrak{n}\leq t \\ \mathfrak{n}\equiv\mathfrak{a} \text{ in } \mathrm{Cl}^{+}(\mathfrak{m})}}\lambda_{\pi}(\mathfrak{n})\Big((y-\mathrm{N}\mathfrak{n})^{\rho-1}-(t-\mathrm{N}\mathfrak{n})^{\rho-1}\Big)+\sum_{\substack{t<\mathrm{N}\mathfrak{n}\leq y \\ \mathfrak{n}\equiv\mathfrak{a} \text{ in } \mathrm{Cl}^{+}(\mathfrak{m})}}	\lambda_\pi(\mathfrak{n})(y-\mathrm{N}\mathfrak{n})^{\rho-1} \\
&\ll zy^{\rho-2}\sum_{\substack{\mathrm{N}\mathfrak{n}\leq y \\ \mathfrak{n}\equiv\mathfrak{a} \text{ in } \mathrm{Cl}^{+}(\mathfrak{m})}}|\lambda_\pi(\mathfrak{n})|
\end{aligned}
\end{equation}
for $y-z\leq t\leq y$. Then it follows from \eqref{int-short-z}--\eqref{intfun2} that
\begin{align*}
\psi_{\rho-1}(y,\pi,\mathfrak{m},\mathfrak{a})=&\frac{y}{z\rho}\psi_{\rho}(y,\pi,\mathfrak{m},\mathfrak{a})-\frac{(y-z)^\rho}{zy^{\rho-1}\rho}\psi_{\rho}(y-z,\pi,\mathfrak{m},\mathfrak{a}) 
+O\Big(\frac{z}{y}\sum_{\substack{\mathrm{N}\mathfrak{n}\leq y \\ \mathfrak{n}\equiv\mathfrak{a} \text{ in } \mathrm{Cl}^{+}(\mathfrak{m})}}|\lambda_\pi(\mathfrak{n})|\Big).
\end{align*}
As a result, we can estimate the sum involving $\psi_{\rho-1}(y,\pi,\mathfrak{m},\mathfrak{a})$ as follows
\begin{equation}\label{bv-mm1}
	\begin{aligned}		
		&\sum_{\mathrm{N}\mathfrak{m}\leq Q} \frac{h(\mathfrak{m})}{\phi_F(\mathfrak{m})}\max_{(\mathfrak{a},\mathfrak{m})=\cO_F}\max_{y\leq x}\Big|\psi_{\rho-1}(y,\pi,\mathfrak{m},\mathfrak{a})\Big| \\
		&\ll (\log x)^{\frac{A}{2}}\sum_{\mathrm{N}\mathfrak{m}\leq Q} \frac{h(\mathfrak{m})}{\phi_F(\mathfrak{m})}\max_{(\mathfrak{a},\mathfrak{m})=\cO_F}\max_{y\leq x}\Big|\psi_{\rho}(y,\pi,\mathfrak{m},\mathfrak{a})\Big|\\
	&	+ \sum_{\mathrm{N}\mathfrak{m}\leq Q}\frac{h(\mathfrak{m})}{\phi_F(\mathfrak{m})} \max_{(\mathfrak{a},\mathfrak{m})=\cO_F}\max_{y\leq x}\,(\log y)^{-\frac{A}{2}}\sum_{\substack{\mathrm{N}\mathfrak{n}\leq y \\ \mathfrak{n}\equiv\mathfrak{a} \text{ in } \mathrm{Cl}^{+}(\mathfrak{m})}}|\lambda_\pi(\mathfrak{n})|.
	\end{aligned}
\end{equation}
Note that the contribution of the term involving $\psi_{\rho}(y-z,\pi,\mathfrak{m},\mathfrak{a}) $ can be dominated by the first term of the right-hand side of \eqref{bv-mm1}, since we have taken the maximum over $y\leq x$.

We first treat the last term on the right-hand side of \eqref{bv-mm1}. By Lemma \ref{ineq-2rs} and the Rankin--Selberg theory, we can get $\sum_{\substack{\mathrm{N}\mathfrak{n}\leq x }}|\lambda_{\pi}(\mathfrak{n})|\leq x^\frac{1}{2}\sum_{\substack{\mathrm{N}\mathfrak{n}\leq x }}|\lambda_{\pi}(\mathfrak{n})|^2\ll x$. As a result,
\begin{equation}\label{LS-ABS}
\begin{aligned}
&\sum_{\mathrm{N}\mathfrak{m}\leq Q}\frac{h(\mathfrak{m})}{\phi_F(\mathfrak{m})} \max_{(\mathfrak{a},\mathfrak{m})=\cO_F}\max_{y\leq x}\,(\log y)^{-\frac{A}{2}}\sum_{\substack{\mathrm{N}\mathfrak{n}\leq y \\ \mathfrak{n}\equiv\mathfrak{a} \text{ in } \mathrm{Cl}^{+}(\mathfrak{m})}}|\lambda_\pi(\mathfrak{n})| \\
\ll &\sum_{\mathrm{N}\mathfrak{m}\leq Q}\frac{h(\mathfrak{m})}{\phi_F(\mathfrak{m})} \max_{(\mathfrak{a},\mathfrak{m})=\cO_F}\max_{x^{\frac{1}{3}}\leq y\leq x}\,(\log y)^{-\frac{A}{2}}\sum_{\substack{\mathrm{N}\mathfrak{n}\leq y \\ \mathfrak{n}\equiv\mathfrak{a} \text{ in } \mathrm{Cl}^{+}(\mathfrak{m})}}|\lambda_\pi(\mathfrak{n})|+x^\frac{5}{6} \\
\ll &(\log x)^{-\frac{A}{2}}\sum_{\mathrm{N}\mathfrak{m}\leq Q} \frac{h(\mathfrak{m})}{\phi_F(\mathfrak{m})}\max_{(\mathfrak{a},\mathfrak{m})=\cO_F}\sum_{\substack{\mathrm{N}\mathfrak{n}\leq x \\ \mathfrak{n}\equiv\mathfrak{a} \text{ in } \mathrm{Cl}^{+}(\mathfrak{m})}}|\lambda_{\pi}(\mathfrak{n})|+x^\frac{5}{6}
\end{aligned}
\end{equation}
since $Q\ll X^\frac{1}{2}$. Using the Cauchy--Schwarz inequality and Lemma \ref{lem-largesieve},  we obtain
\begin{equation*}
\begin{aligned}
&\sum_{\mathrm{N}\mathfrak{m}\leq Q} \frac{h(\mathfrak{m})}{\phi_F(\mathfrak{m})}\max_{(\mathfrak{a},\mathfrak{m})=\cO_F}\sum_{\substack{\mathrm{N}\mathfrak{n}\leq x \\ \mathfrak{n}\equiv\mathfrak{a} \text{ in } \mathrm{Cl}^{+}(\mathfrak{m})}}|\lambda_{\pi}(\mathfrak{n})|\\
&\ll (\log Q)^4\max_{CR\leq Q}\frac{1}{CR} \sum_{\mathrm{N}\mathfrak{r}\sim R}\sum_{\mathrm{N}\mathfrak{c}\sim C}\sideset{}{^*}{\sum}_{\chi\in\widehat{\mathrm{Cl}^{+}(\mathfrak{c})}}\Big|\sum_{\substack{\mathrm{N}\mathfrak{n}\leq x \\(\mathfrak{n},\mathfrak{cr})=\cO_F}}|\lambda_{\pi}(\mathfrak{n})|\chi(\mathfrak{n})\Big| \\
&\ll (\log Q)^4\max_{CR\leq Q}\frac{1}{R} \sum_{\mathrm{N}\mathfrak{r}\sim R}\Big(\sum_{\mathrm{N}\mathfrak{c}\sim C}\sideset{}{^*}{\sum}_{\chi\in\widehat{\mathrm{Cl}^{+}(\mathfrak{c})}}\Big|\sum_{\substack{\mathrm{N}\mathfrak{n}\leq x \\(\mathfrak{n},\mathfrak{cr})=\cO_F}}|\lambda_{\pi}(\mathfrak{n})|\chi(\mathfrak{n})\Big|^2 \Big)^{\frac{1}{2}}\\
&\ll (\log Q)^4 (Q+x^{\frac{1}{2}})x^{\frac{1}{2}}\ll x(\log x)^4.
\end{aligned}
\end{equation*}
Substituting this into \eqref{LS-ABS}, we bound the contribution of the last term on the right-hand side of \eqref{bv-mm1} by $\ll x(\log x)^{4-\frac{A}{2}}$.

By Lemma \ref{lem-ap-smooth}, we see that the first term in the right-hand side of \eqref{bv-mm1} is $O(x(\log x)^{-\frac{A}{2}})$. Thus, we conclude from the above that
\begin{align*}
\sum_{\mathrm{N}\mathfrak{m}\leq Q}\frac{h(\mathfrak{m})}{\phi_F(\mathfrak{m})} \max_{(\mathfrak{a},\mathfrak{m})=\cO_F}\max_{y\leq X}\Big|\psi_{\rho-1}(y,\pi,\mathfrak{m},\mathfrak{a})\Big|\ll \frac{x}{(\log x)^{\frac{A}{2}-4}}
\end{align*}
for $Q=X^\frac{1}{\eta}(\log x)^{-B}$ where $\eta=\max\{2,\frac{n}{2}\}, B=A+2n+3$ and the implied constant depends on $\pi$,$F$ and $A$. Repeating this process $\rho$ times yields
\begin{equation*}
\sum_{\mathrm{N}\mathfrak{m}\leq Q}\frac{h(\mathfrak{m})}{\phi_F(\mathfrak{m})} \max_{(\mathfrak{a},\mathfrak{m})=\cO_F}\max_{y\leq X}\Big|\psi_{0}(y,\pi,\mathfrak{m},\mathfrak{a})\Big|\ll \frac{X}{(\log X)^{\frac{A+8}{2^{\rho}}-8} }.
\end{equation*}
Finally, we transform parameter $\frac{A+8}{2^{\rho}}-8$ to $A$, finishing the proof.
\end{proof}


\vskip 5mm

\section{A bilinear form}

\begin{lemma}\label{lem-Vaughan}
	Let $\mathcal{A}=\{a(\mathfrak{l})\}$ and $\mathcal{B}=\{b(\mathfrak{n})\}$ be two sequences of complex numbers with $\mathcal{B}$ satisfying the following Siegel--Walfisz hypothesis: For $(\mathfrak{a},\mathfrak{m})=\cO_F$ and $\N \mathfrak{m}\leq (\log N)^{A}$ for any $A>0$,
	\begin{equation}\label{Siegel-W-condition}
	\sum_{\substack{\mathrm{N}\mathfrak{n}\leq N \\ \mathfrak{n}\equiv\mathfrak{a} \text{ in } \mathrm{Cl}^{+}(\mathfrak{m})}}b(\mathfrak{n})\ll \frac{(\mathcal{B}(N)N)^{\frac{1}{2}}}{(\log N)^{9A}}.
	\end{equation}
	Then we have
	\begin{equation}\label{eq-bilinear}
	\begin{aligned}
	&\sum_{\mathrm{N}\mathfrak{m}\leq Q}\frac{h(\mathfrak{m})}{\phi_F(\mathfrak{m})} \max_{(\mathfrak{a},\mathfrak{m})=\cO_F}\Big| {\underset{\mathfrak{ln}\equiv\mathfrak{a} \text{ in } \mathrm{Cl}^{+}(\mathfrak{m})}{\sum_{\mathrm{N}\mathfrak{l}\leq L}\sum_{\mathrm{N}\mathfrak{n}\leq N}}}a(\mathfrak{l})b(\mathfrak{n}) \Big|\\
	&\ll	\Big(Q+\sqrt{L+N}+\frac{\sqrt{L N}}{(\log N)^{A}}\Big) (\log Q)^2\mathcal{A}(L)^{\frac{1}{2}}\mathcal{B}(N)^{\frac{1}{2}},
	\end{aligned}
	\end{equation}
	where
	$$
	\mathcal{A}(L)=\sum_{\mathrm{N}\mathfrak{l}\leq L}|a(\mathfrak{l})|^{2},\quad \mathcal{B}(N)=\sum_{\mathrm{N}\mathfrak{n}\leq N}|b(\mathfrak{n})|^{2},
	$$
	and the implied constant depends only on $A$.

\end{lemma}

\begin{proof}  Using the orthogonality of multiplicative characters, we see that the left-hand side in \eqref{eq-bilinear} is bounded by
\begin{equation}\label{eq-conv}
\begin{aligned}
\sum_{\mathrm{N}\mathfrak{m}\leq Q} &\frac{1}{\phi_F(\mathfrak{m})}\sum_{\chi \in \widehat{\mathrm{Cl}^{+}(\mathfrak{m})} } \Big|\sum_{\mathrm{N}\mathfrak{l}\leq L} a(\mathfrak{l})\chi(\mathfrak{l})\Big|  \Big|\sum_{\mathrm{N}\mathfrak{n}\leq N} b(\mathfrak{n})\chi(\mathfrak{n})\Big|\\
&\leq \sum_{\mathrm{N}\mathfrak{cr} \leq Q} \frac{1}{\phi_F(\mathfrak{c})\phi_F(\mathfrak{r})} \sideset{}{^*}{\sum}_{\chi\in\widehat{\mathrm{Cl}^{+}(\mathfrak{r})}} \Big|\sum_{\substack{\mathrm{N}\mathfrak{l}\leq L\\ (\mathfrak{l},\mathfrak{c})=\cO_F}} a(\mathfrak{l})\chi(\mathfrak{l})\Big|  \Big|\sum_{\substack{\mathrm{N}\mathfrak{n}\leq N \\ (\mathfrak{n},\mathfrak{c})=\cO_F}} b(\mathfrak{n})\chi(\mathfrak{n})\Big|.
\end{aligned}
\end{equation}
Fix $\mathfrak{c}$ and split the sum over $\mathfrak{r}$ into dyadic segments $\mathrm{N}\mathfrak{r} \sim R_{1}$. Then apply the Cauchy--Schwarz inequality to the dyadic segment to get
$$
\leq \Big(\sum_{\mathrm{N}\mathfrak{r} \sim R_{1}} \frac{1}{\phi_F(\mathfrak{r})} \sideset{}{^*}{\sum}_{\chi\in\widehat{\mathrm{Cl}^{+}(\mathfrak{r})}} \Big|\sum_{\substack{\mathrm{N}\mathfrak{l}\leq L\\ (\mathfrak{l},\mathfrak{c})=\cO_F}} a(\mathfrak{l})\chi(\mathfrak{l})\Big|^{2}\Big)^{\frac{1}{2}}\Big(\sum_{\mathrm{N}\mathfrak{r} \sim R_{1}} \frac{1}{\phi_F(\mathfrak{r})} \sideset{}{^*}{\sum}_{\chi\in\widehat{\mathrm{Cl}^{+}(\mathfrak{r})}} \Big|\sum_{\substack{\mathrm{N}\mathfrak{n}\leq N \\ (\mathfrak{n},\mathfrak{c})=\cO_F}} b(\mathfrak{n})\chi(\mathfrak{n})\Big|^{2}\Big)^{\frac{1}{2}}
$$
By Lemma \ref{lem-largesieve}, the above is
\begin{equation*}
\begin{aligned}
\ll\Big(R_{1}+\frac{L}{R_{1}}\Big)^{\frac{1}{2}}\Big(R_{1}+\frac{N}{R_{1}}\Big)^{\frac{1}{2}}\Big(\sum_{\mathrm{N}\mathfrak{l}\leq L}|a(\mathfrak{l})|^{2}\Big)^{\frac{1}{2}}\Big(\sum_{\mathrm{N}\mathfrak{n}\leq N}|b(\mathfrak{n})|^{2}\Big)^{\frac{1}{2}} \\
\ll\Big(R_{1}+\sqrt{L+N}+\frac{\sqrt{L N}}{R_{1}}\Big)\Big(\sum_{\mathrm{N}\mathfrak{l}\leq L}|a(\mathfrak{l})|^{2}\Big)^{\frac{1}{2}}\Big(\sum_{\mathrm{N}\mathfrak{n}\leq N}|b(\mathfrak{n})|^{2}\Big)^{\frac{1}{2}}.
\end{aligned}
\end{equation*}
Summing this over $R \leq R_{1} \leq Q /\mathrm{N}\mathfrak{c}$, we get
\begin{equation}\label{sum-R1-Large}
\Big(\frac{Q}{\mathrm{N}\mathfrak{c}}+\sqrt{L+N} \log Q+\frac{\sqrt{L N}}{R}\Big)\Big(\sum_{\mathrm{N}\mathfrak{l}\leq L}|a(\mathfrak{l})|^{2}\Big)^{\frac{1}{2}}\Big(\sum_{\mathrm{N}\mathfrak{n}\leq N}|b(\mathfrak{n})|^{2}\Big)^{\frac{1}{2}}.
\end{equation}
It remains to estimate the contribution of the primitive characters $\chi\in\widehat{\mathrm{Cl}^{+}(\mathfrak{r})}$ with $1\leq \mathrm{N}\mathfrak{r} \leq R.$ For each of these we appeal to the condition \eqref{Siegel-W-condition}.

First, we define M\"obius function $\mu_{F}(\mathfrak{a})$ for $F$ by
\[
\mu_{F}(\mathfrak{a})=\left\{\begin{array}{ll}
1 & \mathfrak{a}=\cO_F, \\
(-1)^{r} & \mathfrak{a} \text { is a product of } r \text { distinct prime ideals, } \\
0 & \mathfrak{a} \text { is divided by square of a prime ideal. }
\end{array}\right.
\]
It is easy to see that $\mu_{F}(\mathfrak{a})$ is a multiplicative function over the ideals which satisfies
\begin{equation*}
\sum_{\mathfrak{o}|\mathfrak{a}}\mu_{F}(\mathfrak{o})=\left\{\begin{array}{ll}
1 & \mathfrak{a}=\cO_F, \\
0 & \text {otherwise.}
\end{array}\right.
\end{equation*}
We detect the coprimality condition $(\mathfrak{n},\mathfrak{c})=\cO_F$ by this formula and get
\begin{equation*}
\begin{aligned}
\sum_{\substack{\mathrm{N}\mathfrak{n}\leq N \\ (\mathfrak{n},\mathfrak{c})=\cO_F}} b(\mathfrak{n})\chi(\mathfrak{n}) &=\sum_{\mathfrak{o}|\mathfrak{c}}\mu_{F}(\mathfrak{o})	\sum_{\substack{\mathrm{N}\mathfrak{n}\leq N \\ \mathfrak{o}|\mathfrak{n}}} b(\mathfrak{n})\chi(\mathfrak{n}) \\
&=\sum_{\substack{\mathfrak{o}|\mathfrak{c}\\ \mathrm{N}\mathfrak{o}\leq K}}\mu_{F}(\mathfrak{o})	 \sum_{\mathfrak{l}|\mathfrak{o}}\mu_{F}(\mathfrak{l})	\sum_{\substack{\mathrm{N}\mathfrak{n}\leq N \\ (\mathfrak{n},\mathfrak{l})=\cO_F}} b(\mathfrak{n})\chi(\mathfrak{n})+\sum_{\substack{\mathfrak{o}|\mathfrak{c}\\ \mathrm{N}\mathfrak{o}>K }}\mu_{F}(\mathfrak{o})	\sum_{\substack{\mathrm{N}\mathfrak{n}\leq N \\ \mathfrak{o}|\mathfrak{n}}} b(\mathfrak{n})\chi(\mathfrak{n})\\
&:=W_1+W_2.
\end{aligned}
\end{equation*}
where $K$ will be chosen later. Next we estimate the innermost sum in $W_1$ by splitting into classes in $\mathrm{Cl}^{+}(\mathfrak{m})$, and for each class we apply the Siegel--Walfisz hypothesis \eqref{Siegel-W-condition}. Consequently, we have that
\[
\sum_{\substack{\mathrm{N}\mathfrak{n}\leq N \\ (\mathfrak{n},\mathfrak{l})=\cO_F}} b(\mathfrak{n})\chi(\mathfrak{n})=\sum_{\mathfrak{h}\in \mathrm{Cl}^{+}(\mathfrak{rl})} \chi(\mathfrak{h})	\sum_{\substack{\mathrm{N}\mathfrak{n}\leq N \\ \mathfrak{n}\equiv\mathfrak{h} \text{ in } \mathrm{Cl}^{+}(\mathfrak{rl})}}b(\mathfrak{n})\ll \frac{ h(\mathfrak{rl}) N^{\frac{1}{2}}}{(\log N)^{9A}}\Big(\sum_{\mathrm{N}\mathfrak{n}\leq N}|b(\mathfrak{n})|^{2}\Big)^{\frac{1}{2}},
\]
where we choose $\mathfrak{h}$ to be integral representatives. This further yields
\[
W_1\ll \frac{ N\mathfrak{r} N^{\frac{1}{2}}}{(\log N)^{A}}\Big(\sum_{\mathrm{N}\mathfrak{n}\leq N}|b(\mathfrak{n})|^{2}\Big)^{\frac{1}{2}}\sum_{\substack{\mathfrak{o}|\mathfrak{c}\\ \mathrm{N}\mathfrak{o}\leq K}}	 \sum_{\mathfrak{l}|\mathfrak{o}}|\mu_{F}(\mathfrak{l})|N\mathfrak{l}\ll\frac{ RK N^{\frac{1}{2}}d_F(\mathfrak{c})}{(\log N)^{9A}}\Big(\sum_{\mathrm{N}\mathfrak{n}\leq N}|b(\mathfrak{n})|^{2}\Big)^{\frac{1}{2}}.
\]
Then we treat the sum $W_2$ and apply the Cauchy--Schwarz inequality getting
\[
W_2\ll \Big(\frac{N}{K}\Big)^{\frac{1}{2}} \Big(\sum_{\mathrm{N}\mathfrak{n}\leq N}|b(\mathfrak{n})|^{2}\Big)^{\frac{1}{2}}d_F(\mathfrak{c}).
\]
Adding these two estimates for $W_1$ and $W_2$ and then choosing $K=(\log N)^{6A}$, we obtain
\begin{equation}\label{sum-b-copime}
\sum_{\substack{\mathrm{N}\mathfrak{n}\leq N \\ (\mathfrak{n},\mathfrak{c})=\cO_F}} b(\mathfrak{n})\chi(\mathfrak{n})\ll \frac{ R N^{\frac{1}{2}}d_F(\mathfrak{c})}{(\log N)^{3A}}\Big(\sum_{\mathrm{N}\mathfrak{n}\leq N}|b(\mathfrak{n})|^{2}\Big)^{\frac{1}{2}}.
\end{equation}
Moreover, we use the trivial bound
\[
\sum_{\substack{\mathrm{N}\mathfrak{l}\leq L\\ (\mathfrak{l},\mathfrak{c})=\cO_F}} a(\mathfrak{l})\chi(\mathfrak{l})\ll L^{\frac{1}{2}}\Big(\sum_{\mathrm{N}\mathfrak{l}\leq L}|a(\mathfrak{l})|^{2}\Big)^{\frac{1}{2}}.
\]
Combining this with the estimate \eqref{sum-b-copime} together, we obtain
\begin{equation}\label{sum-R1-small}
\begin{aligned}
&\sum_{\mathrm{N}\mathfrak{r} \leq R} \frac{1}{\phi_F(\mathfrak{r})} \sideset{}{^*}{\sum}_{\chi\in\widehat{\mathrm{Cl}^{+}(\mathfrak{r})}} \Big|\sum_{\substack{\mathrm{N}\mathfrak{l}\leq L\\ (\mathfrak{l},\mathfrak{c})=\cO_F}} a(\mathfrak{l})\chi(\mathfrak{l})\Big|  \Big|\sum_{\substack{\mathrm{N}\mathfrak{n}\leq N \\ (\mathfrak{n},\mathfrak{c})=\cO_F}} b(\mathfrak{n})\chi(\mathfrak{n})\Big|\\
&\ll  \frac{R^2\sqrt{L N}d_F(\mathfrak{c})}{(\log N)^{3A}} \Big(\sum_{\mathrm{N}\mathfrak{l}\leq L}|a(\mathfrak{l})|^{2}\Big)^{\frac{1}{2}}\Big(\sum_{\mathrm{N}\mathfrak{n}\leq N}|b(\mathfrak{n})|^{2}\Big)^{\frac{1}{2}}.
\end{aligned}
\end{equation}
Summing \eqref{sum-R1-small} and \eqref{sum-R1-Large} over $\mathfrak{c}$, we infer the following bound for our original sum in \eqref{eq-conv}:
$$
\Big(Q+\sqrt{L+N} +\frac{\sqrt{L N}}{R}+\frac{R^2\sqrt{L N}}{(\log N)^{3A}}\Big) (\log Q)^2\Big(\sum_{\mathrm{N}\mathfrak{l}\leq L}|a(\mathfrak{l})|^{2}\Big)^{\frac{1}{2}}\Big(\sum_{\mathrm{N}\mathfrak{n}\leq N}|b(\mathfrak{n})|^{2}\Big)^{\frac{1}{2}}
$$
Choosing $R=(\log N)^{A}$, we complete the proof of the theorem.
\end{proof}

\begin{lemma}\label{max-lem-Vaughan}
With the notation and conditions as in Lemma \ref{lem-Vaughan}, we have
\begin{equation*}
\begin{aligned}
&\sum_{\mathrm{N}\mathfrak{m}\leq Q}\frac{h(\mathfrak{m})}{\phi_F(\mathfrak{m})} \max_{(\mathfrak{a},\mathfrak{m})=\cO_F} \max_{y\leq x} \Big| {\underset{\mathrm{N}\mathfrak{ln}\leq y\atop \mathfrak{ln}\equiv\mathfrak{a} \text{ in } \mathrm{Cl}^{+}(\mathfrak{m})}{\sum_{\mathrm{N}\mathfrak{l}\leq L}\sum_{\mathrm{N}\mathfrak{n}\leq N}}}a(\mathfrak{l})b(\mathfrak{n}) \Big|\\
&\ll	\Big(Q+\sqrt{L+N}+\frac{\sqrt{L N}}{(\log N)^{A}}\Big) (\log xLN)(\log Q)^2\mathcal{A}(L)^{\frac{1}{2}}\mathcal{B}(N)^{\frac{1}{2}}.
\end{aligned}
\end{equation*}
\end{lemma}
\begin{proof}
We follow the trick of Vaughan \cite[Lemma 2]{Vaughan-1980}. If $\gamma>0$, then
\[
\frac{1}{\pi}\int_{-\infty}^{\infty} e^{i \beta \alpha} \frac{\sin \gamma \alpha}{\alpha} d \alpha=\begin{cases}
1&\mbox{if $0\leq\beta<\gamma$,}\\
0&\mbox{if $\beta>\gamma$}	
\end{cases}
=:\delta(\beta)
\]
via the product-to-sum identities of trigonometric functions and  the identity $e^{i \beta \alpha}=\cos\beta \alpha+i\sin\beta \alpha$.  By integration by parts, we get
$$
\int_{A}^{\infty} \frac{\sin Y \alpha}{\alpha} d \alpha \ll \frac{1}{Y A}
$$
for any $Y>0$ and $A>0.$ Thus, using the product-to-sum identities of trigonometric functions again, we have
\begin{equation}\label{delta-beta}
\delta(\beta)=\frac{1}{\pi}\int_{-A}^{A} e^{i \beta \alpha} \frac{\sin \gamma \alpha}{\alpha} d \alpha+O\Big(\frac{1}{A|\gamma-\beta|}\Big).
\end{equation}
Taking $\gamma=\log (\lfloor y\rfloor+\frac{1}{2})$ and $\beta=\log \mathrm{N}\mathfrak{ln}$, we obtain the equality
$$
\delta(\log \mathrm{N}\mathfrak{ln})=\left\{\begin{array}{ll}
1 & \mathrm{N}\mathfrak{ln} \leq y, \\
0 & \mathrm{N}\mathfrak{ln}>y.
\end{array}\right.
$$
Then it follows from \eqref{delta-beta} that
$$
\delta(\log \mathrm{N}\mathfrak{ln})=\frac{1}{\pi}\int_{-A}^{A}(\mathrm{N}\mathfrak{ln})^{i \alpha} \frac{\sin \gamma \alpha}{\alpha} d \alpha+O\Big(\frac{y}{A}\Big).
$$
Hence
\begin{multline*}
{\underset{\mathrm{N}\mathfrak{ln}\leq y\atop \mathfrak{ln}\equiv\mathfrak{a} \text{ in } \mathrm{Cl}^{+}(\mathfrak{m})}{\sum_{\mathrm{N}\mathfrak{l}\leq L}\sum_{\mathrm{N}\mathfrak{n}\leq N}}}a(\mathfrak{l})b(\mathfrak{n}) ={\underset{\mathfrak{ln}\equiv\mathfrak{a} \text{ in } \mathrm{Cl}^{+}(\mathfrak{m})}{\sum_{\mathrm{N}\mathfrak{l}\leq L}\sum_{\mathrm{N}\mathfrak{n}\leq N}}}a(\mathfrak{l})b(\mathfrak{n})	\delta(\log \mathrm{N}\mathfrak{ln})  \\
=\frac{1}{\pi}\int_{-A}^{A} {\underset{\mathfrak{ln}\equiv\mathfrak{a} \text{ in } \mathrm{Cl}^{+}(\mathfrak{m})}{\sum_{\mathrm{N}\mathfrak{l}\leq L}\sum_{\mathrm{N}\mathfrak{n}\leq N}}}a(\mathfrak{l})(\mathrm{N}\mathfrak{l})^{i \alpha}b(\mathfrak{n})(\mathrm{N}\mathfrak{n})^{i \alpha}	\frac{\sin \gamma \alpha}{\alpha} d \alpha+O\Big(\frac{y}{A} \sum_{\mathrm{N}\mathfrak{l}\leq L}\sum_{\mathrm{N}\mathfrak{n}\leq N}|a(\mathfrak{l})b(\mathfrak{n})|\Big) .
\end{multline*}
The error term here is manageable if we take $A=xLN$. For $y \leq x$, the integral is
$$
\ll \int_{-A}^{A}\Big| {\underset{\mathfrak{ln}\equiv\mathfrak{a} \text{ in } \mathrm{Cl}^{+}(\mathfrak{m})}{\sum_{\mathrm{N}\mathfrak{l}\leq L}\sum_{\mathrm{N}\mathfrak{n}\leq N}}}a(\mathfrak{l})(\mathrm{N}\mathfrak{l})^{i \alpha}b(\mathfrak{n})(\mathrm{N}\mathfrak{n})^{i \alpha}\Big| \min \Big(\log x, \frac{1}{|\alpha|}\Big) d \alpha.
$$
Summing the integrand over modulus $\mathfrak{m}$ and applying Lemma \ref{lem-Vaughan}, we have 
\begin{equation*}
\begin{aligned}
&\sum_{\mathrm{N}\mathfrak{m}\leq Q}\frac{h(\mathfrak{m})}{\phi_F(\mathfrak{m})} \max_{(\mathfrak{a},\mathfrak{m})=\cO_F} \max_{y\leq x} \Big| {\underset{ \mathfrak{ln}\equiv\mathfrak{a} \text{ in } \mathrm{Cl}^{+}(\mathfrak{m})}{\sum_{\mathrm{N}\mathfrak{l}\leq L}\sum_{\mathrm{N}\mathfrak{n}\leq N}}}a(\mathfrak{l})(\mathrm{N}\mathfrak{l})^{i \alpha}b(\mathfrak{n})(\mathrm{N}\mathfrak{n})^{i \alpha}\Big| \\
&\ll	\Big(Q+\sqrt{L+N}+\frac{\sqrt{L N}}{(\log N)^{A}}\Big) (\log Q)^2\mathcal{A}(L)^{\frac{1}{2}}\mathcal{B}(N)^{\frac{1}{2}}.
\end{aligned}
\end{equation*}
Since $\int_{-A}^{A} \min (\log x,|\alpha|^{-1}) d \alpha \ll \log xLN$ trivially, the lemma follows.
\end{proof}


\vskip 5mm

\section{Proof of Theorem \ref{thm:prime_power_version}}

It remains to estimate the averages
\[
\sum_{\mathrm{N}\mathfrak{m}\leq Q} \frac{h(\mathfrak{m})}{\phi_F(\mathfrak{m})}\max_{(\mathfrak{a},\mathfrak{m})=\cO_F}\max_{y\leq x}|S_i|
\]
for $i=1,2,3,4$, where the terms $S_i$ are given by (3.4)-(3.7).  To begin, we define
\[
\alpha_{F,\pi}(\mathfrak{n})=\sum_{\substack{\mathfrak{bc=n} \\ \mathrm{N}\mathfrak{b}\leq X,~\mathrm{N}\mathfrak{c}\leq X}}\mu_\pi(\mathfrak{b})\Lambda_F(\mathfrak{c})a_\pi(\mathfrak{c}), \quad \beta_{F,\pi}(\mathfrak{n})=\sum_{\substack{\mathfrak{bd=n}\\ \mathfrak{b}> X}}\mu_\pi(\mathfrak{b})\lambda_\pi(\mathfrak{d}).
\]
By inequalities in Section \ref{sec-inequality}, we can estimate the second moments of $\Lambda_F(\mathfrak{n})a_\pi(\mathfrak{n}), \alpha_{F,\pi}(\mathfrak{n})$ and $\beta_{F,\pi}(\mathfrak{n})$. Firstly, we get from the  Cauchy--Schwarz inequality and \eqref{PNT-RS} that
\begin{equation}\label{square-lambda}
\begin{aligned}
\sum_{\mathrm{N}\mathfrak{n}\leq x}|\Lambda_F(\mathfrak{n})a_\pi(\mathfrak{n})|^2 &\ll\Big(\sum_{\mathrm{N}\mathfrak{n}\leq x}\Lambda_F(\mathfrak{n})\Big)\Big(\sum_{\mathrm{N}\mathfrak{n}\leq x}\Lambda_F(\mathfrak{n})a_{\pi\otimes \tilde\pi}(\mathfrak{n})\Big)\ll x.
\end{aligned}
\end{equation}
Secondly, by the bound \eqref{boundcoeff} and the definitions of $\Lambda_F(\mathfrak{n})$ and $a_{\pi\otimes \tilde\pi}(\mathfrak{n})$, we have
\begin{equation*}
\begin{aligned}
\sum_{\mathrm{N}\mathfrak{n}\leq x}d_F(\mathfrak{n})\Lambda_F(\mathfrak{n})a_{\pi\otimes \tilde\pi}(\mathfrak{n}) &= \sum_{k\ll \log x} \sum_{\mathrm{N}\mathfrak{p}^k\leq x}(k+1) a_{\pi\otimes \tilde\pi}(\mathfrak{p}^k)\log \mathrm{N}\mathfrak{p}\\
&= \sum_{k\leq n^2+1} \sum_{\mathrm{N}\mathfrak{p}^k\leq x}(k+1) a_{\pi\otimes \tilde\pi}(\mathfrak{p}^k)\log \mathrm{N}\mathfrak{p}+x^{1-\frac{1}{n^2+1}+\varepsilon}\\
&\ll \sum_{\mathrm{N}\mathfrak{n}\leq x}\Lambda_F(\mathfrak{n})a_{\pi\otimes \tilde\pi}(\mathfrak{n})  +x^{1-\frac{1}{n^2+1}+\varepsilon}\ll x.
\end{aligned}
\end{equation*}
This further implies that
\begin{equation}\label{square-alpha}
\begin{aligned}
\sum_{\mathrm{N}\mathfrak{n}\leq x}|\alpha_{F,\pi}(\mathfrak{n})|^2 &\ll \sum_{\mathrm{N}\mathfrak{n}\leq x}\sum_{\mathfrak{bc=n}}|\mu_\pi(\mathfrak{b})\Lambda_F(\mathfrak{c})a_\pi(\mathfrak{c})|^2 d_F(\mathfrak{n})\\
&\ll (\log x)\sum_{\mathrm{N}\mathfrak{b}\leq x}d_F(\mathfrak{b})|\mu_\pi(\mathfrak{b})|^2\sum_{\mathrm{N}\mathfrak{c}\leq x/\mathrm{N}\mathfrak{b}}d_F(\mathfrak{c})\Lambda_F(\mathfrak{c})a_{\pi\otimes \tilde\pi}(\mathfrak{c}) \ll x(\log x)^{n+2},
\end{aligned}
\end{equation}
by inserting Lemma \ref{ba2-sum}, partial summation and the trivial fact $d_F(\mathfrak{bc})\leq d_F(\mathfrak{b})d_F(\mathfrak{c})$.
Thirdly, we immediately obtain from Lemma \ref{ba2-sum} that
\begin{equation}\label{square-beta}
\begin{aligned}
\sum_{\mathrm{N}\mathfrak{n}\leq x}|\beta_{F,\pi}(\mathfrak{n})|^2 \ll x(\log x)^{4n+3}.
\end{aligned}
\end{equation}

\subsection{The average for $S_1$}

To estimate the contribution from $S_1,$ we use the Cauchy--Schwarz inequality and \eqref{square-lambda} to obtain
\begin{equation*}\label{est-S1}
\sum_{\mathrm{N}\mathfrak{m}\leq Q} \frac{h(\mathfrak{m})}{\phi_F(\mathfrak{m})}\max_{(\mathfrak{a},\mathfrak{m})=\cO_F} \max_{y\leq x}|S_1| \ll QX^{\frac{1}{2}} \Big(\sum_{\mathrm{N}\mathfrak{n}\leq X}|\Lambda_F(\mathfrak{n})a_\pi(\mathfrak{n})|^2\Big)^{\frac{1}{2}}\ll QX.
\end{equation*}

\subsection{The average for $S_2$}
We split $S_2$ in the following way:
\[
S_2 = \Big(\sum_{\mathrm{N}\mathfrak{b}\leq H}+\sum_{H<\mathrm{N}\mathfrak{b}\leq X} \Big)   \mu_\pi(\mathfrak{b}) \sum_{\substack{\mathrm{N}\mathfrak{c}\leq y/\mathrm{N}\mathfrak{b}\\ \mathfrak{c}\equiv\mathfrak{ab^{-1}} \text{ in } \mathrm{Cl}^{+}(\mathfrak{m})}}\lambda_\pi(\mathfrak{c})\log \mathrm{N}\mathfrak{c}=:S_2^{'}+S_2^{''},
\]
where $H<X$.  We deduce from partial summation and Theorem \ref{thm-bv-integers} that
\begin{equation*}\label{S2prime}
\begin{aligned}
&\sum_{\mathrm{N}\mathfrak{m}\leq Q} \frac{h(\mathfrak{m})}{\phi_F(\mathfrak{m})}\max_{(\mathfrak{a},\mathfrak{m})=\cO_F}\max_{y\leq x} |S_2^{'}|\\
&\leq  \sum_{\mathrm{N}\mathfrak{b}\leq H} |\mu_\pi(\mathfrak{b})| \sum_{\mathrm{N}\mathfrak{m}\leq Q}
\frac{h(\mathfrak{m})}{\phi_F(\mathfrak{m})}\max_{(\mathfrak{ab},\mathfrak{m})=\cO_F}\max_{y\leq x}\Big|  \sum_{\substack{\mathrm{N}\mathfrak{c}\leq y/\mathrm{N}\mathfrak{b}\\ \mathfrak{c}\equiv\mathfrak{ab^{-1}} \text{ in } \mathrm{Cl}^{+}(\mathfrak{m})}}\lambda_\pi(\mathfrak{c})\log \mathrm{N}\mathfrak{c}\Big|\\
&\ll  \log x\Big(\sum_{\mathrm{N}\mathfrak{b}\leq H} |\mu_\pi(\mathfrak{b})|\Big)\sum_{\mathrm{N}\mathfrak{m}\leq Q}
\frac{h(\mathfrak{m})}{\phi_F(\mathfrak{m})}\max_{(\mathfrak{a},\mathfrak{m})=\cO_F}\max_{y\leq x}\Big|  \sum_{\substack{\mathrm{N}\mathfrak{c}\leq y\\ \mathfrak{c}\equiv\mathfrak{a} \text{ in } \mathrm{Cl}^{+}(\mathfrak{m})}}\lambda_\pi(\mathfrak{c})\Big|\ll  \frac{Hx}{(\log x)^A}.
\end{aligned}
\end{equation*}
We use the Cauchy--Schwarz inequality and Lemma \ref{ba2-sum} in the last step, and emphasize that $Q\leq x^{\frac{1}{\eta}}(\log x)^{-B}$ with $\eta=\max\{\frac{n}{2},2\}$, $B=2^{\frac{n[F:\Q]}{4}}(2A+16)+2n-4$.
In order to estimate the contribution of $S_2^{''},$ splitting the range of summation over $\mathfrak{b}$ into intervals of the form $\mathrm{N}\mathfrak{b}\sim M$ with $H<M\leq X$, we get
\begin{equation}\label{s2-doube}
\begin{aligned}
& \sum_{\mathrm{N}\mathfrak{m}\leq Q} \frac{h(\mathfrak{m})}{\phi_F(\mathfrak{m})}\max_{(\mathfrak{a},\mathfrak{m})=\cO_F}\max_{y\leq x} |S_2^{''}| \\
&\ll (\log x)\max_{H\leq M\leq X}  \sum_{\mathrm{N}\mathfrak{m}\leq Q} \frac{h(\mathfrak{m})}{\phi_F(\mathfrak{m})}\max_{(\mathfrak{a},\mathfrak{m})=\cO_F}\max_{y\leq x}  \Big| {\underset{\mathrm{N}\mathfrak{bc}\leq y\atop \mathfrak{bc}\equiv\mathfrak{a} \text{ in } \mathrm{Cl}^{+}(\mathfrak{m})}{\sum_{\mathrm{N}\mathfrak{b}\sim M}\sum_{\mathrm{N}\mathfrak{c}\leq x/M}}}\mu_\pi(\mathfrak{b})\lambda_\pi(\mathfrak{c})\log \mathrm{N}\mathfrak{c}\Big|.
\end{aligned}
\end{equation}

In order to apply Lemma \ref{max-lem-Vaughan},  we need to verify the sequence $\{\lambda_\pi(\mathfrak{n})\log \mathrm{N}\mathfrak{n}\}$ satisfies the Siegel--Walfisz hypothesis. So we have to estimate the sum of type $\sum_{\mathrm{N}\mathfrak{n} \leq x } \lambda_\pi(\mathfrak{n})\chi(\mathfrak{n})$ for any $\chi \in \widehat{\mathrm{Cl}^{+}(\mathfrak{m})}$. We choose a function $h$ supported on $[0,x+Y]$, such that $h(z)=1$ if $Y\leq z \leq x$ and $h^{(j)}(x) \ll_j Y^{-j}$ for all $j \geq 0$. Here, the parameter $Y$ will be chosen later subject to $1\leq Y \leq x$. By partial integration, the Mellin transform of $h$ satisfies
\begin{equation*}
\hat{h}(s)=\int_{0}^{x+Y} h(z)z^{s-1}dz  \ll \frac{Y}{x^{1-\sigma}}\cdot \Big(\frac{x}{|s|Y}\Big)^j
\end{equation*}
for any $j\geq 1$ and $1/2\leq \sigma=\Re s \leq 2$. Moreover, we derive from \eqref{asymRS}, Lemma \ref{ineq-2rs} and the Cauchy--Schwarz inequality that
\begin{equation*}
\sum_{x<\mathrm{N}\mathfrak{n}  \leq x+Y }| \lambda_\pi(\mathfrak{n} )|\ll Y^{\frac12}\Big(\sum_{\mathrm{N}\mathfrak{n} \leq x+Y }| \lambda_\pi(\mathfrak{n} )|^2\Big)^{\frac12}\ll  (xY)^{\frac12}
\end{equation*}
for $0<Y\leq x$.  Consequently, we have that
\begin{equation*}\label{syupper-1}
\sum_{\mathrm{N}\mathfrak{n} \leq x } \lambda_\pi(\mathfrak{n})\chi(\mathfrak{n})=\sum_{\mathfrak{n}} \lambda_\pi(\mathfrak{n})\chi(\mathfrak{n}) h(\mathrm{N}\mathfrak{n})+O((xY)^{\frac{1}{2}}).
\end{equation*}
By Mellin's inverse transform, we can write
\begin{equation}\label{syupper-2}
\sum_{\mathfrak{n}} \lambda_\pi(\mathfrak{n})\chi(\mathfrak{n}) h(\mathrm{N}\mathfrak{n})=\frac{1}{2\pi i}\int_{(2)}\hat{h}(s)\Big(\sum_{\mathfrak{n}}\frac{\lambda_{\pi}(\mathfrak{n})\chi(\mathfrak{n})}{{\mathrm{N}\mathfrak{n}}^s}\Big)ds.
\end{equation}
We know from \eqref{DS-L} that the Dirichlet series appearing above equals $L(s, \pi\otimes \chi_1)$ except for some Euler factors, where $\chi_1 (\text{mod}\; \mathfrak{m}^\prime)$ is a primitive character induced by $\chi$ with $\mathfrak{m}^\prime|\mathfrak{m}$.
The convexity bound for $L(s, \pi\otimes \chi_1)$ gives
\begin{equation*}
L(\tfrac{1}{2}+it, \pi\otimes \chi_1)  \ll  ({\mathrm{N}\mathfrak{m}^\prime}(3+|t|)^{[F:\Q]})^{\frac{n}{4}+\varepsilon}.
\end{equation*}
Moving the vertical line of integration in \eqref{syupper-2} to $\Re s=1/2$,
we obtain by Cauchy's theorem and \eqref{DS-L} that
\begin{multline*}
\sum_{\mathfrak{n}} \lambda_\pi(\mathfrak{n})\chi(\mathfrak{n}) h(\mathrm{N}\mathfrak{n})\ll d_F(\mathfrak{m})^{2n} \int_{(1/2)}|\hat{h}(s)L(s, \pi\otimes \chi_1)|ds\\
\ll\int_{0}^{x/Y}\frac{x^{\frac{1}{2}}}{1+t}({\mathrm{N}\mathfrak{m}^\prime}(3+t)^{[F:\Q]})^{\frac{n}{4}}	(\mathrm{N}\mathfrak{m}(3+t))^\varepsilon dt \\
+\int_{x/Y}^{\infty}\frac{Y}{x^{\frac{1}{2}}}\cdot \Big(\frac{x}{tY}\Big)^{\frac{n[F:\Q]}{4}+2}({\mathrm{N}\mathfrak{m}^\prime} t^{[F:\Q]})^{\frac{n}{4}}	(\mathrm{N}\mathfrak{m}t)^\varepsilon dt\ll\mathrm{N}\mathfrak{m}^\varepsilon {\mathrm{N}\mathfrak{m}^\prime}^{\frac{n}{4}} \Big(\frac{ x}{Y}\Big)^{\frac{n[F:\Q]}{4}+\varepsilon}x^{\frac12}.
\end{multline*}
Gathering the above results we arrive at
\begin{equation*}
\sum_{\mathrm{N}\mathfrak{n} \leq x } \lambda_\pi(\mathfrak{n})\chi(\mathfrak{n}) \ll \mathrm{N}\mathfrak{m}^\varepsilon {\mathrm{N}\mathfrak{m}^\prime}^{\frac{n}{4}} \Big(\frac{ x}{Y}\Big)^{\frac{n[F:\Q]}{4}+\varepsilon}x^{\frac12}+ (xY)^{\frac12}.
\end{equation*}
We choose $Y=({\mathrm{N}\mathfrak{m}^\prime}x^d)^{\frac{n}{n[F:\Q]+2}}$, thus obtaining
\begin{equation*}
\sum_{n\leq x }\lambda_\pi(\mathfrak{n})\chi(\mathfrak{n})\ll \mathrm{N}\mathfrak{m}^{\frac{n}{2n[F:\Q]+4}+\varepsilon}x^{\frac{n[F:\Q]+1}{n[F:\Q]+2}+\varepsilon}.
\end{equation*}
Since the orthogonality of characters yields
\begin{equation*}
\sum_{\substack{\mathrm{N}\mathfrak{n}\leq x \\ \mathfrak{n}\equiv \mathfrak{a} \text{ in } \mathrm{Cl}^{+}(\mathfrak{m})}}\lambda_\pi(\mathfrak{n}) \ll \max_{\chi \in \widehat{\mathrm{Cl}^{+}(\mathfrak{m})} }\Big|\sum_{\mathrm{N}\mathfrak{n}\leq x}\lambda_\pi(\mathfrak{n}) \chi(\mathfrak{n})\Big|\leq x^{1-\varepsilon}
\end{equation*}
for any $\mathfrak{m}$ with $\mathrm{N}\mathfrak{m}\leq x^{\frac{2}{n}-\varepsilon}$,  the sequence $\{\lambda_\pi(\mathfrak{n})\}$ satisfies the Siegel--Walfisz hypothesis, and so does $\{\lambda_\pi(\mathfrak{n})\log \mathrm{N}\mathfrak{n} \}$ by partial summation. Thus, we can apply Lemma \ref{max-lem-Vaughan} to the sum on the last line of \eqref{s2-doube} and then get from the second estimate in Lemma \ref{ba2-sum}
\begin{equation*}
\begin{aligned}		
&\ll (\log xQ)^4\max_{H\leq M\leq X}  \Big(Q+\sqrt{M+\frac{x}{M}}+\frac{\sqrt{x}}{(\log x/M)^{A}}\Big) \sqrt{\sum_{\mathrm{N}\mathfrak{b}\sim M}|\mu_\pi(\mathfrak{b})|^2 \sum_{\mathrm{N}\mathfrak{c}\sim x/M}|\lambda_\pi(\mathfrak{c})\log\mathfrak{c}|^2}\\
&\ll (\log xQ)^5 \Big(Q\sqrt{x}+\sqrt{Xx}+\frac{x}{\sqrt{H}}+\frac{x}{(\log x/X)^{A}}\Big).
\end{aligned}
\end{equation*}
To sum up, the contribution from $S_2$ is
\begin{equation*}
\sum_{\mathrm{N}\mathfrak{m}\leq Q} \frac{h(\mathfrak{m})}{\phi_F(\mathfrak{m})}\max_{(\mathfrak{a},\mathfrak{m})=\cO_F}\max_{y\leq x} |S_2|
\ll \frac{Hx}{(\log x)^A} +(\log xQ)^{5} \Big(Q\sqrt{x}+\sqrt{Xx}+\frac{x}{\sqrt{H}}+\frac{x}{(\log x/X)^{A}}\Big).
\end{equation*}

\subsection{The average for $S_3$.}
The treatment is similar to that of $S_2$. If we use \eqref{square-alpha}, then
\begin{equation*}
\begin{aligned}
& \sum_{\mathrm{N}\mathfrak{m}\leq Q} \frac{h(\mathfrak{m})}{\phi_F(\mathfrak{m})}\max_{(\mathfrak{a},\mathfrak{m})=\cO_F}\max_{y\leq x} |S_3| \\
&\ll \frac{Hx}{(\log x)^{A-\frac{n}{2}-1}} +(\log xQ)^{\frac{n}{2}+5} \Big(Q\sqrt{x}+X\sqrt{x}+\frac{x}{\sqrt{H}}+\frac{x}{(\log x/X^2)^{A}}\Big).
\end{aligned}
\end{equation*}

\subsection{The average for $S_4$}
This contribution is of the same form as $S_2^{''}$ except the coefficients $\beta_{F,\pi}(\mathfrak{n})$ instead of $\mu_\pi(\mathfrak{n})$ and $\Lambda_F(\mathfrak{c})a_\pi(\mathfrak{c})$ instead of $\lambda_\pi(\mathfrak{n})\log\mathfrak{n}$. Since the Siegel--Walfisz condition of $\Lambda_F(\mathfrak{c})a_\pi(\mathfrak{c})$ follows from Corollary \ref{cor-SW-thm}, the estimates \eqref{square-lambda} and \eqref{square-beta} give that
\begin{equation*}
\sum_{\mathrm{N}\mathfrak{m}\leq Q} \frac{h(\mathfrak{m})}{\phi_F(\mathfrak{m})}\max_{(\mathfrak{a},\mathfrak{m})=\cO_F}\max_{y\leq x} |S_4|
\ll (\log xQ)^{2n+6} \Big(Q\sqrt{x}+\frac{x}{\sqrt{X}}+\frac{x}{(\log X)^{A}}\Big).
\end{equation*}
\subsection{Finishing the proof}
Finally, we collect all of our estimates for the averages of $S_1$, $S_2$, $S_3$, and $S_4$, taking $H=(\log x)^{\frac{2}{3}(A+4)} $ and $X=x^{\frac{1}{3}}$, to arrive at
\begin{equation*}
\begin{aligned}
\sum_{\mathrm{N}\mathfrak{m}\leq x^{\frac{1}{\eta}}(\log x)^{-B}} \frac{h(\mathfrak{m})}{\phi_F(\mathfrak{m})}\max_{(\mathfrak{a},\mathfrak{m})=\cO_F} \max_{y\leq x}\Big|\sum_{\substack{\mathrm{N}\mathfrak{n}\leq y \\ \mathfrak{n}\equiv\mathfrak{a} \text{ in } \mathrm{Cl}^{+}(\mathfrak{m})}}\Lambda_F(\mathfrak{n})a_\pi(\mathfrak{n})\Big|\ll \frac{x}{(\log x)^{\frac{A}{3}-2n-6} },
\end{aligned}
\end{equation*}
where $\eta=\max\{\frac{n}{2},2\}, B=2^{\frac{n[F:\Q]}{4}}(2A+16)+2n-4$. After changing parameter $\frac{A}{3}-2n-6\mapsto A$ and inserting the corresponding estimate into \eqref{primetopower}, Theorem \ref{thm 1.1} follows.
	
\section{An arithmetic application:  Proof of Corollary \ref{cor-titchmarsh}}
By the definition of divisor function, we have
$$
d(m)=2 \sum_{\substack{r \mid m \\ r<\sqrt{m}}} 1+\delta_{\square}(m),\qquad \delta_{\square}(m):=\left\{\begin{array}{ll}
1 & \text { if } m \text { is a perfect square,} \\
0 & \text { otherwise.}
\end{array}\right.
$$
Hence, we deduce that
\begin{equation}\label{shiftsum-prime}
\begin{aligned}
\sum_{p\leq x}\lambda_\pi(p)d(p-1)&= 2\sum_{p\leq x}\lambda_\pi(p)\sum_{\substack{r|(p-1) \\ r< \sqrt{p-1}}}1
+\sum_{p\leq x}\lambda_\pi(p)\delta_{\square}(p-1)\\
&=2\sum_{r<\sqrt{x-1}}\sum_{\substack{p\leq x \\ p\equiv 1\bmod r}}\lambda_\pi(p)+O(x^{\frac{3}{4}}).\\
\end{aligned}
\end{equation}
Here the error term comes from
\[
\sum_{p\leq x }\lambda_\pi(p)\delta_{\square}(p-1)\ll \Big(\sum_{p\leq x}|\lambda_\pi(p)|^2\Big)^{\frac{1}{2}} \Big(\sum_{p\leq x}\delta_{\square}(p-1)\Big)^{\frac{1}{2}}
\ll x^{\frac{3}{4}}
\]
by using the Cauchy--Schwarz inequality and the Rankin--Selberg theory.

For $r \leq x^\frac{1}{2}(\log x)^{-B}$, the result in Theorem \ref{thm 1.1} produces a small enough estimate. Then we have
\begin{equation}\label{d-small}
\begin{aligned}
\sum_{r\leq x^\frac{1}{2}(\log x)^{-B}}\sum_{\substack{p\leq x\\ p\equiv 1\bmod r}}\lambda_\pi(p)\ll \sum_{r\leq x^\frac{1}{2}(\log x)^{-B}}\Big|\sum_{\substack{p\leq x\\ p\equiv 1\bmod r}}\lambda_\pi(p)\Big|\ll  \frac{x}{(\log x)^A}.
\end{aligned}
\end{equation}

For the contribution of the terms $\sqrt{x}(\log x)^{-B} <r<\sqrt{x-1},$  the Brun-Titchmarsh inequality is used in general. However, it is not suitable for the automorphic context, since GRC remains open.  By using the orthogonality relation of additive characters, we get
\begin{equation*}
\sum_{\substack{p\leq x\\ p\equiv 1\bmod r}}\lambda_\pi(p)=\frac{1}{r}\sum_{q|r }\sideset{}{^*}\sum_{a\bmod q }e\Big(\frac{-a}{q}\Big)\sum_{p\leq x}\lambda_\pi(p)e\Big(\frac{ap}{q}\Big).
\end{equation*}
Thus,
\begin{equation}\label{d-large}
\begin{aligned}
&\sum_{\frac{\sqrt{x}}{(\log x)^{B}} <r< \sqrt{x-1}}\sum_{\substack{p\leq x\\ p\equiv 1\bmod r}}\lambda_\pi(p)\\
&=\sum_{l< \sqrt{x-1}}\frac{1}{l}\sum_{\frac{\sqrt{x}}{l(\log x)^B}< q< \frac{\sqrt{x-1}}{l}}\frac{1}{q}\; \; \sideset{}{^*}\sum_{a\bmod q }e\Big(\frac{-a}{q}\Big) \sum_{p\leq x}\lambda_\pi(p)e\Big(\frac{ap}{q}\Big):= T_1+T_2,
\end{aligned}
\end{equation}
where $T_1$ denotes the contribution of the terms $l\leq (\log x)^B$ and $T_2$ denotes the contribution of the other terms. First we treat the sum $T_1$. Applying the Cauchy--Schwarz inequality and the additive large sieve inequality (e.g. \cite[Theorem 7.11]{IK-2004}), we have
\begin{equation}\label{T1-estimate}
\begin{aligned}
T_1&\ll \sum_{l\leq  (\log x)^B}\frac{1}{l}\sum_{\frac{\sqrt{x}}{l(\log x)^B}<q< \frac{\sqrt{x+1}}{l}}\frac{1}{q} \; \; \sideset{}{^*}\sum_{a\bmod q }\Big|\sum_{p\leq x}\lambda_\pi(p)e\Big(\frac{ap}{q}\Big)\Big|\\
&\ll 	\sum_{l\leq  (\log x)^B}\frac{1}{l}\Big(\sum_{q< \frac{\sqrt{x+1}}{l}} \; \; \sideset{}{^*}\sum_{a\bmod q }\Big|\sum_{p\leq x}\lambda_\pi(p)e\Big(\frac{ap}{q}\Big)\Big|^2\Big)^{\frac{1}{2}}
\Big(\sum_{\frac{\sqrt{x}}{l(\log x)^B}<q< \frac{\sqrt{x+1}}{l}}\frac{1}{q}\Big)^{\frac{1}{2}}\\
&\ll x^{\frac{1}{2}}\Big(\sum_{p\leq x}|\lambda_\pi(p)|^2\Big)^{\frac{1}{2}}(\log\log x)^{\frac{3}{2}}\ll \frac{x(\log\log x)^{\frac{3}{2}}}{\sqrt{\log x}},
\end{aligned}
\end{equation}
where we use the fact that $|\lambda_\pi(p)|^2\ll a_{\pi\times\tilde \pi}(p)$ and the estimate \eqref{PNT-RS} in the last step.
Next, we treat the sum $T_2$. Interchanging the order of summation and applying the formula for the Ramanujan sum
$$ \sideset{}{^*}\sum_{a\bmod q }e\Big(\frac{(p-1)a}{q}\Big) =\sum_{d|(p-1,q)}d\mu\Big(\frac{q}{d}\Big),$$
we obtain from Theorem \ref{thm 1.1} that
\begin{multline}
\label{T2-estimate}	
T_2= \sum_{ (\log x)^B< l< \sqrt{x-1}}\frac{1}{l}\sum_{m< \frac{\sqrt{x-1}}{l}}\frac{\mu(m)}{m} \sum_{\frac{\sqrt{x}}{ml(\log x)^B}<q< \frac{\sqrt{x-1}}{ml}}\sum_{\substack{p\leq x\\ p\equiv 1\bmod q }}\lambda_\pi(p) \\
\ll (\log x)^2\sum_{q< \frac{\sqrt{x-1}}{(\log x)^B}}\Big|\sum_{\substack{p\leq x\\ p\equiv 1\bmod q }}\lambda_\pi(p)\Big|\ll  \frac{x}{(\log x)^A}.
\end{multline}
Assembling these estimates in \eqref{d-small}--\eqref{T2-estimate} yields
\[
\sum_{r<\sqrt{x-1}}\sum_{\substack{p\leq x\\ p\equiv 1\bmod r}}\lambda_\pi(p)\ll  \frac{x(\log\log x)^{\frac{3}{2}}}{\sqrt{\log x}}.
\]
Inserting this into \eqref{shiftsum-prime}, this corollary follows.

\vskip 5mm
\bibliographystyle{plain}

\bibliography{JIANG-Bibtex}

\end{document}